\documentclass[11pt,reqno]{amsart}

\usepackage{geometry}
\usepackage[T1]{fontenc}
\usepackage{newtxtext}   
\usepackage[hidelinks]{hyperref} 
\hypersetup{colorlinks=true,linkcolor=red,citecolor=blue}

\usepackage{breakcites}

\usepackage{algorithm}
\usepackage{algpseudocode}
\usepackage{graphicx}
\usepackage{dsfont}
\usepackage{setspace}
\usepackage{enumitem}
\usepackage[dvipsnames]{xcolor}
\usepackage{graphicx}
\usepackage{epstopdf}
\usepackage{csquotes}
\usepackage{mathtools}
\usepackage{amssymb, amsmath,amsthm, amsfonts}
\usepackage{ifthen}

\usepackage{caption}

\usepackage{verbatim}
\usepackage{pifont}
\usepackage{bm}  
\usepackage{stackengine}
\usepackage{bbm}

\allowdisplaybreaks
\usepackage{accents}
\usepackage[capitalise]{cleveref}
\usepackage{xpatch}
\makeatletter
\xpatchcmd{\@thm}{\thm@headpunct{.}}{\thm@headpunct{}}{}{}

\makeatother

\makeatother

\newtheorem{theorem}{Theorem}
\newtheorem{lemma}{Lemma}
\newtheorem{remark}{Remark}
\newtheorem{corollary}{Corollary}
\newtheorem{proposition}{Proposition}

\usepackage{graphicx}
\usepackage{subcaption}
\usepackage{enumitem}
\usepackage{hyperref}
\usepackage{url}
\usepackage{wrapfig,framed}
\usepackage{setspace}
\usepackage{graphicx}
\usepackage{epstopdf}
\usepackage{csquotes}
\usepackage{mathtools}
\usepackage{amssymb, amsmath,amsthm, amsfonts}
\usepackage{ifthen}
\usepackage{tikz}
\usepackage{verbatim}
\usepackage{pifont}
\usepackage{bm}  
\usepackage{stackengine}
\usepackage{bbm}

\allowdisplaybreaks
\usepackage{accents}

\usepackage{xpatch}
\makeatletter
\xpatchcmd{\@thm}{\thm@headpunct{.}}{\thm@headpunct{}}{}{}
\makeatother

\newcommand{\bsigma}{\bm{\Sigma}}
\newcommand{\bgamma}{\bm{\Gamma}}
\newcommand{\blambda}{\bm{\Lambda}}

\DeclareMathOperator{\supp}{spt}

\newcommand{\R}{\mathbb{R}}
\newcommand{\N}{\mathbb{N}}

\newcommand{\eps}{\varepsilon}
\newcommand\numberthis{\addtocounter{equation}{1}\tag{\theequation}}

\newcommand{\Wp}{\mathsf{W}_p}

\newcommand{\cA}{\mathcal{A}}

\newcommand{\cC}{\mathcal{C}}
\newcommand{\cD}{\mathcal{D}}

\newcommand{\cF}{\mathcal{F}}
\newcommand{\cG}{\mathcal{G}}

\newcommand{\cP}{\mathcal{P}}

\newcommand{\cR}{\mathcal{R}}
\newcommand{\cS}{\mathcal{S}}

\newcommand{\cX}{\mathcal{X}}
\newcommand{\cY}{\mathcal{Y}}

\newcommand{\sd}{\mathsf{d}}

\newcommand{\sF}{\mathsf{F}}
\newcommand{\sG}{\mathsf{G}}

\newcommand{\sh}{\mathsf{h}}

\newcommand{\sW}{\mathsf{W}}

\newcommand{\EE}{\mathbb{E}}

\newcommand{\NN}{\mathbb{N}}

\newcommand{\PP}{\mathbb{P}}

\newcommand{\RR}{\mathbb{R}}

\newcommand{\ZZ}{\mathbb{Z}}

\newcommand{\op}{\mathrm{op}}
\newcommand{\F}{\mathrm{F}}
\newcommand{\lip}{\mathrm{Lip}}

\newcommand{\bP}{\mathbf{P}}
\newcommand{\bQ}{\mathbf{Q}}
\newcommand{\W}{\mathsf{W}}

\newcommand{\bA}{\mathbf{A}}

\newcommand{\dkl}{\mathsf{D}_{\mathsf{KL}}}

\newcommand*{\dd}{\, \mathrm{d}}

\newcommand{\GW}{\mathsf{D}}
\newcommand{\EGW}{\mathsf{S}}

\definecolor{darkblue}{rgb}{0.0,0.0,0.66}  
\definecolor{darkred}{rgb}{100,0.0,0.0} 
\hypersetup{colorlinks=true,linkcolor=darkred,citecolor=darkblue}


\begin{document}

\title[Gromov-Wasserstein Distances: Entropic Regularization, Duality, and Sample Complexity]{Gromov-Wasserstein Distances:\\Entropic Regularization, Duality, and Sample Complexity}

\thanks{Z. Goldfeld is partially supported by the NSF CAREER award under grant CCF-2046018 and NSF grant DMS-2210368. B. K. Sriperumbudur is partially supported by the NSF CAREER award under grant DMS-1945396 and NSF grant DMS-19453.}

\author[Z. Zhang]{Zhengxin Zhang}
\address[Z. Zhang]{Center for Applied Mathematics, Cornell University.}
\email{zz658@cornell.edu}

\author[Z. Goldfeld]{Ziv Goldfeld}
\address[Z. Goldfeld]{School of Electrical and Computer Engineering, Cornell University.}
\email{goldfeld@cornell.edu}

\author[Y. Mroueh]{Youssef Mroueh}
\address[Y. Mroueh]{IBM Research AI.}
\email{mroueh@us.ibm.com}

\author[B. K. Sriperumbudur]{Bharath K. Sriperumbudur}

\address[B. K. Sriperumbudur]{Department of Statistics, Pennsylvania State University.}
\email{bks18@psu.edu}

\begin{abstract}
The Gromov-Wasserstein (GW) distance, rooted in optimal transport (OT) theory, quantifies dissimilarity between metric measure spaces and provides a framework for aligning heterogeneous datasets. While computational aspects of the GW problem have been widely studied, a duality theory and fundamental statistical questions concerning empirical convergence rates remained obscure. This work closes these gaps for the quadratic GW distance over Euclidean spaces of different dimensions $d_x$ and $d_y$. We treat both the standard and the entropically regularized GW distance, and derive dual forms that represent them in terms of the well-understood OT and entropic OT (EOT) problems, respectively. This enables employing proof techniques from statistical OT based on regularity analysis of dual potentials and empirical process theory, using which we establish the first GW empirical convergence rates.
The derived two-sample rates are $n^{-2/\max\{\min\{d_x,d_y\},4\}}$ (up to a log factor when $\min\{d_x,d_y\}=4$) for standard GW and $n^{-1/2}$ for EGW, which matches the corresponding rates for standard and entropic OT.  {The parametric rate for EGW is evidently optimal, while for standard GW we provide matching lower bounds, which establish sharpness of the derived rates.} We also study stability of EGW in the entropic regularization parameter and prove approximation and continuity results for the cost and optimal couplings. Lastly, the duality is leveraged to shed new light on the open problem of the one-dimensional GW distance between uniform distributions on $n$ points, illuminating why the identity and anti-identity permutations may not be optimal. Our results serve as a first step towards a comprehensive statistical theory as well as computational advancements for GW distances, based on the discovered dual formulations.
\end{abstract}


\maketitle

\section{Introduction}

The Gromov-Wasserstein (GW) distance, proposed by M\'emoli in \cite{Memoli11}, quantifies discrepancy between probability distributions supported on different metric spaces by aligning them with one another. Given two metric measure (mm) spaces $(\cX,\sd_\cX,\mu)$ and $(\cY,\sd_\cY,\nu)$, the $(p,q)$-GW distance between them is \cite{sturm2012space} 
\begin{equation}
 \GW_{p,q}(\mu,\nu)\coloneqq  \inf_{\pi\in\Pi(\mu,\nu)}\left( \int_{\cX\times\cY}\int_{\cX\times\cY}\big| \sd_\cX(x,x')^q-\sd_\cY(y,y')^q\big|^p d\pi\otimes\pi(x,y,x',y')\right)^{\frac{1}{p}},\label{eq:GW_intro}
\end{equation}
where $\Pi(\mu,\nu)$ is the set of all couplings between $\mu$ and $\nu$. The GW distances thus equals the least amount of distance distortion one can achieve between the mm spaces when optimizing over all possible alignments thereof (as modeled by couplings). This approach, which is rooted in optimal transport (OT) theory \cite{villani2009optimal,santambrogio15}, is an $L^p$ relaxation of the Gromov-Hausdorff distance between metric spaces and enjoys various favorable properties. Among others, the GW distance (i)~identifies pairs of mm spaces between which there exists an measure preserving isometry; (ii) defines a metric on the space of all mm spaces modulo the aforementioned isomorphic relation; and (iii) captures empirical convergence of mm space, i.e., when $\mu,\nu$ are replaced with their empirical measures $\hat{\mu}_n,\hat{\nu}_n$ based on $n$ samples. As such, the GW framework has been utilized for many applications concerning  heterogeneous data, including single-cell genomics \cite{blumberg2020mrec,demetci2020gromov}, alignment of language models \cite{alvarez2018gromov}, shape and graph matching \cite{memoli2009spectral,xu2019gromov,xu2019scalable,koehl2023computing}, heterogeneous domain adaptation \cite{yan2018semi}, and generative modeling~\cite{bunne2019learning}.

While such applications predominantly run on sampled data, a statistical GW theory to guarantee valid estimation and inference has remained elusive. This gap can be attributed, in part, to the quadratic (in $\pi$) structure of the GW functional, which prevents directly using well-developed proof techniques from statistical OT. Indeed, the linear OT problem enjoys strong duality, which enables analyzing empirical OT distances via techniques from empirical process theory, such as chaining, entropy integral bounds, and the functional delta method. These approaches have proven central to the development of statistical OT, leading to a comprehensive account of empirical convergence rates \cite{dudley1969speed,chizat2020faster,manole2021sharp,hundrieser2022empirical} and limit distributions of both classical \cite{Sommerfeld2018,Tameling2019,del2019central,manole2021plugin,hundrieser2022,goldfeld2022statistical} and regularized OT distances \cite{mena2019statistical,bigot2019central,klatt2020empirical,goldfeld2020gaussian,goldfeld2022limitsmooth,goldfeld2022statistical,del2022improved,goldfeld2022limit};  cf. Remarks \ref{rem:duality_OT} and \ref{rem:duality_EOT} ahead for a detailed discussion about the utility of duality for the statistical analysis of standard and regularized OT, respectively. 
For the GW distance, on the other hand, while we know that $\GW_{p,q}(\hat{\mu}_n,\hat{\nu}_n)\to \GW_{p,q}(\mu,\nu)$ as $n\to\infty$ \cite{Memoli11},\footnote{\cite{Memoli11} established this convergence for compact mm spaces and $q=1$, but the argument readily extends to any $q\geq 1$ and arbitrary mm space, so long that $\mu,\nu$ have bounded $pq$-th moments.} the rate at which this convergence happens is an open problem of theoretical and practical importance. 
This work closes this gap by deriving a dual formulation for the 
(standard and entropic) $(2,2)$-GW distance over Euclidean spaces, and leveraging it to establish the first empirical convergence rates for the GW~problem.

\subsection{Contribution} For probability distributions $\mu$ and $\nu$ supported in $\RR^{d_x}$ and $\RR^{d_y}$, respectively, we study both the standard $(2,2)$-GW distances from \eqref{eq:GW_intro} and its entropically regularized version~\cite{solomon2016entropic}
\[
\EGW_\eps(\mu,\nu)\coloneqq \inf_{\pi\in\Pi(\mu,\nu)}\iint\big| \|x-x'\|^2-\|y-y'\|^2\big|^2  d \pi\otimes \pi(x,y,x',y')+ \eps \dkl(\pi \| \mu\otimes \nu),
\]
where $\dkl(\cdot\|\cdot)$ is the Kullback-Leibler (KL) divergence. 
 {The interest in entropic GW (EGW) stems from its computational tractability \cite{solomon2016entropic,peyre2016gromov,scetbon2022linear,rioux2023entropic}, which makes it a popular approach in practice.} Our first main contribution is a duality theory for GW and EGW, which linearizes these quadratic functionals and ties them, respectively, to the well understood problems of OT and EOT. 
This is done by introducing an auxiliary, matrix-valued optimization variable $\mathbf{A}\in\RR^{d_x\times d_y}$ that enables linearizing the dependence on the coupling. We then interchange the optimization over $\mathbf{A}$ and $\pi$ and identify the inner problem as classical or entropic OT (EOT) with respect to (w.r.t.) a cost function $c_\mathbf{A}$ that depends on $\mathbf{A}$. 
Upon verifying that $c_\mathbf{A}$ satisfies mild regularity conditions, we invoke OT or EOT duality to arrive at a dual formulation for $\GW_{2,2}(\mu,\nu)^2$ and $\EGW_\eps(\mu,\nu)$. The dual form involves optimization over $\mathbf{A}$, which we show can be restricted to a hypercube whose side length depends only on the second moments of $\mu,\nu$. 

The GW and EGW dual forms enable an analysis of expected empirical convergence rates by drawing upon proof techniques from statistical OT. Namely, we consider the rates at which $\EE\big[\big|\GW_{2,2}(\mu,\nu)^2-\GW_{2,2}(\hat{\mu}_n,\hat{\nu}_n)^2\big|\big]$ and $\EE\big[\big|\EGW_{\eps}(\mu,\nu)-\EGW_{\eps}(\hat{\mu}_n,\hat{\nu}_n)\big|\big]$ decay to zero with $n$, as well as the one-sample case where $\nu$ is not estimated. Invoking strong duality we bound the empirical estimation error by the suprema of empirical processes indexed by OT or EOT dual potentials w.r.t. the cost $c_\mathbf{A}$, supremized over all feasible matrices $\mathbf{A}$. We then study the regularity of optimal potentials, uniformly in $\mathbf{A}$, which is the main technical difference from the corresponding OT and EOT analyses. For EGW, we show that the potentials are H\"older smooth to an arbitrary order and provide bounds on the growth rate of their H\"older norm. Combining the regularity theory with a chaining argument and entropy integral bounds for H\"older classes, we arrive at $n^{-1/2}$ as the empirical convergence rate for EGW. This parametric rate holds in any ambient dimensions $d_x,d_y$ and is inline with EOT empirical rates \cite{genevay2019sample,mena2019statistical}. For the unregularized GW problem, we focus on compactly supported distributions and exploit smoothness and marginal-concavity of the cost $c_\mathbf{A}$ to show that optimal potentials are concave and Lipschitz. Following a similar  analysis to the entropic case while leveraging the so-called lower complexity adaptation (LCA) principle form \cite{hundrieser2022empirical} leads to an $n^{-2/\max\{\min\{d_x,d_y\},4\}}$ upper bound on the two-sample rate of the quadratic GW distance (up to a log factor when $\min\{d_x,d_y\}=4$; in the one-sample case, $d_y$ is omitted).  {We then establish matching lower bounds on the one- and two-sample empirical estimation errors, demonstrating that the said rates are sharp. The lower bound proof is constructive and utilizes a novel inequality between the quadratic GW distance and the 2-Wasserstein procrustes \cite{grave2019unsupervised}, which may be of independent interest.}

We also address basic structural properties of the GW and EGW distances. First, we study stability of the entropic variant in the regularization parameter $\eps$ and establish an $O\big(\eps\log(1/\eps)\big)$ bound on the gap between (squared) GW and EGW. This bound matches the entropic approximation error in the standard OT case \cite{genevay2019sample}. However, unlike the result from \cite{genevay2019sample}, that accounts only for compactly supported distributions, our derivation relies on maximum entropy inequalities and holds for arbitrary distributions. After treating the entropic approximation of the GW cost, we prove that optimal entropic couplings weakly converge towards an optimal GW coupling as $\eps\to0$ by leveraging the notion of $\Gamma$-convergence. Lastly, we revisit the open problem of the one-dimensional GW distance between uniform distributions on $n$ points and use our duality theory to shed new light on it. We consider the peculiar example from \cite{beinert2022assignment}, where, contrary to common belief (cf. \cite{vayer2020sliced}), the identity and anti-identity permutations were shown to not necessarily be optimal. Our dual form allows representing the GW distance on $\RR$ as a sum of concave and convex functions, explaining why the optimum need not be attained at the boundary. We verify and visualize the different regimes of optimal solutions via simple numerical simulations.

\subsection{Literature Review}\label{subsec:lit_review}

The GW distance was first proposed in \cite{Memoli11} as an $L^p$ relaxation of the Gromov-Hausdorff distance between metric spaces. Basic structural properties of the distance were also established in that work, with more advanced aspects concerning topology and curvature addressed in \cite{sturm2012space}. The existence of Gromov-Monge maps was studied in \cite{dumont2022existence}, showing that optimal couplings are induced by a bimap (viz. two-way map) under quite general conditions. Targeting analytic solutions, optimal couplings between Gaussian distributions were explored in \cite{delon2022gromov}, but only upper and lower bounds on the GW value were derived. An exact characterization of the optimal coupling and cost is known for the entropic inner product GW distance between Gaussians \cite{le2022entropic}.

As GW distances grew in popularity for applications, computational tractability became increasingly important. However, exact computation of the GW distance is generally a quadratic assignment problem, which is NP-complete \cite{Commander2005}. For this reason, significant attention was devoted to variants of the GW problem that circumvent this computational hardness. The sliced GW distance \cite{vayer2020sliced} attempts to reduce the computational burden by considering the average of GW distances between one-dimensional projections of the marginals. However, unlike one-dimensional OT, the GW problem does not have a known simple solution even in one dimension \cite{beinert2022assignment}. Another approach is to relax the strict marginal constraints to obtain the unbalanced GW distance \cite{sejourne2021unbalanced}, which lends well for convex/conic relaxations. A variant that directly optimizes over bi-directional Monge maps between the mm space was considered in \cite{zhang2022cycle}. While these methods offer certain advantages, it is the approach based on entropic regularization that is most frequently used in practice. This is since EGW is computable via iterative optimization routines that employ Sinkhorn iterations \cite{solomon2016entropic,peyre2016gromov,scetbon2022linear,rioux2023entropic}, which allows scalability and parallelization in large-scale applications.

\medskip

 {
\noindent\textbf{Follow-up work.} We address two follow-up works that appeared on arXiv several months after the original submission of this work (and its upload to arXiv on March 2). The paper \cite{rioux2023entropic}, by one of the authors of the current work and other collaborators, leveraged the duality theory proposed herein to study algorithms, limit theorems, and resampling methods for the EGW distance. Their approach relied on a stability analysis of the dual formulation in $\mathbf{A}$, based on which $L$-smoothness and sufficient conditions for convexity of the objective function were derived. These, in turn, were used to propose the first algorithms for computing EGW in $O(n^2)$ time (where $n$ is the number of support points of the two marginals) subject to formal convergence guarantees in both the convex and non-convex regimes. That work also considered stability of the dual in the marginals $\mu,\nu$, which led to a limit distribution theory for the empirical EGW distance and, under additional conditions, asymptotic normality, bootstrap consistency, and semiparametric efficiency. Our results along with those form \cite{rioux2023entropic} now provide the statistical and computational foundations for valid estimation and inference for the EGW distance, with efficient implementations via the aforementioned algorithms.  

Another notable follow-up work is \cite{groppe2023lower}, which appeared online two months after our paper was posted to arXiv and submitted to the journal. That work studied the LCA principle from \cite{hundrieser2022empirical} under the EOT setting. In particular, they observed that the dependence on dimension in our empirical convergence rate bounds can be relaxed from $\max\{d_x,d_y\}$ to $\min\{d_x,d_y\}$, provided that the populations are compactly supported. For EGW, as the rate is parametric and dimension-free, this observation only serves to improve the constant. Furthermore, our EGW bounds hold for distributions with unbounded supported, which are beyond the scope of \cite{groppe2023lower}. For the standard GW distance, our original submission proved an $n^{-2/\max\{d_x,d_y,4\}}$ upper bound on the two-sample rate, but Remark~5.6 of \cite{groppe2023lower} observed that it can be improved to $n^{-2/\max\{\min\{d_x,d_y\},4\}}$ and provided high-level proof outline.\footnote{Our $n^{-2/\max\{d_x,4\}}$ one-sample rate bound for standard GW is unaffected by the results of \cite{groppe2023lower}.} Herein we provide a full proof of the two-sample upper bound with the dependence on the smaller dimension, and also establish new lower bounds that demonstrate the sharpness of the derived one- and two-sample empirical convergence rates. The reader is referred to Remarks \ref{rem:EGW_min_dimension} and \ref{rem:GW_min_dimension} for a detailed discussion and comparison to \cite{groppe2023lower}.}

\subsection{Organization}
The rest of this paper is organized as follows. In \cref{sec:background}, we collect background material on the OT, EOT, GW, and EGW problems. \cref{sec:EGW} treats the EGW distance, covering stability in the regularization parameter, duality, and sample complexity. In \cref{sec:GW}, we extend the duality and the statistical treatment to the (unregularized) GW distance itself. \cref{sec:proofs} contains proofs for Sections \ref{sec:EGW} and \ref{sec:GW}. \cref{sec:summary} leaves concluding remarks and discusses future directions. The Appendix contains proofs of technical results that are omitted from the main text.

\subsection{Notation}
Let $\| \cdot \|$ and $\langle \cdot, \cdot \rangle$ denote the Euclidean norm and inner product, respectively. Let $B_d(x,r) \coloneqq  \{ y \in \R^{d} : \|y-x\| \le r \}$ denote the closed ball with center $x$ and radius $r$. We use $\|\cdot\|_\op$ and $\|\cdot\|_\F$ for the operator and Frobenius norms of matrices, respectively. For a topological space $S$, $\cP(S)$ denotes the class of Borel probability measures on it. For $p\in[1, \infty)$, let $\cP_p(\RR^d)$ be the space of Borel probability measures with finite $p$-th absolute moment, i.e., $M_p(\rho)\coloneqq \int_{\RR^d}\|x\|^p d \rho(x)<\infty$ for any $\rho\in\cP_p(\RR^d)$. For a signed Borel measure $\rho$ and a measurable function $f$, we use the shorthand $\rho f := \int f d\rho$, whenever the integral exists. The support of  $\rho\in\cP(\RR^d)$ is $\supp(\rho)$, while its covariance matrix (when exists) is denoted by $\Sigma_{\rho}$. For a sequence of probability measure $(\rho_n)_{n\in\NN}$ that weakly converges to $\rho$, we write $\rho_n\xrightarrow[]{w}\rho$. A probability distribution $\rho\in\cP(\RR^d)$ is called $\beta$-sub-Weibull with parameter $\sigma^2$ for $\sigma\geq0$ if $\int \exp\left(\|x\|^\beta/2\sigma^2\right) d \rho(x)\leq 2$. In particular, $\rho$ is sub-Gaussian if it is 2-sub-Weibull. Notice that $X\sim \rho$ being 4-sub-Weibull is equivalent to $\|X\|^2$ being sub-Gaussian, in which case $\int e^{t\|x\|^2}d\rho(x)\leq 2e^{t^2\sigma^2/2}$. The latter bound is repeatedly used in our derivations. 

Let $C_b(\RR^d)$ be the space of bounded continuous functions on $\RR^d$ equipped with the~$L^\infty$ norm. The Lipschitz seminorm of a function $f:\RR^d\to\RR$ is $\|f\|_{\mathrm{Lip}}\coloneqq \sup_{x\neq x'}\frac{|f(x)-f(x')|}{\|x-x'\|}$. For $p \in [1,\infty)$ and $\rho\in\cP(\RR^d)$, let $L^p(\rho)$ be the space of measurable functions $f$ of $\R^d$ such that $\|f\|_{L^p(\rho)}\coloneqq  (\int_{\R^d}|f|^p  d  \rho)^{1/p} < \infty$. For any multi-index $k=(k_1,\dots,k_d) \in \mathbb{N}_0^d$ with $|k| = \sum_{j=1}^ d k_j$ ($\NN_0 = \NN \cup \{ 0 \}$), define the differential operator $D^k = \frac{\partial^{|k|}}{\partial x_1^{k_1} \cdots \partial x_{d}^{k_d}}$
with $D^0 f = f$. We write $N(\eps, \cF, \mathsf{d})$ for the $\eps$-covering number of a function class $\cF$ w.r.t. a metric $\mathsf{d}$, and $N_{[\ ]}(\eps, \cF, \mathsf{d})$ for the bracketing number. We use $\lesssim_x$ to denote inequalities up to constants that only depend on $x$; the subscript is dropped when the constant is universal. For $a,b \in \R$, let $a \vee b = \max \{ a,b \}$ and $a \land b = \min \{ a,b \}$.

\section{Background and Preliminaries}\label{sec:background}

\subsection{Classical and Entropic Optimal Transport}\label{sec:OT_EOT}

We briefly review basic definitions and results concerning the classical and entropic OT problems, which serve as building blocks for our subsequent analysis of the GW distance. For a detailed exposition the reader is referred to \cite{villani2009optimal,santambrogio15,peyre2019computational}. Let $\cX,\cY$ be two Polish spaces and consider a lower semicontinuous cost function $c:\cX\times\cY\to\RR$, where note that we allow $c$ to take negative value.

\paragraph{\textbf{Classical optimal transport.}}

The OT problem between $(\mu,\nu)\mspace{-1.5mu}\in\mspace{-1.5mu}\cP(\cX)\mspace{-1.5mu}\times\mspace{-1.5mu}\cP(\cY)$ with cost $c$~is
\begin{equation}
    \mathsf{OT}_c(\mu,\nu) \coloneqq  \inf_{\pi\in\Pi(\mu,\nu)} \int_{\cX\times\cY} c\, d \pi,\label{eq:OT}
\end{equation}
where $\Pi(\mu,\nu)$ is the set of all couplings of $\mu$ and $\nu$, i.e., each $\pi\in\Pi(\mu,\nu)$ is a probability distribution on $\cX\times\cY$ that has $\mu$ and $\nu$ as its first and second marginals, respectively. The special case of the $p$-Wasserstein distance, for $p \in [1,\infty)$, is given by $\mathsf{W}_p(\mu,\nu)\coloneqq \left(\mathsf{OT}_{\|\cdot\|^p}(\mu,\nu)\right)^{1/p}$. $\mathsf{W}_p$ is a metric on $\cP_p(\R^d)$ which metrizes weak convergence plus convergence of $p$-th moments, i.e., $\Wp(\hat{\mu}_n,\mu) \to 0$ if and only if $\hat{\mu}_n \stackrel{w}{\to} \mu$ and $M_p(\hat{\mu}_n) \to M_p(\mu)$.

OT is a linear program and as such it admits strong duality. Suppose that the cost function satisfies $c(x,y)\geq a(x)+b(y)$, for all $(x,y)\in\cX\times\cY$, for some upper semicontinuous functions $(a,b)\in L^1(\mu)\times L^1(\nu)$. Then (cf. \cite[Theorem 5.10]{villani2009optimal}): 
\begin{equation}
    \mathsf{OT}_c(\mu,\nu) = \sup_{(\varphi,\psi)\in\Phi_c} \int_\cX \varphi d \mu + \int_\cY \psi d \nu,\label{eq:OT_dual}
\end{equation}
where $\Phi_c\coloneqq \big\{(\varphi,\psi)\in C_b(\cX)\times C_b(\cY): \varphi(x)+\psi(y)\leq c(x,y),\,\forall (x,y)\in\cX\times\cY\big\}$. Furthermore, defining the $c$- and $\bar{c}$-transform of $\varphi\in C_b(\cX)$ and $\psi\in C_b(\cY)$ as $\varphi^c(y)\coloneqq \inf_{x\in\cX} c(x,y)-\varphi(x)$ and $\psi^{\bar{c}}(x)\coloneqq \inf_{y\in\cY} c(x,y)-\psi(y)$, respectively, the optimization above can be restricted to pairs $(\varphi,\psi)$ such that $\psi=\varphi^c$ and $\varphi=\psi^{\bar{c}}$.

\begin{remark}[Duality for statistical OT]\label{rem:duality_OT}
The dual form in \eqref{eq:OT_dual} is key for the statistical analysis of OT, encompassing empirical convergence rates \cite{dudley1969speed,chizat2020faster,manole2021sharp,hundrieser2022empirical} and limit distribution theorems \cite{Sommerfeld2018,Tameling2019,del2019central,manole2021plugin,hundrieser2022,goldfeld2022statistical}. For instance, if optimal dual potentials in  \eqref{eq:OT_dual} lie, respectively, in functional classes $\cF$ and $\cG$, one can bound the two-sample error as
\[
\EE\big[\big|\mathsf{OT}_c(\mu,\nu)-\mathsf{OT}_c(\hat{\mu}_n,\hat{\nu}_n)\big|\big]\lesssim \EE\bigg[\sup_{\varphi\in\cF_c}(\mu-\hat{\mu}_n)\varphi\bigg]+\EE\bigg[\sup_{\psi\in\cG_c}(\nu-\hat{\nu}_n)\psi\bigg].
\]
This reduces the error analysis to that of the expected suprema of two empirical processes indexed by the classes $\cF_c$ and $\cG_c$. One may then use techniques from empirical process theory \cite{van1996weak} to obtain the desired convergence rate. This requires studying regularity of optimal dual potentials to obtain bounds on the covering numbers of the corresponding function classes. Given bounds of the form $N(\eps,\cF_c,L^\infty)\vee N(\eps,\cG_c,L^\infty)\lesssim \eps^{-k}$, a convergence rate of $n^{-1/k}$ immediately follows by standard chaining arguments and  entropy integral bounds.

Dudley used this argument in \cite{dudley1969speed} to prove an $n^{-1/d}$ rate for one-sample estimation of the 1-Wasserstein distance, relying on the fact that the covering number of the Lipschitz class scales like $\eps^{-1/d}$. More recently, \cite{manole2021sharp} refined this argument to establish a two-sample rate of $n^{-2/d}$, under smoothness and convexity assumptions on the cost, by observing that dual potentials are not only Lipschitz but also convex in that case. This yielded a covering number bound of $\eps^{-d/2}$ and the rate follows. The primal form was used in \cite{dereich2013constructive,fournier2015rate} to derive sharp empirical convergence rates for the $p$-Wasserstein distance via a block approximation argument for the optimal coupling. More recently, however, duality enabled an even finer statistical analysis for deriving limit distribution theorems of empirical OT, e.g., via linearization arguments \cite{del2019central,manole2021plugin} or by proving weak convergence of the underlying empirical processes and invoking the functional delta method \cite{shapiro1990,romisch2004}, as done in \cite{Sommerfeld2018,Tameling2019,hundrieser2022,goldfeld2022statistical}. 
\end{remark}

\paragraph{\textbf{Entropic optimal transport.}}
EOT is a convexification of the classical OT problem by means of an entropic penalty. For a regularization parameter $\eps>0$, EOT is given by
\begin{equation}
\mathsf{OT}_{c,\eps}(\mu,\nu) \coloneqq  \inf_{\pi\in\Pi(\mu,\nu)} \int c\,  d \pi + \eps \dkl(\pi\|\mu\otimes\nu),\label{eq:EOT}
\end{equation}
where the KL divergence is given by $\dkl(\alpha\|\beta)\coloneqq \int\log ( d  \alpha/ d  \beta) \,  d \alpha$ if $\alpha \ll \beta$, and equals $+\infty$ otherwise. 
The optimization objective in \eqref{eq:EOT} is strongly convex in $\pi$ and thus admits a unique solution $\pi^\star$. The entropic cost $\mathsf{OT}_{c,\eps}$ and the optimal solutions $\pi^\star$ are known to converge towards the classical OT cost \cite{genevay2019sample} and a corresponding optimal plan \cite{carlier2017convergence} as $\eps\to 0$.\footnote{For the plan, convergence happens in the weak topology and possibly along a subsequence.} In particular, Theorem 1 from \cite{genevay2019sample} shows that for smooth costs and compact spaces $\cX,\cY$, the entropic approximation gap is 
$\big|\mathsf{OT}_\eps(\mu,\nu)-\mathsf{OT}(\mu,\nu)|\lesssim \eps\log (1/\eps)$.

EOT satisfies duality and can be rewritten as (cf. \cite{nutz2021entropic}):
\begin{equation}
    \mathsf{OT}_{c,\eps}(\mu,\nu) = \sup_{(\varphi,\psi) \in L^1(\mu)\times  L^1(\nu)} \int \varphi d \mu + \int \psi d \nu - \eps \int e^{\frac{\varphi(x)+\psi(y)-c(x,y)}{\eps}}  d \mu\otimes\nu(x,y)+ \eps.\label{eq:EOT_dual}
\end{equation}
There exist functions $(\varphi,\psi) \in L^1 (\mu) \times L^1(\nu)$ that achieve the supremum in \eqref{eq:EOT_dual}, which we call \textit{EOT potentials}. EOT potentials are almost surely (a.s.) unique up to additive constants in the sense that if $(\tilde{\varphi},\tilde{\psi})$ is another pair of EOT potentials, then there exists a constant $a \in \R$ such that $\tilde{\varphi}= \varphi +a\ $ $\mu$-a.s. and $\tilde{\psi} = \psi-a\ $ $\nu$-a.s. Furthermore, a pair of functions $(\varphi,\psi)\in L^1 (\mu) \times L^1(\nu)$ are EOT potentials if and only if they satisfy the so-called \textit{Schr\"{o}dinger system}
\begin{equation}
\int_{\cX} e^{\frac{\varphi(x)+\psi(\cdot)-c(x,\cdot)}{\eps}}  d \mu(x)= 1 \quad \nu\mbox{-}a.s.\qquad \mbox{and}\qquad \int_{\cY} e^{\frac{\varphi(\cdot)+\psi(y)-c(\cdot,y)}{\eps}}  d \nu(y)= 1 \quad \mu\mbox{-}a.s.\label{EQ:Schrodinger}
\end{equation}
Given EOT potentials $(\varphi,\psi)$, the unique EOT plan can be expressed in their terms as $ d \pi^\star(x,y) = e^{\frac{\varphi(x)+ \psi(y) - c(x,y)}{\eps}}  d  \mu \otimes \nu(x,y)$.

\begin{remark}[Duality for statistical EOT]\label{rem:duality_EOT}
Akin to the utility of duality for the statistical analysis of OT, as discussed in \cref{rem:duality_OT}, the EOT dual served a pivotal role in the development of a statistical theory under entropic regularization. Thanks for the Schr\"{o}dinger system in \eqref{EQ:Schrodinger}, smoothness of the cost $c$ implies the existence of EOT potentials that reside in a H\"older space of arbitrarily large smoothness; cf., e.g., \cite[Lemma 1]{goldfeld2022limit}. This, in turn, enables establishing parametric $n^{-1/2}$ convergence rates for EOT \cite{genevay2019sample,mena2019statistical,rigollet2022sample,groppe2023lower} under quite general conditions and a rich limit distribution theory for the EOT cost, plan, dual potentials, and the barycentric projection \cite{mena2019statistical,bigot2019central,klatt2020empirical,gunsilius2021matching,goldfeld2022limit,delbarrio22EOT,gonzalez2022weak,gonzalez2023weak}. 
\end{remark}

\subsection{Classical and Entropic Gromov-Wasserstein Distance}\label{subsec:GW}

The objects of interest in this work are the GW distance and its entropic version. The $(p,q)$-GW distance quantifies similarity between (complete and separable) mm spaces $(\cX,\sd_\cX,\mu)$ and $(\cY,\sd_\cY,\nu)$ as \cite{Memoli11,sturm2012space}.
\[
    \GW_{p,q}(\mu,\nu)\coloneqq  \inf_{\pi\in\Pi(\mu,\nu)} \big\|\Delta_{q}^{\cX,\cY}\big\|_{L^p(\pi\otimes\pi)},
\]
where $\Delta_q^{\cX,\cY}(x,y,x',y')=\big| \sd_\cX(x,x')^q-\sd_\cY(y,y')^q\big|$. This definition is an $L^p$  relaxation of the Gromov-Hausdorff distance between metric spaces,\footnote{The Gromov-Hausdorff distance between $(\cX,\sd_\cX)$ and $(\cY,\sd_\cY)$ is given by $\frac 12 \inf_{R\in\cR(\cX,\cY)}\|\Delta^{\cX,\cY}_{1,1}\|_{L^\infty(R)}$, where $\cR(\cX,\cY)$ is the collection of all correspondence sets of $\cX$ and $\cY$, i.e., subsets $R\subset\cX\times\cY$ such that the coordinate projection maps
are surjective when restricted to $R$. The correspondence set can be thought of as $\supp(\pi)$ in the GW formulation.} and gives rise to a metric on the collection of all isomorphism classes of mm spaces\footnote{The mm spaces $(\cX,\sd_\cX,\mu)$ and $(\cY,\sd_\cY ,\nu)$ are isomorphic if there is an isometry $f:\cX\to\cY$ with $f_\sharp \mu=\nu$.}  with finite $pq$-size, i.e., $\int \sd_\cX(x,x')^{pq} d \mu\otimes \mu(x,x')<\infty$ and similarly for $\nu$. Like the $p$-Wasserstein distance, Theorem 5.1 in \cite{Memoli11} reveals that $\GW_{p,q}$ captures empirical convergence of mm spaces: if $X_1,\ldots,X_n$ are samples from $\mu\in\cP(\cX)$ and $\hat{\mu}_n\coloneqq n^{-1}\sum_{i=1}^n \delta_{X_i}$ is their empirical measures, then $\GW_{p,q}(\hat{\mu}_n,\mu)\to 0$ a.s. The rate at which this empirical convergence happens is, however, an open problem.

Towards a complete resolution,
one of our main contributions is to quantify the empirical convergence rate of the $(2,2)$-GW distance between Euclidean mm spaces $(\RR^{d_x},\|\mspace{-1mu}\cdot\mspace{-1mu}\|,\mu)$~and $(\RR^{d_y},\|\mspace{-1mu}\cdot\mspace{-1mu}\|,\nu)$ of different dimensions. Abbreviating $\Delta^{\RR^{d_x},\RR^{d_y}}_{2}=\Delta$, the distance of interest~is
\begin{align*}
    \GW(\mu,\nu)& \coloneqq \inf_{\pi\in\Pi(\mu,\nu)} \|\Delta\|_{L^2(\pi\otimes\pi)}\\
    &=\inf_{\pi\in\Pi(\mu,\nu)}\left(\int_{\RR^{d_x}\times\RR^{d_y}}\int_{\RR^{d_x}\times\RR^{d_y}} \big| \|x-x'\|^2-\|y-y'\|^2\big|^2  d \pi\otimes \pi(x,y,x',y')\right)^{\frac 12}.\numberthis\label{eq:GW}
\end{align*}
We drop subscripts from our notation because we focus on the $(2,2)$-GW case from here on out. For finiteness we will always assume $\mu\in\cP_4(\RR^{d_x})$ and $\nu\in\cP_4(\RR^{d_y})$. We also treat the GW distance with entropic regularization, which, for $\eps>0$, is defined as
\begin{align*}
    \EGW_\eps(\mu,\nu)\coloneqq \inf_{\pi\in\Pi(\mu,\nu)} \|\Delta\|^2_{L^2(\pi\otimes\pi)}+ \eps \dkl(\pi \| \mu\otimes \nu).
\end{align*}
 The motivation for EGW stems from its computational tractability \cite{solomon2016entropic,peyre2016gromov,scetbon2022linear,rioux2023entropic}, which makes it a popular approach in practice.\footnote{As discussed in \cref{subsec:lit_review}, the follow-up work \cite{rioux2023entropic} proposed the first algorithms for computing EGW between discrete distributions on $n$ points to arbitrary precision in $O(n^2)$ time, subject to formal convergence guarantees. These algorithms hinge upon the dual formulation developed herein. This progress, along with the statistical theory we provide, poses EGW as a viable tool for statistical estimation and inference.} With the setup above, we have $\EGW_0(\mu,\nu)=\GW(\mu,\nu)^2$ but approximation bounds that account for the gap $\big|\EGW_\eps(\mu,\nu)-\GW(\mu,\nu)^2\big|$ are currently unknown (nor is there a proof of weak convergence for the corresponding optimal couplings). Another major gap in GW and EGW theory is the lack of dual formulations, without which an empirical convergence rate analysis of standard and entropic GW distances remained obscure. In what follows, we close these gaps.

\section{Entropic Gromov-Wasserstein Distance}\label{sec:EGW}

\subsection{Continuity in Regularization Parameter} We study continuity of the EGW cost and optimal coupling in $\eps$. Our first result quantifies the gap between the GW and EGW costs. 

\begin{proposition}[Cost approximation gap]\label{prop:cont_of_cost_in_eps} For any $\eps\in(0,1]$ and $(\mu,\nu)\in\cP_4(\RR^{d_x})\times\cP_4(\RR^{d_y})$,~we~have 
\begin{align*}
    \big|\EGW_\eps(\mu,\nu)-\GW(\mu,\nu)^2\big|\lesssim_{d_x,d_y,M_4(\mu),M_4(\nu)} \eps\log\frac{1}{\eps}.
\end{align*}
\end{proposition}

The proof of \cref{prop:cont_of_cost_in_eps}, which is given in \cref{proof:prop:cont_of_cost_in_eps}, relies on a block approximation of optimal GW couplings and the Gaussian maximum entropy inequality. Specifically, we decompose the space into cubes with side length $\ell$ and construct a new coupling $\pi^\ell$ that is piecewise uniform on these cubes  {(the block approximation idea for the EOT coupling originally dates back to \cite{carlier2017convergence})}. The error of the entropic approximation is then quantified in terms of $\ell$, with the KL divergence term being bounded using the differential entropy of the Gaussian distribution with a matched covariance matrix. We then optimize the bound over $\ell$ to arrive at the desired dependence~on~$\eps$.

\begin{remark}[Comparison to EOT approximation results]
A similar bound of order $O\big(\eps\log(1/\eps)\big)$ was derived in Theorem 1 of \cite{genevay2019sample} for the entropic approximation gap of~the OT problem on compact domains with Lipschitz cost. Our proof of \cref{prop:cont_of_cost_in_eps}, which relies on a block approximation of optimal GW couplings, is inspired by their derivation but with several key differences. Specifically, by leveraging the Gaussian maximum entropy inequality, we allow for arbitrary distributions with bounded 4th moments (which is always required for finiteness of $\GW$) and costs that grow at most polynomially. Our proof technique can directly be used to relax the assumptions of \cite[Theorem 1]{genevay2019sample} to match those of \cref{prop:cont_of_cost_in_eps}. Another related entropic approximation result appeared in Theorem 1 of \cite{chizat2020faster}, providing an $O(\eps^2)$ bound on the gap between the squared 2-Wasserstein distance and the Sinkhorn divergence (which is a centered version of EOT). Their derivation utilizes a dynamical formulation of the Sinkhorn divergence \cite{chen2016relation,gentil2017analogy,gigli2020benamou,conforti2021formula}, which allows tying it to the Benamou-Brenier formula for $\mathsf{W}_2^2$~\cite{benamou2000}. No dynamical form for the GW distance is currently~known.
\end{remark}

\cref{prop:cont_of_cost_in_eps} guarantees the convergence of the EGW cost towards that of GW, as
$\eps\to0$. It is natural to ask whether the optimal couplings that achieve these costs converge as well?
We answer this question to the affirmative.

\begin{proposition}[Convergence of plans]\label{prop:convergence-of-plans}
Fix $(\mu,\nu)\in \cP_4(\RR^{d_x})\times\cP_4(\RR^{d_y})$ and let $(\eps_k)_{k \in \N}$ be a sequence with $\eps_k\searrow \eps \geq 0$. For each $k \in \N$, let $\pi_k \in \Pi(\mu,\nu)$ be an optimal coupling for $\EGW_{\eps_k}(\mu,\nu)$. Then there exists $\pi \in \Pi(\mu,\nu)$ such that $\pi_k \xrightarrow[]{w} \pi$ as $k \to \infty$ along a subsequence, and $\pi$ is optimal for $\EGW_\eps(\mu,\nu)$. 
\end{proposition}

The proof of \cref{prop:convergence-of-plans}, which is given in \cref{proof:prop:convergence-of-plans}, relies on establishing $\Gamma$-convergence of the EGW functional with $\eps$. Having that, convergence of optimal couplings follows by a tightness argument. In particular, this result implies that a sequence of optimal couplings for $\EGW_\eps(\mu,\nu)$ converges, up to extracting a subsequence, to an optimal coupling for the regular $(2,2)$-GW distance as $\eps\to 0$.

\subsection{Duality}

We next derive a dual formulation for the EGW distance. This duality serves as the key component for our sample complexity analysis of empirical EGW in the next subsection. Towards the dual form, first observe that $\EGW_\eps$ is invariant to isometric operations on the marginal spaces, such as translation and orthonormal rotation. Thus, without loss of generality (w.l.o.g.), we assume that $\mu$ and $\nu$ are centered, i.e., $\int x d  \mu(x)=\int y d  \nu(y)=0$. 

\medskip
Next, by expanding the $(2,2)$-GW cost, we split the EGW functional into two terms as 
\begin{equation}
\EGW_\eps(\mu,\nu) = \EGW^1(\mu,\nu) + \EGW^2_\eps(\mu,\nu),\label{eq:egw_decomposition}
\end{equation}
where
\begin{align*}
    &\EGW^1\mspace{-1mu}(\mu,\nu)\mspace{-3mu}\coloneqq \mspace{-3mu} \int\mspace{-3mu}\|x\mspace{-3mu}-\mspace{-3mu}x'\|^4 d \mu\otimes\mu(x,x') \mspace{-3mu}+ \mspace{-3mu}\int\mspace{-3mu} \|y\mspace{-3mu}-\mspace{-3mu}y'\|^4 d \nu\otimes\nu(y,y') \mspace{-3mu}-4\mspace{-5mu}\int \mspace{-3mu}\|x\|^2\|y\|^2 d \mu\otimes\nu(x,y)\\
    &\EGW^2_{\eps}\mspace{-1mu}(\mu,\nu) \mspace{-3mu}\coloneqq\mspace{-3mu}  \inf_{\pi\in\Pi(\mu,\nu)}\mspace{-3mu}\int\mspace{-3mu}- 4\|x\|^2\|y\|^2  d \pi(x,y) \mspace{-3mu}-\mspace{-3mu}8\sum_{\substack{1\leq i\leq d_1\\1\leq j \leq d_2}}\mspace{-5mu}\Big(\mspace{-3mu}\int\mspace{-3mu}  x_iy_j d \pi(x,y)\Big)^2 \mspace{-3mu}+ \eps \dkl(\pi \|\mu\otimes\nu).
\end{align*}
See \eqref{eq:proof_of_EGW_decomposition} for the derivation. Evidently, the first term depends only on the marginals $\mu,\nu$, while the second captures the dependence on the coupling $\pi$. The following theorem establishes duality for $\mathsf{S}^2_\eps(\mu,\nu)$, which, in turn, yields a dual form for $\EGW_\eps(\mu,\nu)$ via the above decomposition.

\begin{theorem}[Entropic GW duality]\label{thm:egw_duality}
Fix $\varepsilon>0$, let $(\mu,\nu)\in\cP_4(\RR^{d_x})\times\cP_4(\RR^{d_y})$, and define $M_{\mu,\nu}:=\sqrt{M_2(\mu)M_2(\nu)}$. We have
    \begin{equation}
    \EGW^2_{\eps}(\mu,\nu)=\inf_{\mathbf{A}\in\RR^{d_x\times d_y}}32\|\mathbf{A}\|_\F^2+\mathsf{OT}_{\mathbf{A},\eps}(\mu,\nu),\label{eq:EGW_dual}
     \end{equation}
    where $\mathsf{OT}_{\mathbf{A},\eps}$ is the EOT problem with cost function $c_{\mathbf{A}}\mspace{-3mu}:\mspace{-3mu}(x,y)\in\RR^{d_x}\times \RR^{d_y}\mapsto\mspace{-3mu}-4\|x\|^2\|y\|^2-32x^{\intercal}\mathbf{A}y$. Moreover, the infimum is achieved at some $\mathbf{A}^{\star}\mspace{-3mu}\in\mspace{-3mu}\cD_{M_{\mu,\nu}}\mspace{-3mu}\coloneqq\mspace{-3mu}[-M_{\mu,\mspace{-3mu}\nu}/2,M_{\mu,\nu}/2]^{d_x\times d_y}$.
\end{theorem}

 {The variational representation above relates the EGW to the well understood problem of EOT. This enables leveraging knowledge on the latter to make progress in the study of EGW. In particular, this representation unlocks the sample complexity analysis in the next subsection, which relies on inserting the EOT dual from \eqref{eq:EOT_dual} into the above. Since \eqref{eq:EGW_dual} allows utilizing EOT duality for the EGW analysis, we synonymously refer to it as the EGW dual (even though it is somewhat of a misnomer, since strictly speaking, \eqref{eq:EGW_dual} is not a dual problem for $\EGW_{\eps}(\mu,\nu)$ in the standard optimization theory sense).}

The proof of \cref{thm:egw_duality} is given in \cref{proof:thm:egw_duality}. The key idea in deriving the above representation is to introduce the additional dual variable $\mathbf{A}$ as a means to linearize the quadratic (in fact, concave) in $\pi$ term of $\EGW^2_{\eps}(\mu,\nu)$. The resulting objective comprises two infima, over $\mathbf{A}\in\cD_{M_{\mu,\nu}}$ and $\pi\in\Pi(\mu,\nu)$, which we may interchange. Upon doing so, we identify the inner optimization as the primal EOT problem (up to a minus sign) with the cost $c_{\mathbf{A}}$. Existence of an optimal $\mathbf{A}$ follow from continuity of the functional and compactness of the optimization domain. As the optimum is always achieved inside $\cF_{M_{\mu,\nu}}$, we may restrict the optimization domain to $\mathbf{A}\in\cD_M$, for any $M\geq M_{\mu,\nu}$, without changing the value. 
The flexibility of choosing $M$ an optimizing over the compact set $\cD_M$ is crucial for our sample complexity analysis.

\subsection{Sample Complexity}

The dual formulation from \cref{thm:egw_duality} enables deriving, for the first time, the sample complexity of empirical EGW distances. Let $X_1,\ldots,X_n$ and $Y_1,\ldots,Y_n$ be independently and identically distributed (i.i.d.) samples from $\mu$ and $\nu$, respectively, and denote their empirical measures by $\hat{\mu}_n = n^{-1}\sum_{i=1}^n \delta_{X_i}$ and $\hat{\nu}_n = n^{-1}\sum_{i=1}^n \delta_{Y_i}$. We study one- and two-sample empirical convergence, i.e., the rate at which $\EGW_\eps(\hat{\mu}_n,\nu)$ and $\EGW_\eps(\hat{\mu}_n,\hat{\nu}_n)$ approach  $\EGW_\eps(\mu,\nu)$, under a sub-Weibull condition on the population distributions. 

\begin{theorem}[Entropic GW sample complexity]\label{thm:egw_sample_complex}
Fix $\eps>0$ and let $(\mu,\nu)\in\cP(\RR^{d_x})\times\cP(\RR^{d_y})$ be a pair of 4-sub-Weibull distributions with parameter $\sigma^2>0$. We have
\begin{align*}
    \EE\big[\big|\EGW_\eps(\mu,\nu)-\EGW_\eps(\hat{\mu}_n,\nu)\big|\big]&\lesssim_{d_x,d_y} \frac{1+\sigma^4}{\sqrt{n}}+\eps\Bigg(1+\bigg(\frac{\sigma}{\sqrt{\eps}}\bigg)^{9\left\lceil \frac{d_x}{2}\right\rceil+11}\Bigg)\frac{1}{\sqrt{n}}\\    
    \EE\big[\big|\EGW_\eps(\mu,\nu)-\EGW_\eps(\hat{\mu}_n,\hat{\nu}_n)\big|\big]&\lesssim_{d_x,d_y}\frac{1+\sigma^4}{\sqrt{n}}+\eps\Bigg(1+\bigg(\frac{\sigma}{\sqrt{\eps}}\bigg)^{9\left\lceil \frac{d_x\vee d_y}{2}\right\rceil+11}\Bigg)\frac{1}{\sqrt{n}}.
\end{align*}
\end{theorem}

\cref{thm:egw_sample_complex} is derived in \cref{proof:thm:egw_sample_complex}. Here, we provide a proof outline, and~explain how the duality from \cref{thm:egw_duality} facilitates the derivation. The proof follows three main steps:
\begin{enumerate}[leftmargin=*]
    \item \underline{Decomposition:} We first split the empirical estimation error of $\EGW_\eps$ to that of its components $\EGW^1$ and $\EGW^2_\eps$. Notice that the decomposition is not straightforward since $\EGW_\eps=\EGW^1+\EGW^2_\eps$ holds only for centered measures, and while we may assume this w.l.o.g. on the populations $(\mu,\nu)$, centering need not hold for the empirical measures. Thus, to perform the split we first center $(\hat{\mu}_n,\hat{\nu}_n)$ by their sample means, and then further account for the bias induced by this centering step, which is shown to be at most $\sigma^2/\sqrt{n}$. Altogether, we obtain 
    \[
\EE\big[\big|\EGW_\eps(\mu,\nu)\mspace{-3mu}-\mspace{-3mu}\EGW_\eps(\hat{\mu}_n,\hat{\nu}_n)\big|\big]\leq \EE\big[\big|\EGW^1(\mu,\nu)\mspace{-3mu}-\mspace{-3mu}\EGW^1(\hat{\mu}_n,\hat{\nu}_n)\big|\big] \mspace{-2mu}+\mspace{-3mu} \EE\big[\big|\EGW^2_{\eps} (\mu,\nu)\mspace{-3mu}-\mspace{-3mu}\EGW^2_{\eps} (\hat{\mu}_n,\hat{\nu}_n)\big|\big]\mspace{-3mu}+\mspace{-3mu}\frac{\sigma^2}{\sqrt{n}}, 
    \]
    and may analyze each component separately.
    
    \vspace{2mm}
    \item \underline{$\EGW^1$ analysis:} The first term on the right-hand side (RHS) above is simple to analyze, as estimation of $\EGW^1$ boils down to estimating moments of $(\mu,\nu)$. Since the sub-Weibull condition implies finite moments, we establish an $O(1/\sqrt{n})$ bound on the $\EGW^1$ estimation error. 

    \vspace{2mm}
    \item \underline{$\EGW^2_\eps$ analysis:} The treatment of the $\EGW^2_{\eps}$ is more involved and hinges on the dual representation from \cref{thm:egw_duality}. Specifically, using our dual with any $M\geq M_{\mu,\nu}$, we obtain
    \[
    \big|\EGW^2_{\eps}(\mu,\nu)-\EGW^2_{\eps}(\hat{\mu}_n,\hat{\nu}_n)\big| \leq \sup_{\mathbf{A}\in \cD_{M}} \big|\mathsf{OT}_{\mathbf{A},\eps}(\mu,\nu)-\mathsf{OT}_{\mathbf{A},\eps}(\hat{\mu}_n,\hat{\nu}_n)\big|,
    \]
    where the RHS can be controlled by the suprema of empirical processes indexed by optimal entropic potentials. As the potentials depend on the cost $c_\mathbf{A}$, we analyze regularity of optimal $(\varphi,\psi)$ pairs by bounding these functions and their partial derivative of arbitrary order, uniformly in $\mathbf{A}\in\cD_M$. Given the derivative bounds, a chaining argument and entropy integral bounds yield the second term on the RHS above as a bound on the empirical convergence rate for $\EGW^2_{\eps} $. The overall rate we obtain is parametric, and hence optimal, although the dependence of the bound on $\sigma$ and $\eps$ could possibly be improved.
\end{enumerate}

\begin{remark}[Dependence on dimension]\label{rem:EGW_min_dimension}
The empirical convergence rate of EGW given in \cref{thm:egw_sample_complex} is parametric, and hence cannot be improved. The dependence of the constant in the two-sample bound on the maximal dimension, however, can be relaxed. The follow-up work \cite{groppe2023lower}, which was posted on arXiv several months after our original submission and arXiv upload, observed that the dependence on dimension can be improved from $d_x\vee d_y$ to $d_x\wedge d_y$, for compactly supported populations. That work studied the LCA principle from \cite{hundrieser2022empirical} in the context of EOT. Relying on our duality theory, Theorem 5.4 of \cite{groppe2023lower} showed that, when $\mu,\nu$ are compactly supported, an empirical convergence rate with $d_x\wedge d_y$ instead of $d_x\vee d_y$ in the constant holds true.\footnote{More precisely, \cite{groppe2023lower} improves the exponent of the $\eps^{-1/2}$ term in the two-sample rate bound from \cref{thm:egw_sample_complex} to $d_x\wedge d_y$, but their overall bound still contains an implicit constant that depends on the maximal dimension $d_x\vee d_y$.} This result does not cover the full scope of \cref{thm:egw_sample_complex}, which treats unboundedly supported distributions with 4-sub-Weibull tails. The LCA principle, in its current form, does not seem to be compatible with unbounded supports, since it inherently relies on covering the class of EOT potentials in the $L^\infty$ norm. 
\end{remark}

\begin{remark}[Comparison to EOT]\label{rem:EGW_EOT_comparison}
The EGW empirical convergence rates from \cref{thm:egw_sample_complex} are similar to the corresponding rates for the EOT problem, which are also parametric. Specifically, the $n^{-1/2}$ rate was established in \cite{genevay2019sample} for  EOT between compactly supported distributions and assuming that the cost is $\cC^\infty$ and Lipschitz, although their bound contained an undesirable exponential dependence on $1/\eps$. This result was extended to sub-Gaussian distributions and quadratic cost in \cite{mena2019statistical}, while shaving off the said exponential factor and arriving to a bound that is similar to ours. More recently, \cite{groppe2023lower} observed that the LCA principle holds for EOT, showing that the constant in front of the $n^{-1/2}$ term adapts to the smaller intrinsic dimension of the two measures. 

Our approach for proving \cref{thm:egw_sample_complex} is inspired by \cite{mena2019statistical}, but requires overcoming several new challenges. First, a strong duality theory, which is at the core of the proof technique, was not available until now for the EGW distance. Second, our analysis goes through the decomposition \eqref{eq:egw_decomposition}, which needs the distributions to be centered. While we may assume this w.l.o.g. on $\mu,\nu$, the empirical measures are generally non-centered, which necessitates a bias analysis of the EGW functional due to centering, as discussed above. Lastly, as the dual form in \eqref{eq:EGW_dual} involves optimization over $\mathbf{A}\in\cD_M$, with $M\geq M_{\mu,\nu}$, our regularity analysis of EGW potentials must hold uniformly in $\mathbf{A}$, so as to allow the reduction to empirical processes. 
\end{remark}

\section{Gromov-Wasserstein Distance}\label{sec:GW}

\subsection{Duality and Sample Complexity}

We now consider the unregularized $(2,2)$-GW distance from~\eqref{eq:GW}, establish duality, derive its sample complexity, and study its one-dimensional structure. Let $(\mu,\nu)\in\cP_4(\RR^{d_x})\times\cP_4(\RR^{d_y})$ be centered w.l.o.g. and note that, similarly to the EGW case, the $(2,2)$-GW distance decomposes as
\[
\GW(\mu,\nu)^2=\EGW^1(\mu,\nu)+\EGW^{2}(\mu,\nu),
\]
where $\EGW^2\coloneqq \EGW^{2}_0$, with $\EGW^1$ and $\EGW^2_{\eps} $ as given after \eqref{eq:egw_decomposition}. To obtain a dual form for $\EGW^2$, an inspection of the proof of \cref{thm:egw_duality} reveals that the same argument holds also for $\eps = 0$ (i.e., any $\eps\geq 0$ is allowed in that statement), up to replacing the EOT problem $\mathsf{OT}_{\mathbf{A},\eps}$ in \eqref{eq:EGW_dual} with the standard (unregularized) OT problem $\mathsf{OT}_{\mathbf{A}}\coloneqq \mathsf{OT}_{\mathbf{A},0}$. Recalling the definitions of $M_{\mu,\nu}$, $\cD_{M_{\mu,\nu}}$, and $c_\mathbf{A}$ from \cref{thm:egw_duality}, we have the following corollary.

 {
\begin{corollary}[GW duality]\label{cor:gw_duality}
For any $(\mu,\nu)\in\cP_4(\RR^{d_x})\times\cP_4(\RR^{d_y})$, we have 
    \begin{equation}    \EGW^2(\mu,\nu)=\inf_{\mathbf{A}\in\RR^{d_x\times d_y}}32\|\mathbf{A}\|_\F^2+\mathsf{OT}_{\mathbf{A}}(\mu,\nu),\label{eq:GW_dual}
     \end{equation}
    where $\mathsf{OT}_{\mspace{-1mu}\mathbf{A}}$ is the OT problem with cost $c_{\mathbf{A}}$ and the infimum is achieved at some $\mathbf{A}^{\mspace{-2mu}\star}\mspace{-1mu}\in\mspace{-1mu}\cD_{M_{\mu,\nu}}$.
\end{corollary}}

Given this dual form for $\GW(\mu,\nu)^2$ we proceed with a sample complexity analysis. We focus on compactly supported distributions and refer the reader to \cref{rem:ubdd_domains} ahead for a discussion on extensions to unbounded domains.  {The following theorem gives a sharp characterization of the one- and two-sample empirical convergence rate of the quadratic GW distance, providing matching upper and lower rate bounds.

\begin{theorem}[GW sample complexity]\label{thm:GW_sample_complex}
Let $(\mu,\nu)\in\cP(\cX)\times\cP(\cY)$, where $\cX\subset \RR^{d_x}$ and $\cY\subset\RR^{d_y}$ are compact, and let $R=\mathrm{diam}(\cX)\vee\mathrm{diam}(\cY)$. We have
\begin{align*}
\EE\big[\big|\GW(\mu,\nu)^2-\GW(\hat{\mu}_n,\nu)^2\big|\big]&\lesssim_{d_x,d_y} \frac{R^4}{\sqrt{n}}+\big(1+R^4\big)n^{-\frac{2}{d_x\vee 4}}(
\log n)^{\mathds{1}_{\{d_x=4\}}}\\    
\EE\big[\big|\GW(\mu,\nu)^2-\GW(\hat{\mu}_n,\hat{\nu}_n)^2\big|\big]&\lesssim_{d_x,d_y}\frac{R^4}{\sqrt{n}}+\big(1+R^4\big)n^{-\frac{2}{(d_x\wedge d_y) \vee 4}}(
\log n)^{\mathds{1}_{\{d_x\wedge d_y=4\}}},    
\end{align*}
and if $\mu,\nu$ are separated in the $(2,2)$-GW distance, i.e., $\GW(\mu,\nu)>0$, then the same rates hold for estimating $\GW$ itself, without the square.

Furthermore, the above rates are sharp in the sense that for any $n$ large enough, we have
\begin{align*}
\sup_{(\mu,\nu)\in\cP(\cX)\times\cP(\cY)}\EE\big[\big|\GW(\mu,\nu)^2-\GW(\hat{\mu}_n,\nu)^2\big|\big]&\gtrsim_{\,d_x,d_y,R}\, n^{-\frac{2}{d_x\vee 4}}\\    
\sup_{(\mu,\nu)\in\cP(\cX)\times\cP(\cY)}\EE\big[\big|\GW(\mu,\nu)^2-\GW(\hat{\mu}_n,\hat{\nu}_n)^2\big|\big]&\gtrsim_{\,d_x,d_y,R}\, n^{-\frac{2}{(d_x\wedge d_y)\vee 4}}. 
\end{align*}
\end{theorem}

\begin{remark}[Chronology of results]\label{rem:GW_min_dimension}
The originally submitted version of this work included only upper bounds on the one- and two-sample empirical convergence rates of $\GW$, where the dependence on dimension was through the maximum $d_x\vee d_y$, as opposed to the minimum as above. The follow-up work \cite{groppe2023lower}, which appeared online two months after our paper was uploaded to arXiv and submitted to the journal, studied the LCA principle in the context of the EOT problem. Remark 5.6 of that work, observed that the LCA principle applies to our original \cref{thm:GW_sample_complex} and commented that the dependence on dimension can be improved the $d_x\wedge d_y$. A full proof of that claim was not provided in \cite{groppe2023lower}, only an high-level outline of the argument. Herein, in \cref{sec:GW_sample_complex_proof}, we provide a full derivation of the upper bounds with the dependence on the smaller dimension. In addition, we establish a novel lower bound that demonstrates that these empirical convergence rates are sharp.
\end{remark}

\cref{thm:GW_sample_complex} is proven in \cref{sec:GW_sample_complex_proof}. The upper bounds leverage the duality from \cref{cor:gw_duality} to reduce the empirical estimation analysis of $\GW^2$ to that of the OT problem with cost $c_\mathbf{A}$. The OT estimation error is then bounded by the suprema of empirical processes indexed by dual OT potentials. To control the corresponding entropy integrals, we exploit smoothness of our cost as well as Lipschitzness and convexity of optimal potentials as $c$-transforms of each other. The fact that the two-sample convergence rate adapts to the smaller dimension is a consequence of the LCA principle \cite[Lemma 2.1]{hundrieser2022empirical}, whereby the $L^\infty$ covering number of a function class $\cF$ is no less than that of its $c$-transform $\cF^c$. This observation enables adapting the bound to the class of dual potentials over the lower-dimensional space. Still, when the estimated measure(s) are high-dimensional, both the one- and two-sample rates for the GW distance suffer from the curse of dimensionality. This is expected in the absence of entropic regularization and is in line with empirical convergence rates for OT; see \cref{rem:ubdd_domains} ahead for further discussion on the comparison between the empirical rates for GW and OT. 

To prove the lower bound, we present a reduction from GW distance estimation to that of the 2-Wasserstein procrustes $\inf_{\mathbf{U}\in E(d)} \mathsf{W}_2(\mu,\mathbf{U}_\sharp\nu)$, where $E(d)$ is the isometry group on $\RR^d$ \cite{grave2019unsupervised} (see also \cite{schonemann1966generalized,goodall1991procrustes}). This relies on the following lemma, which may be of independent interest. We state two-sided bounds, but only the lower bound is used in the derivation.

\begin{lemma}[GW vs. W-procrustes]\label{lemma:equivalence_gw_w}
For any $p,q\in[1,\infty)$ and $\mu,\nu\in\cP_{pq}(\RR^d)$, we have
\[
    \GW_{p,q}(\mu,\nu) \leq q^p\, 2^{pq+p-1+1/q}  \big(M_{pq}(\mu)+M_{pq}(\nu) \big)^{\frac{q-1}{pq}} \mathsf{W}_{pq}(\mu,\nu).
\] 
Furthermore, for $p=q=2$, if $\mu$ and $\nu$ have covariance matrices $\bsigma_\mu$ and $\bsigma_\nu$ with full rank and smallest eigenvalues $\lambda_{\mathrm{min}}(\bsigma_\mu)$ and $\lambda_{\mathrm{min}}(\bsigma_\nu)$, respectively, then
\[
\Big(  32\big(\lambda_{\mathrm{min}}(\bsigma_\mu)^2+\lambda_{\mathrm{min}}(\bsigma_\nu)^2\big)\Big)^{\frac 14}\inf_{\mathbf{U}\in E(d)} \mathsf{W}_2(\mu,\mathbf{U}_\sharp\nu) \leq  \GW(\mu,\nu).
\]
If $\mu$ and $\nu$ are also centered, then it suffices to optimize only over the orthogonal group $O(d)$.
\end{lemma}

The lemma enables showing that the empirical GW rate, when the population measures are uniform over the unit ball and its scaled version, is at least as large as that of the Wasserstein procrustes. We then develop a new lower bound on the convergence rate of the latter, showing that it is at least $n^{-1/d}$. This, in turn, gives rise to the rates from \cref{thm:GW_sample_complex}. 
}

\begin{remark}[Suboptimal $(p,q)$-GW rates from \cref{lemma:equivalence_gw_w}]\label{rem:suboptimal_rates}
    Fix any $(p,q)$ and $\mu\in\cP_{pq}(\RR^d)$. The upper bound from \cref{lemma:equivalence_gw_w}, directly yields $\EE[\GW_{p,q}(\hat\mu_n,\mu)]\lesssim\EE[\W_{pq}(\hat\mu_n,\mu)]$ $\lesssim n^{-1/d_x}$. Via the triangle inequality we can further obtain an $n^{-1/(d_x\vee d_y)}$ two-sample rate bound for the $(p,q)$-GW distance. However, as seen from the lower bounds in \cref{thm:GW_sample_complex}, this rate is suboptimal and does adapt to the lower of the two dimensions.
\end{remark}

\begin{remark}[Comparison to OT and unbounded domains]\label{rem:ubdd_domains}
The rates in \cref{thm:GW_sample_complex} are inline with those for the classical OT problem with H\"older smooth costs \cite{manole2021sharp} (although our analysis is different from theirs). Over compact domains, this smoothness of the cost enables establishing global Lipschitzness and convexity of OT potentials, which, in turn, leads to the quadratic improvement from the standard $n^{-1/d}$ empirical convergence rate to $n^{-2/d}$, when $d>4$. Evidently, a similar phenomenon happens in the GW case. Unbounded domains are treated in Theorem 13 of \cite{manole2021sharp}, but this result relies on restrictive assumptions on the population distributions and the cost. Namely, the distributions must satisfy certain high-level concentration and anti-concentration conditions, while the cost must be locally H\"older smooth and be lower and upper bounded by a polynomial of appropriate degree. Our cost $c_\mathbf{A}$ does not immediately adhere to these assumptions. While we believe that the argument can be adapted, we leave this extension as a question for future work. 
\end{remark}

\subsection{One-Dimensional Case Study}\label{subsec:GW_1d}
We leverage our duality theory to shed new light on the one-dimensional GW distance. The solution to the GW problem between distributions on $\RR$ is currently unknown and remains one of the most basic open questions in that space. While the standard $p$-Wasserstein distance between distributions on $\RR$ is given by the $L^p([0,1])$ distance between their quantile functions,\footnote{For $p=1$, the formula further simplifies to the $L^1(\RR)$ distance between the cumulative distribution functions.} there is no known simple solution for the one-dimensional GW problem. Even for uniform distributions over $n$ distinct points, for which it was previously believed that the optimal GW coupling is always induced by the identity or anti-identity permutations \cite{vayer2020sliced}, it was recently shown that this is not true in general \cite{beinert2022assignment}. Indeed, \cite{beinert2022assignment} produced an example of discrete distributions, defined up to a tuning parameter $\xi$, for which the identity or anti-identity become suboptimal once $\xi$ surpasses a~certain~threshold. We revisit this example and attempt to better understand it using our dual~formulation. 

Consider two uniform distributions on $n$ distinct points, i.e., $\mu=n^{-1}\sum_{i=1}^n\delta_{x_i}$ and $\nu=n^{-1}\sum_{i=1}\delta_{y_i}$, where $(x_i)_{i=1}^n,(y_i)_{i=1}^n\subset\RR$ with $x_1<x_2<\ldots<x_n$ and $y_1<y_2<\ldots<y_n$. To compute $\GW(\mu,\nu)$ it suffices to optimize over couplings induced by permutations 
\cite[Theorem 9.2]{vayer2020sliced} (see also \cite{maron2018probably}), i.e., 
\begin{equation}
\GW(\mu,\nu)^2=\frac{1}{n^2}\min_{\sigma\in S_n}\sum_{i=1}^n\sum_{j=1}^n\big||x_i-x_j|^2-|y_{\sigma(i)}-y_{\sigma(j)}|^2\big|^2,\label{eq:GW_1d}
\end{equation}
where $S_n$ is the symmetric group over $n$ elements. 
For $\xi\in(0,2/(n-3))$ and $n>6$, define the point sets $x^\xi=(x^\xi_i)_{i=1}^n$ and $y^\xi=(y^\xi_i)_{i=1}^n$ as
\begin{equation}
x^\xi_i\coloneqq \begin{cases}
    -1,& i=1\\
    \frac{2i-n-1}{2}\xi,&2\leq i\leq n-1\\
    1,&i=n
\end{cases}
\qquad\mbox{and}\qquad
y^\xi_i\coloneqq \begin{cases}
    -1,& i=1\\
    -1+\xi,& i=2\\
    (i-2)\xi,&3\leq i\leq n
\end{cases}.\label{eq:GW_1d_datasets}
\end{equation}
Note that each of these sets indeed has ascending ordered, pairwise distinct components. 
The proof of Proposition 1 in \cite{beinert2022assignment} shows that there exists $\xi^\star\in(0,2/(n-3))$, such that the cyclic permutation $\sigma_\mathrm{cyc}(i)=i+1\mod n$ between $x^{\xi^\star}$ and $y^{\xi^\star}$ achieves a strictly smaller cost in \eqref{eq:GW_1d} than both the identity $\mathrm{id}(i)=i$ and the anti-identity $\overline{\mathrm{id}}(i)=n-i+1$ permutations.

To better understand the reason for the existence of strict optimizers outside the boundary, we recall 
that $\GW(\mu,\nu)^2=\EGW^1(\mu,\nu)+\EGW^{2}(\mu,\nu)$ and henceforth focus on $\EGW^{2}(\mu,\nu)$, which is the term that depends on the coupling. As mentioned before, this decomposition requires $\mu$ and $\nu$ to be centered, but we may assume this w.l.o.g. due the translation invariance of the GW-distance and of optimal~permutations. By \cref{cor:gw_duality} we have the following representation:
\[
    \EGW^2(\mu,\nu) = \inf_{\mathbf{A}\in\cD_M} 32\|\mathbf{A}\|_\F^2+  \inf_{\pi\in\Pi(\mu,\nu)}\int c_\mathbf{A}(x,y)  d \pi(x,y).
\]
Specializing to the one-dimensional case, we further obtain
\begin{equation}
    \EGW^2(\mu,\nu) = \inf_{a\in[0.5 W_-,0.5 W_+]} 32a^2+  \inf_{\pi\in\Pi(\mu,\nu)}\int \big(-4x^2y^2-32axy \big) d \pi(x,y),\label{eq:GW_1d_S2}
\end{equation}
where $W_-\coloneqq \inf_{\pi\in\Pi(\mu,\nu)}\int xy d \pi(x,y)$ and $W_+\coloneqq \sup_{\pi\in\Pi(\mu,\nu)}\int xy d \pi(x,y)$. Here, we have used the fact that, switching the infima order, for each $\pi\in\Pi(\mu,\nu)$, optimality is attained at $a^\star(\pi)=\frac{1}{2}\int xy d \pi(x,y)$. The notation $W_-$ and $W_+$ reflects the relation to the 2-Wasserstein distance: indeed, $2W_+=M_2(\mu)+M_2(\nu)-\sW_2^2(\mu,\nu)$, while $W_-$ is OT with product~cost.

Once we identify the optimal $a^\star$ in \eqref{eq:GW_1d_S2}, the GW problem is reduced to an OT problem. Hence, we investigate the optimization in $a$. Define $f(a)\coloneqq 32a^2$ and $g(a)\coloneqq \inf_{\pi\in\Pi(\mu,\nu)}\int \big(-4x^2y^2-32axy \big) d \pi(x,y)$, and note that $g$ is concave (as the infimum of affine functions). We see that the optimization over~$a$ in \eqref{eq:GW_1d_S2}, which is rewritten as $\inf_{a\in[0.5 W_-,0.5 W_+] }(f+g)(a)$, minimizes the sum of a convex and a concave function. The next proposition identifies a correspondence between the boundary values of $a$ and optimal permutations in \eqref{eq:GW_1d}; see \cref{sec:prop:1d_opt_proof} for the proof.

\begin{proposition}[Boundary values and optimal permutations]\label{prop:1d_opt}
    Consider the GW problem from \eqref{eq:GW_1d} between uniform distributions over $n$ distinct points and its representation as $\GW(\mu,\nu)^2=\EGW^1(\mu,\nu)+\EGW^2(\mu,\nu)$, where $\EGW^2(\mu,\nu)$ is given in \eqref{eq:GW_1d_S2}. Let $\cS^\star\subset S_n$ and $\cA^\star\subset[0.5 W_-,0.5 W_+]$ be the argmin sets for \eqref{eq:GW_1d} and \eqref{eq:GW_1d_S2}, respectively. Then $\cA^\star\subset\{0.5 W_-,0.5 W_+\}$ if and only if $\cS^\star\subset\{\mathrm{id},\overline{\mathrm{id}}\}$.
\end{proposition}

\begin{figure*}[!t]
\hspace{-10mm}
\centering
 \begin{subfigure}[t]{0.45\textwidth}
 \centering
\includegraphics[width=1.1\linewidth]{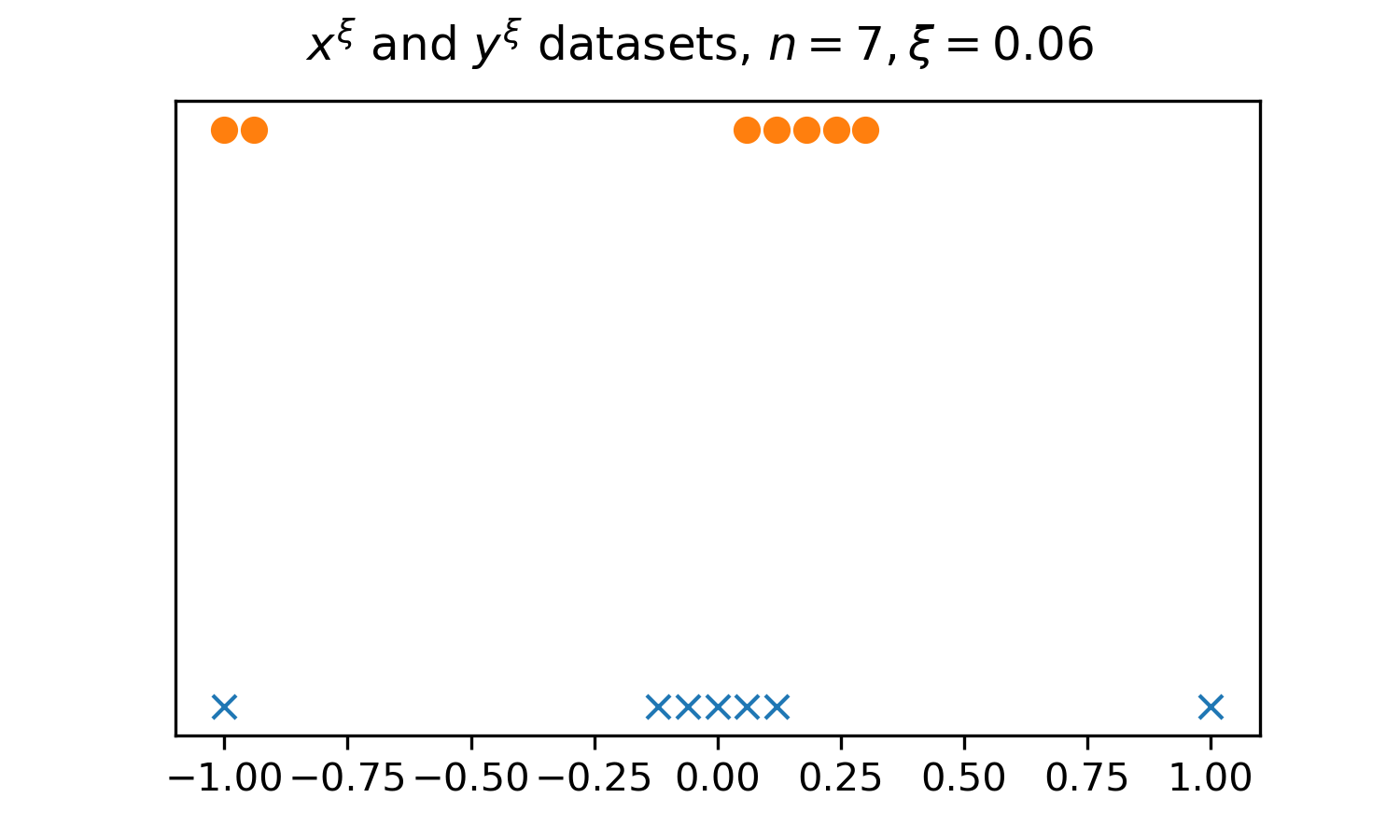}
 \end{subfigure}
\hspace{-2mm}
\centering
 \begin{subfigure}[t]{0.55\textwidth}
 \centering
\includegraphics[width=1.1\linewidth]{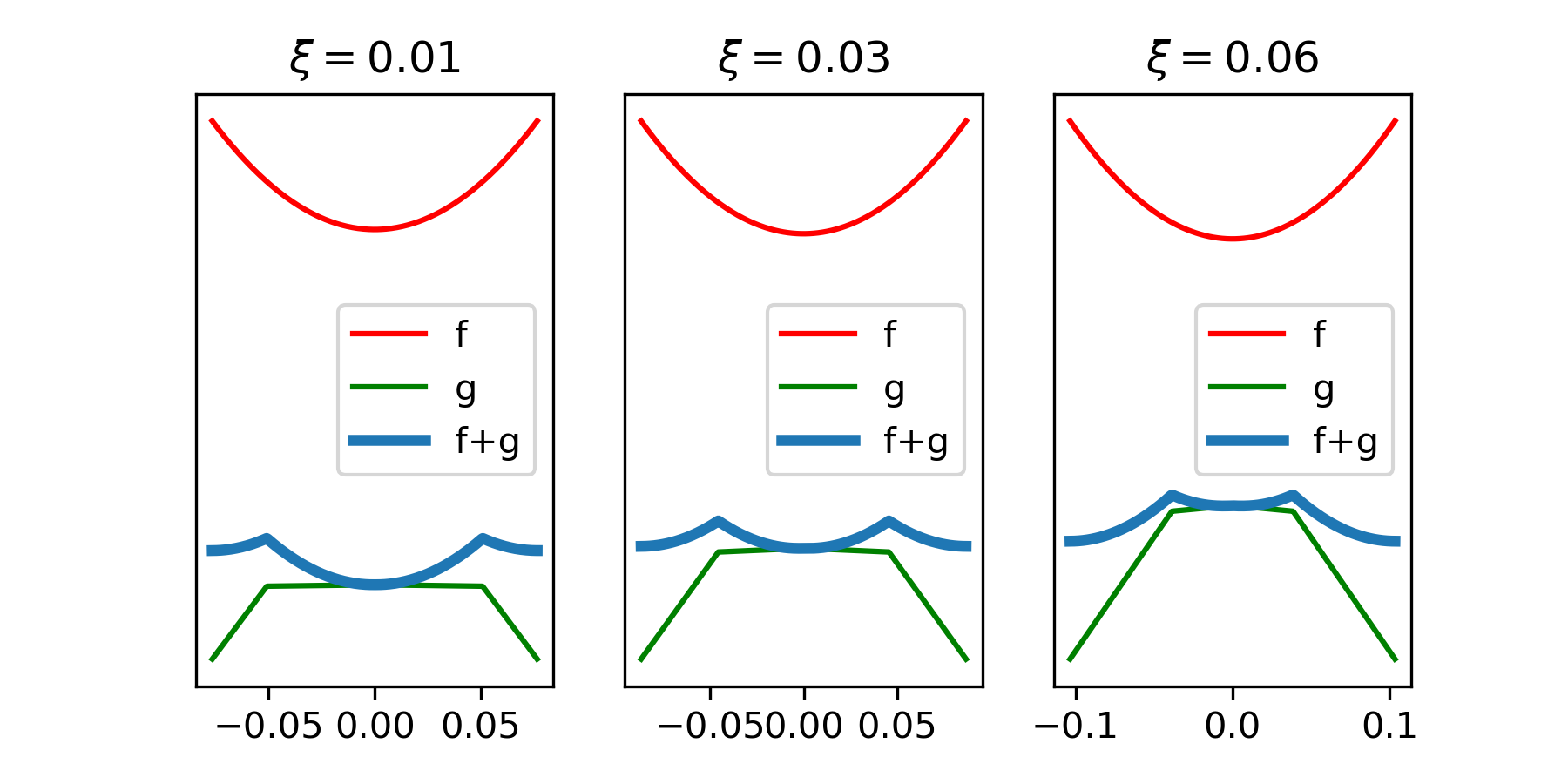}
 \end{subfigure}
 \captionsetup{width=1\linewidth}
 \caption{(Left) The datasets $x^\xi$ and $y^\xi$ from \eqref{eq:GW_1d_datasets}, for $n=7$ and $\xi=0.06$; (Right) The functions $f$, $g$, and $f+g$ on the interval $a\in[0.5 W_-,0.5 W_+]$, for $\xi=0.01,\,0.03,\,0.06$. When $\xi=0.01$, the minimizer of $f+g$ is attained outside the boundary and thus the corresponding optimal permutation is neither the identity nor the anti-identity.
 }
  \label{fig:1d_example}
\end{figure*}

\cref{prop:1d_opt} thus implies that the identity and anti-identity can only optimize the GW distance when \eqref{eq:GW_1d_S2} achieves its minimum on the boundary. 
However, as $f$ is convex and $g$ is concave, it is not necessarily the case that $\cA^\star$ contains only boundary points, as other values may be optimal. 
To visualize this behavior, \cref{fig:1d_example} plots the two datasets $x^\xi$ and $y^\xi$ from \eqref{eq:GW_1d_datasets} and the corresponding $f$, $g$, and $f+g$ functions for different $\xi$ values. While the infimum is achieved at the boundaries for $\xi=0.06$ and $\xi=0.03$, when $\xi=0.01$ the optimizing $a^\star\approx0$ and, by \cref{prop:1d_opt}, the optimal permutation is different from $\mathrm{id}$ and $\overline{\mathrm{id}}$. The structure of the corresponding optimal coupling is not trivial, as already seen from the proof of Proposition 1 from \cite{beinert2022assignment}. Better understanding the~relation between optimal $a$ values and their corresponding couplings is an interesting research avenue. Nevertheless, the above clarifies the optimization structure of the one-dimensional GW problem and provides a visual argument for the suboptimality of $\mathrm{id}$ and $\overline{\mathrm{id}}$ in the example above.

\section{Proofs of Main Theorems}\label{sec:proofs}

\subsection{Proof of Theorem \ref{thm:egw_duality}}\label{proof:thm:egw_duality} 

For completeness we first show the decomposition of $\EGW_\eps(\mu,\nu)$ for centered $\mu,\nu$, given in \eqref{eq:egw_decomposition}. Expanding the $(2,2)$-GW cost we have
\begin{align*}\numberthis\label{eq:proof_of_EGW_decomposition}
 &\EGW_\eps(\mu,\nu)=  \int\|x\mspace{-1mu}-\mspace{-1mu}x'\|^4 d  \mu\mspace{-2mu}\otimes\mspace{-2mu}\mu(x,x') \mspace{-2mu}+ \mspace{-3mu}\int \mspace{-3mu}\|y\mspace{-1mu}-\mspace{-1mu}y'\|^4d  \nu\mspace{-2mu}\otimes\mspace{-2mu}\nu(y,y') \mspace{-1mu}-\mspace{-1mu} 4\mspace{-3mu}\int \mspace{-3mu}\|x\|^2\|y\|^2  d  \mu\mspace{-2mu}\otimes\mspace{-2mu} \nu(x,y)\\
 &\hspace{3em} +\inf_{\pi\in\Pi(\mu,\nu)} \Bigg\{-4 \int \|x\|^2\|y\|^2  d \pi(x,y) -8\int \langle x,x'\rangle\langle y,y'\rangle d \pi\otimes\pi(x,y,x',y')\\
 &\hspace{6em}+8\int \big(\langle x,x' \rangle \|y\|^2 + \|x\|^2\langle y,y' \rangle\big) d \pi\otimes\pi(x,y,x',y')+ \eps \dkl(\pi \| \mu\otimes \nu)\Bigg\}.
\end{align*}
By the centering assumption, the term in the last line nullifies, while the first and second lines on the RHS correspond to $\EGW_{1}(\mu,\nu)$ and $\EGW^2_{\eps} (\mu,\nu)$, respectively.

We now move to derive the dual form for $\EGW^2_{\eps} $. Recall that $M_{\mu,\nu}\coloneqq \sqrt{M_2(\mu)M_2(\nu)}$, $\cD_{M_{\mu,\nu}}\coloneqq  [-M_{\mu,\nu}/2,M_{\mu,\nu}/2]^{d_x\times d_y}$. Consider:
\begin{align*}
    \EGW^2_{\eps} (\mu,\nu) &= \mspace{-3mu}\inf_{\pi\in\Pi(\mu,\nu)}\int- 4\|x\|^2\|y\|^2  d \pi(x,y) \mspace{-3mu}-8\mspace{-3mu}\sum_{\substack{1\leq i\leq d_x\\1\leq j \leq d_y}}\mspace{-3mu}\Big(\mspace{-3mu}\int\mspace{-3mu}  x_iy_j d \pi(x,y)\Big)^2 \mspace{-3mu}+ \eps \dkl(\pi \|\mu\otimes\nu)\\
    &= \mspace{-3mu}\inf_{\pi\in\Pi(\mu,\nu)}\int\mspace{-3mu}- 4\|x\|^2\|y\|^2  d \pi(x,y) \mspace{-3mu}+\mspace{-3mu}\sum_{\substack{1\leq i\leq d_x\\1\leq j \leq d_y}} \mspace{-3mu}  \inf_{|a_{ij}|\leq \frac{M_{\mu,\nu}}{2}}  \mspace{-3mu}32\left(\mspace{-3mu}a_{ij}^2 \mspace{-3mu}-\mspace{-3mu} \int  a_{ij} x_iy_j d \pi(x,y)\mspace{-3mu}\right)\\
    &\mspace{461mu} + \eps \dkl(\pi \|\mu\otimes\nu)\\
    &=\mspace{-3mu}\inf_{\mathbf{A}\in\cD_{M_{\mu,\nu}}} \inf_{\pi\in\Pi(\mu,\nu)}  \int- 4\|x\|^2\|y\|^2  d \pi(x,y) \mspace{-1mu}+\mspace{-8mu}\sum_{\substack{1\leq i\leq d_x\\1\leq j \leq d_y}}  \mspace{-5mu}32\left(a_{ij}^2    -\int  a_{ij} x_iy_j d \pi(x,y)\right)\\
    &\mspace{461mu}+ \eps \dkl(\pi \|\mu\otimes\nu)\\
    &=\mspace{-3mu}\inf_{\mathbf{A}\in\cD_{M_{\mu,\nu}}} \mspace{-3mu}32\|\mathbf{A}\|_\F^2+ \inf_{\pi\in\Pi(\mu,\nu)}  \int c_\mathbf{A}(x,y)  d \pi(x,y)  + \eps \dkl(\pi \|\mu\otimes\nu)\\
\end{align*}

\noindent where in the second step we introduced $a_{ij}$ whose optimum is achieved at $\frac 12 \int x_iy_j d \pi(x,y)$. This means we may restrict the optimization to $\cD_{M_{\mu,\nu}}$ without affecting the value since $\int x_iy_j d \pi(x,y)\leq M_{\mu,\nu}$ by the Cauchy–Schwarz inequality. We also switched the order of the two $\inf$ and claimed that the optimums are achieved, which follows from the lower semicontinuity in $\pi$ and $\mathbf{A}$. We conclude by identifying the EOT problem $\mathsf{OT}_{\mathbf{A},\eps}$ in the last line. 
\qed

\subsection{Proof of Theorem \ref{thm:egw_sample_complex}}\label{proof:thm:egw_sample_complex} 
We only prove the two-sample case; the one-sample derivation is similar, except that in \eqref{eq:sample_complex_emp_process} ahead one would only consider the empirical process induced by $\mu$. 
Proofs of technical lemmas stated throughout this proof are given in~\cref{appen:proof:egw_sample_complex_lemmas}. We proceed with the three steps described in the proof outline, after the theorem statement.

\medskip

\noindent\underline{\textit{Decomposition}:} Recall from \eqref{eq:egw_decomposition} that $\EGW_\eps(\mu,\nu) = \EGW^1(\mu,\nu) + \EGW^2_{\eps} (\mu,\nu)$ holds if $\mu,\nu$ are centered distributions. This decomposition is convenient for
analysis as it allows  separately treating the marginals- and the coupling-dependents terms. Namely, we would like to have
\[
\big|\EGW_\eps(\mu,\nu)-\EGW_\eps(\hat{\mu}_n,\hat{\nu}_n)\big|\leq \big|\EGW^1(\mu,\nu)-\EGW^1(\hat{\mu}_n,\hat{\nu}_n)\big| + \big|\EGW^2_{\eps} (\mu,\nu)-\EGW^2_{\eps} (\hat{\mu}_n,\hat{\nu}_n)\big|.
\]
However, while the EGW distance $\EGW_\eps$ is translation invariant and we may assume $\int x d \mu=\int y d \nu=0$ w.l.o.g., the empirical measures $\hat{\mu}_n,\hat{\nu}_n$ are generally not centered and the decomposition into $\EGW^1$ and $\EGW^2_{\eps} $ may not hold. To amend this, we center $\hat{\mu}_n,\hat{\nu}_n$ and quantify the bias that this incurs on $\EGW_\eps$. This is stated in the following lemma, which is proven in \cref{appen:proof:lemma:bias}

\begin{lemma}[Centering bias]\label{lemma:bias}
If $\mu,\nu$ are centered, then 
\[
\EE\big[\big|\EGW_\eps(\mu,\nu)\mspace{-3mu}-\mspace{-3mu}\EGW_\eps(\hat{\mu}_n,\hat{\nu}_n)\big|\big]\lesssim \EE\big[\big|\EGW^1(\mu,\nu)\mspace{-3mu}-\mspace{-3mu}\EGW^1(\hat{\mu}_n,\hat{\nu}_n)\big|\big] + \EE\big[\big|\EGW^2_{\eps} (\mu,\nu)-\EGW^2_{\eps} (\hat{\mu}_n,\hat{\nu}_n)\big|\big]+\frac{\sigma^2}{\sqrt{n}} .
\]
\end{lemma}
\noindent Given this decomposition, we proceed to separately treat the empirical errors of $\EGW^1$ and $\EGW^2_{\eps}$. 

\medskip

\paragraph{\textbf{Sample complexity of $\bm{\EGW^1}$.}}
The analysis of $\EGW^1$ reduces to estimating moments of $\mu,\nu$, with parametric convergence rate for the error. The following lemma is proven in \cref{appen:proof:lemma:EGW_S1}.

\begin{lemma}[$\EGW^1$ parametric rate]\label{lemma:EGW_S1}
    If $\mu,\nu$ are 4-sub-Weibull with parameter $\sigma^2>0$, then 
\[
    \EE\big[\big|\EGW^1(\mu,\nu) - \EGW^1(\hat{\mu}_n,\hat{\nu}_n)\big|\big] 
    \lesssim \frac{1+\sigma^4}{\sqrt{n}}.
\]
\end{lemma}

\paragraph{\textbf{Sample complexity of $\bm{\EGW^2_{\eps} }$.}} 
It remains to analyze the sample complexity of $\EGW^2_{\eps}$. To that end, we use the dual form of $\EGW^2_{1}$ to control its empirical error by the supremum of an empirical process indexed by optimal EGW potentials. We then derive regularity properties of the potentials, based on which standard empirical process techniques via entropy integral bounds yield the desired rate. For ease of presentation, the derivation is split into several~steps.

\medskip
\paragraph{(i) \underline{Normalization and reduction to EOT}}
Observe that if $\mu^\eps,\nu^\eps$ are the pushforward measures of $\mu,\nu$ through the mapping $x\mapsto \eps^{-1/4} x$, then we have $\EGW^2_{\eps} (\mu,\nu) = \eps \EGW^2_{1}(\mu^\eps,\nu^\eps)$. Also note that $\mu^\eps,\nu^\eps$ are 4-sub-Weibull distributions with parameter $\sigma^2/\eps$. Thus, we henceforth set $\eps=1$ and later adapt to a general $\eps>0$ using the aforementioned observation. Invoking \cref{thm:egw_duality} for $\EGW^2_{1}$, while optimizing over $\mathbf{A}\in\cD_M$, for some $M\geq M_{\mu,\nu}$ to be specified later (which does not change the optimization value since $A^\star\in\cD_{M_{\mu,\nu}}$), we obtain 
\begin{equation}
\big|\EGW^2_{1}(\mu,\nu)-\EGW^2_{1}(\hat{\mu}_n,\hat{\nu}_n)\big| \leq \sup_{\mathbf{A}\in \cD_{M}} \big|\mathsf{OT}_{\mathbf{A},1}(\mu,\nu)-\mathsf{OT}_{\mathbf{A},1}(\hat{\mu}_n,\hat{\nu}_n)\big|,\label{eq:sample_complex_EOT}
\end{equation}
which reduces the analysis to that of EOT with the cost function $c_{\mathbf{A}}$, uniformly over $\mathbf{A}\in\cD_M$. We next analyze the regularity of optimal dual potentials for the EOT problems on the RHS above. This regularity theory is later used to decompose the RHS into suprema of empirical processes indexed by these potentials and to analyze their expected convergence rates.

\medskip
\paragraph{(ii) \underline{Smoothness of EOT potentials}}
 {To simplify notation, we henceforth drop the subscript $\mathbf{A}$ from the EOT potentials $(\varphi_\mathbf{A},\psi_\mathbf{A})$ for $\mathsf{OT}_{\mathbf{A},\eps}(\mu,\nu)$, writing only $(\varphi,\psi)$.} The following lemma provides bounds on the magnitude of partial derivatives (of any order) of EOT potentials between any two sub-Weibull distribution, w.r.t. the cost $c_\mathbf{A}$, uniformly in $\mathbf{A}\in\cD_M$. 
To state the result, for any $\sigma,M>0$, let $\cF_{\sigma,M}$ be the class of $\cC^\infty(\RR^{d_x})$ functions $\varphi$ satisfying:
\begin{align*}
    &\varphi(x)\mspace{-3mu}\leq \mspace{-3mu}4\sigma^2+8M\sqrt{2\sigma d_xd_y}\left(\sqrt{\sigma}+2\|x\|^2\right)\\
    &\mspace{-3mu}-\mspace{-3mu}\varphi(x)\mspace{-3mu}\leq \mspace{-3mu}\log2\mspace{-3mu}+\mspace{-3mu}4\sigma^2\mspace{-3mu}+\mspace{-3mu}8M\mspace{-3mu}\sqrt{2\sigma d_xd_y}\mspace{-3mu}\left(\mspace{-1mu}1\mspace{-3mu}+\mspace{-3mu}\sqrt{\sigma}\mspace{-3mu}+\mspace{-3mu}\frac{\|x\|^2}{\sqrt{2\sigma}}\right)\mspace{-3mu}+\mspace{-3mu}8\sigma^2\mspace{-3mu}\left(2M\sqrt{d_xd_y}\big(1\mspace{-3mu}+\mspace{-3mu}\sqrt{2\sigma}\big)\mspace{-3mu}+\mspace{-3mu}\|x\|^2\right)^{\mspace{-3mu}2} \\
    &|D^\alpha \varphi(x)|\mspace{-3mu}\leq \mspace{-3mu}C_{\alpha} \mspace{-3mu}\big(1\mspace{-3mu}+\mspace{-3mu}M\sqrt{d_y}\mspace{-3mu}+\mspace{-3mu}\|x\|\big)^{\mspace{-3mu}|\alpha|}\left(1\mspace{-3mu} +\mspace{-3mu}\sigma^{|\alpha|} \mspace{-3mu}+ \mspace{-3mu}\big(1+M\sqrt{d_xd_y}\big)^{|\alpha|}\big(1\mspace{-3mu}+\mspace{-3mu}\sigma^5 \mspace{-3mu}+\mspace{-3mu}\sigma^4\|x\|^4\big)^{\frac{|\alpha|}{2}}\right),
\end{align*}
for all multi-indices $\alpha\in\NN_0^{d_x}$ and some constant $C_\alpha>0$ that depends only on $\alpha$. Define the class $\cG_{\sigma,M}$ analogously but for functions $\psi:\RR^{d_y}\to\RR$. 
\begin{lemma}[Uniform regularity of EOT potentials]\label{lemma:smoothness_of_EGW_potential}
Fix $M\geq M_{\mu,\nu}$, $\mathbf{A}\in\cD_M$, and suppose that $\mu,\nu$ are 4-sub-Weibull with parameter $\sigma^2$. Then there exist optimal EOT potentials $(\varphi,\psi)$ for $\mathsf{OT}_{\mathbf{A},1}(\mu,\nu)$ from \eqref{eq:EOT}, such that $\varphi\in\cF_{M,\sigma}$ and $\psi\in\cG_{M,\sigma}$.
\end{lemma}

The proof of the lemma is deferred to \cref{appen:proof:lemma:smoothness_of_EGW_potential}. The key idea is that given optimal EOT potentials $(\varphi_0,\psi_0)$, we may define new potentials $(\varphi,\psi)$ via the Schr\"{o}dinger systems \eqref{EQ:Schrodinger} and show that the pairs agree $\mu\otimes\nu$-a.s. Consequently, $(\varphi,\psi)$ are also optimal for $\mathsf{OT}_{\mathbf{A},1}(\mu,\nu)$, but they enjoy an explicit representation via the Schr\"{o}dinger systems, which evidently renders $(\varphi,\psi)$ smooth functions.

Lemma \ref{lemma:smoothness_of_EGW_potential} allows restricting the optimization domain in the dual form of $\mathsf{OT}_{\mathbf{A},1}$ from $L^1(\mu)\times L^1(\nu)$ to $\cF_{M,\tilde\sigma}\times \cG_{M,\tilde\sigma}$, for an appropriately chosen $\tilde{\sigma}$. Let $\tilde{\sigma}$ be the random variables defined as the smallest $\sigma'>0$ such that $\mu,\nu,\hat{\mu}_n,\hat{\nu}_n$ are all 4-sub-Weibull with parameter $\sigma'^2$. Clearly, any $\varphi\in\cF_{M,\tilde\sigma}$ with $M=\sqrt{2}\tilde \sigma$ also satisfies
\begin{align*}
    |\varphi(x)|&\leq C_{d_x,d_y}\big(1+\tilde{\sigma}^5\big)\big(1+\|x\|^4\big)\\
    |D^\alpha \varphi(x)|&\leq C_{\alpha,d_x,d_y}   \big(1+\tilde{\sigma}^{9|\alpha|/2}\big)\big(1+\|x\|^{3|\alpha|}\big),\qquad \forall \alpha\in\NN_0^{d_x},
\end{align*}
and similarly for $\psi\in\cG_{M,\tilde\sigma}$. Recalling that \cref{thm:egw_duality} requires $M^2$ to be at least as large as the product of the 2nd moments of the involved distributions, we note that $M=\sqrt{2}\tilde \sigma$ is feasible for $\mathsf{OT}_{\mathbf{A},1}(\cdot,\cdot)$ between any pair from $\{\mu,\nu,\hat{\mu}_n\}$ and $(\nu,\hat{\nu}_n)$, for any $\mathbf{A}\in\cD_M$. Lastly, define the H\"older class
\begin{align*}
    \cF_s=\left\{\varphi:\RR^{d_x}\to\RR:|\varphi|\leq C_{s,d_x,d_y}\big(1+\|\cdot\|^4\big), |D^\alpha \varphi|\leq C_{s,d_x,d_y} \big(1+\|\cdot\|^{3s}\big),\ \forall |\alpha|\leq s\right\},
\end{align*}
with $\cG_s$ defined analogously. We conclude that for each $\mathbf{A}\in\cD_M$, any smooth potentials $(\varphi,\psi)$ for the corresponding EGW problem satisfy $(1+\tilde{\sigma}^{5s})^{-1}\varphi\in\cF_s$ and $(1+\tilde{\sigma}^{5s})^{-1}\psi\in\cG_s$. This regularity of potentials will be used to derive the parametric rate of convergence for empirical $\EGW^2_{1}$, following the decomposition presented in the next part. 

\medskip
\paragraph{(iii) \underline{Decomposition into suprema of empirical processes}}
We upper bound the empirical estimation error of $\EGW^2_{1}$ by the suprema of empirical processes indexed by optimal potentials. To simplify notation, recall the shorthand $\rho \varphi\coloneqq \int \varphi d \rho$ for any signed Borel measure $\rho$. Starting from~\eqref{eq:sample_complex_EOT}, we have 
\begin{align*}
    \big|\EGW^2_{1}(\mu,\nu)-\EGW^2_{1}(\hat{\mu}_n,\hat{\nu}_n)\big| &\leq \sup_{\mathbf{A}\in \cD_{M}} \big|\mathsf{OT}_{\mathbf{A},1}(\mu,\nu)-\mathsf{OT}_{\mathbf{A},1}(\hat{\mu}_n,\hat{\nu}_n)\big|\\
    &\lesssim \big(1+\tilde{\sigma}^{5s}\big)\bigg(\sup_{\varphi\in\cF_s}\big|(\mu-\hat{\mu}_n)\varphi\big| + \sup_{\psi\in\cG_s}\big|(\nu-\hat{\nu}_n)\psi\big|\bigg),\numberthis\label{eq:sample_complex_emp_process}
\end{align*}
where the second inequality follows by \cite[Proposition 2]{mena2019statistical}, which uses the fact that the optimal EOT potentials between $(\mu,\nu)$, $(\hat{\mu}_n,\nu)$, and $(\hat{\mu}_n,\hat{\nu}_n)$ belong to $\cF_s\times\cG_s$. We have also used the fact that \cref{lemma:smoothness_of_EGW_potential} holds uniformly in $\mathbf{A}\in\cD_M$ to remove the supremum.

\medskip

\paragraph{(iv) \underline{Sample complexity analysis}}
We are now in place to establish that $\EGW^2_{1}(\hat{\mu}_n,\hat{\nu}_n)$ converges towards $\EGW^2_{1}(\mu,\nu)$ at the parametric rate. More specifically, we will show    
\begin{equation}
    \EE\Big[\big|\EGW^2_{1}(\mu,\nu)-\EGW^2_{1}(\hat{\mu}_n,\hat{\nu}_n)\big|\Big]\lesssim_{d_x,d_y}   \left(1+\sigma^{7\left\lceil\frac{d_x\vee d_y}{2}\right\rceil+9}\right)n^{-\frac 12}.\label{eq:sample_complex_to_show}
\end{equation}
Starting from the RHS of \eqref{eq:sample_complex_emp_process}, we present the analysis of the first supremum, with the second one being treated similarly. We bound it as 
\begin{equation}
\EE\left[\big(1+\tilde{\sigma}^{5s}\big)\sup\nolimits_{\varphi\in\cF_s}\big|(\mu-\hat{\mu}_n)\varphi\big|\right]\leq \sqrt{\EE\left[\big(1+\tilde{\sigma}^{5s}\big)^2\right]\EE\left[\left(\sup\nolimits_{\varphi\in\cF_s}(\mu-\hat{\mu}_n)\varphi\right)^2\right]},\label{eq:sample_complex_UB}
\end{equation}
and proceed to bound the second term. By Theorem 3.5.1. from \cite{gine2016mathematical}, we have 
\begin{equation}
\EE\left[\left(\sup_{\varphi\in\cF_s} (\mu-\hat{\mu}_n)\varphi\right)^2\right] \lesssim_{d_x} \frac{1}{n} \EE\left(\int_0^{\sqrt{\max_{\varphi\in\cF_s} \|\varphi\|^2_{L^2(\hat{\mu}_n)} }}\mspace{-4mu} \sqrt{\log\Big( 2N\big(\xi,\cF_s,L^2(\hat{\mu}_n)\big)\Big) }  d \xi \right)^{\mspace{-3mu}2}\mspace{-7mu}.\label{eq:sample_complex_decomp}
\end{equation}
The integration domain is bounded by observing that
\[
    \max_{\varphi\in\cF_s} \|\varphi\|^2_{L^2(\hat{\mu}_n)}\leq C_{d_x,d_y}\frac{1}{n} \sum_{i=1}^n (1+\|x_i\|^8)\leq C_{d_x,d_y} \big(1+\sigma^4L\big)
\]
where $L\coloneqq \frac{1}{n}\sum_{i=1}^n e^{\frac{\|x_i\|^4}{2\sigma^2}}$ satisfies $\EE[L]\leq2$. To control the integrand, we apply Corollary 2.7.4. from \cite{van1996weak}
(see also \cite{mena2019statistical} Proposition 3) as follows. First, define $Q_j\coloneqq  [-2^j\sqrt{\sigma},2^j\sqrt{\sigma}]^{d_x}$ for $j\in\NN_0$, and partition $\RR^{d_x}$ into the sets $I_j = Q_j\backslash Q_{j-1}$. Note that the
Lebesgue measure of each $\big\{x\in\RR^{d_x}:\, \|x-I_j\|\leq1\big\}$ is bounded by $C_{d_x}(1+2^{jd_x}\sigma^{d_x/2})$, and by Markov's inequality we further obtain $\hat{\mu}_n(I_j)\leq Le^{-2^{4j-5}}$. Lastly, for any $j\in\NN_0$ and $\varphi\in\cF_s$, the restriction $\varphi|_{I_j}$ has a $\cC^s(I_j)$-H\"older norm bounded by $C_{s,d_x}(1+\sigma^{3s/2}) 2^{3js}$. This verifies the conditions of \cite[Corollary 2.7.4.]{van1996weak}, which we invoke with $s=\lceil d_x/2 \rceil+1$, $V=d_x/s$, and $r=2$, to get
\begin{align*}
    \log N\big(&\xi,\cF_s,L^2(\hat{\mu}_n)\big)\\
    &\leq \log N_{[\  ]}\big(2\xi,\cF_s,L^2(\hat{\mu}_n)\big)\\
    &\leq C_{d_x,d_y} \xi^{-V} L^{\frac{V}{r}} \left(  \sum_{j=0}^\infty \left(1+2^{jd_x}\sigma^{\frac{d_x}{2}}\right)^{\frac{r}{V+r}}    \big(e^{-2^{4j-5}}\big)^{\frac{V}{V+r}} \left(\big(1+\sigma^{\frac{3s}{2}}) 2^{3js}\right)^{\frac{Vr}{V+r}}\right)^{\frac{V+r}{r}}\\
    &\leq C_{d_x,d_y} \xi^{-V} L^{\frac{V}{r}} \left(1+\sigma^{\frac{d_x+3sV}{2}}\right) \left(  \sum_{j=0}^\infty   e^{-\frac{2^{4j-5} V}{V+r}}  2^{\frac{(3sV+d_x)jr}{V+r}}  \right)^{\frac{V+r}{r}}\\
    &\leq C_{d_x,d_y} \xi^{-\frac{d_x}{s}} L^{\frac{d_x}{2s}} \big(1+\sigma^{2d_x}\big)
\end{align*}
where the last line follows because the summation is finite and only depends on $d_x$. Inserting this bound back into \eqref{eq:sample_complex_decomp}, we have
\begin{align*}
\EE\left[\left(\sup\nolimits_{\varphi\in\cF_s} (\mu-\hat{\mu}_n)\varphi\right)^2\right]  &\lesssim_{d_x,d_y} \frac{1}{n} \EE\left[\left( \int_0^{\sqrt{ 1+\sigma^4L}} \sqrt{ \xi^{-\frac{d_x}{s}} L^{\frac{d_x}{2s}} \big(1+\sigma^{2d_x}\big)}  d \xi \right)^2\right]\\
    &\lesssim_{d_x,d_y} \frac{(1+\sigma^{2d_x})}{n} \EE \left[L^{\frac{d_x}{2s}} (1+\sigma^4L)^{1-\frac{d_x}{2s}}\right] \\
    &\lesssim_{d_x,d_y} \big(1+\sigma^{2(2+d_x)}\big)n^{-1}.
\end{align*}
In light of \eqref{eq:sample_complex_UB}, it remains to bound the appropriate moment of $\tilde{\sigma}$. For any $k\in\NN$, set
\[
\tau_k^2 = \sigma^2\vee\left(\frac{k\sigma^2}{n}\sum_{i=1}^n e^{\frac{\|x_i\|^4}{2k\sigma^2}}\right)\vee\left(\frac{k\sigma^2}{n}\sum_{i=1}^n e^{\frac{\|y_i\|^4}{2k\sigma^2}}\right)
\]
so that $\mu,\nu,\hat{\mu}_n,\hat{\nu}_n$ are all 4-sub-Weibull with parameter $\tau_k^2$; cf. \cite[Lemma 4]{mena2019statistical}. Therefore
\[\EE[\tilde{\sigma}^{2k}]\leq \EE[\tau_k^{2k}] \leq \sigma^{2k} + \frac{k^k \sigma^{2k}}{n}\EE\left[ \sum_{i=1}^n e^{\frac{\|x_i\|^4}{2\sigma^2}} + e^{\frac{\|y_i\|^4}{2\sigma^2}} \right] \leq \big(1+4k^k\big)\sigma^{2k}.\]

Combining all the pieces leads to:
\[
\EE\left[\big(1+\tilde{\sigma}^{5s}\big)\sup\nolimits_{\varphi\in\cF_s}(\mu-\hat{\mu}_n)\varphi\right]\mspace{-3mu}\lesssim_{d_x,d_y} \mspace{-5mu}\sqrt{\EE\left[\big(1+\tilde{\sigma}^{5s}\big)^2\right]\mspace{-3mu}\frac{(1+\sigma^{2d_x+4})}{n}}\mspace{-3mu}\lesssim\mspace{-3mu}\frac{(1+\sigma^{9\lceil d_x/2\rceil+11})}{\sqrt{n}} 
\]
with a similar bound holding for the corresponding term with $\cF_s$ replaced by $\cG_s$. Together with \eqref{eq:sample_complex_emp_process}, these two bounds imply \eqref{eq:sample_complex_to_show}.\qed

\subsection{Proof of Theorem \ref{thm:GW_sample_complex}} \label{sec:GW_sample_complex_proof}

 {

\medskip
\paragraph{\textbf{Upper bounds}}

 {We maintain our convention of suppressing the subscript $\mathbf{A}$ from our notation for optimal dual potentials for the OT problem with cost $c_\mathbf{A}$, simply writing $(\varphi,\psi)$.} As in the proof of \cref{thm:egw_sample_complex}, we only prove the two-sample case. The one-sample result follows similarly. Derivations of technical lemmas stated throughout this proof are deferred to \cref{appen:proof:GW_sample_complex_lemmas}.

\medskip

Assume w.l.o.g. that $\mu,\nu$ are centered and recall that we have the decomposition $\GW(\mu,\nu)^2=\EGW^1(\mu,\nu)+\EGW^2(\mu,\nu)$. To split our sample complexity analysis into those of $\EGW^1$ and $\EGW^{2}$, we again need to account for the fact that empirical measures are generally not centered. Let $\tilde{\mu}_n$ and $\tilde{\nu}_n$ be centered versions of the empirical measures $\hat{\mu}_n$ and $\hat{\nu}_n$, respectively. Following the same steps leading to \eqref{eq:sample_complex_bias_S1} and \eqref{eq:sample_complex_bias_S2}, we observe that 
\[
    \EE\big[\big|\EGW^1(\hat{\mu}_n,\hat{\nu}_n)-\EGW^1(\tilde{\mu}_n,\tilde{\nu}_n)\big|\big] \vee    \EE\big[\big|\EGW^2(\hat{\mu}_n,\hat{\nu}_n)-\EGW^2(\tilde{\mu}_n,\tilde{\nu}_n)\big|\big] \lesssim \frac{R^4}{\sqrt{n}},
    \]
which also uses the fact that any distribution whose support diameter is bounded by $R$ is trivially 4-sub-Weibull with parameter $R^4$. Consequently, we may split
\begin{equation}
\EE\big[\big|\GW(\mu,\nu)^2\mspace{-2mu}-\mspace{-2mu}\GW(\hat{\mu}_n,\hat{\nu}_n)^2\big|\big]\mspace{-2mu}\leq\mspace{-2mu}\EE\big[\big|\EGW^1(\mu,\nu)\mspace{-2mu}-\EGW^1(\hat{\mu}_n,\hat{\nu}_n)\big|\big]\mspace{-2mu} + \EE\big[\big|\EGW^2(\mu,\nu)\mspace{-2mu}-\EGW^2(\hat{\mu}_n,\hat{\nu}_n)\big|\big]\mspace{-2mu}+\frac{R^4}{\sqrt{n}},\label{eq:GW_sample_complex_bias}
\end{equation}
and proceed with a separate analysis for $\EGW^1$ and $\EGW^2$.

\medskip
For $\EGW^1$, we follow the steps leading to \eqref{eq:S1_decomposition} in the EGW sample complexity analysis and use the fact that $\mu,\nu$ are 4-sub-Weibull with parameter $R^4$ to deduce
\begin{equation}
        \EE\big[\big|\EGW^1(\hat{\mu}_n,\hat{\nu}_n)-\EGW^1(\mu,\nu)\big|\big] \lesssim  \frac{R^4}{\sqrt{n}}.\label{eq:GW_sample_complex_S1}
\end{equation}
To treat $\EGW^2$, we start from the variational representation from \cref{cor:gw_duality} 
and choose $M=R^2\geq M_{\mu,\nu}$, which is evidently feasible. Invoking this result, we obtain 
\begin{equation}
    \big|\EGW^2(\mu,\nu)-\EGW^2(\hat{\mu}_n,\hat{\nu}_n)\big| \leq \sup_{\mathbf{A}\in\cD_{R^2}}\big|\mathsf{OT}_{\mathbf{A}}(\mu,\nu) - \mathsf{OT}_{\mathbf{A}}(\hat{\mu}_n,\hat{\nu}_n)\big|,\label{eq:GW_sample_complex_S2}
\end{equation}
and proceed to show that for any $\mathbf{A}\in\cD_{R^2}$, corresponding optimal dual potentials can be restricted to concave Lipschitz functions and their $c$-transforms (w.r.t. the cost function $c_\mathbf{A}$). 

\medskip
\paragraph{(i) \underline{Smoothness of OT potentials}}
Let
\begin{align*}
    \cF_R\coloneqq \left\{\varphi:B_{d_x}(0,R)\to\RR:\,
    \begin{array}{cc}
          \varphi\text{ concave, }\ \|\varphi\|_{\infty}\leq 1+10\big(1+4\sqrt{d_x d_y}\big)R^4, \\
          \quad\quad\quad\mspace{2mu}\|\varphi\|_{\lip}\leq8\big( 1+ 2\sqrt{d_x d_y}\big)R^3  \phantom{\bigg|}
    \end{array}\right\}
\end{align*}
and define $\cG_R$ analogously over $B_{d_y}(0,R)$. Recall that the $c$-transform of $\varphi:\RR^{d_x}\to\RR$ w.r.t.  $c_\mathbf{A}$ is a new function $\varphi^c:\RR^{d_y}\to\RR$, given by 
$\varphi^c = \inf_{x\in \cX} c_{\mathbf{A}}(x,\cdot) - \varphi(x)$. The~next lemma allows restricting the set of optimal dual potentials for $\mathsf{OT}_{\mathbf{A}}(\mu,\nu)$ to pairs $(\varphi,\varphi^c)\in \cF_R\times \cG_R$. 

\begin{lemma}[Uniform regularity of OT potentials]
\label{lemma:smoothness_of_GW_potential}
Fix $R>0$ and suppose that $(\mu,\nu)\in \cP(\cX)\times  \cP(\cY)$, with $\cX\subset B_{d_x}(0,R)$ and $\cY\subset B_{d_y}(0,R)$. Then, for any $\mathbf{A}\in\cD_{R^2}$, there exist $\varphi\in\cF_R$ with $\varphi^c\in\cG_R$, such that $(\varphi,\varphi^c)$ is a pair of optimal dual potentials for $\mathsf{OT}_{\mathbf{A}}(\mu,\nu)$. 
\end{lemma}

The proof, which is given in \cref{appen:proof:lemma:smoothness_of_GW_potential}, arrives at the above properties by exploiting concavity of $c_{\mathbf{A}}$ and the $c$-transform representation of optimal dual pairs.

\medskip
\paragraph{(ii) \underline{Sample complexity analysis}}
Equipped with \cref{lemma:smoothness_of_GW_potential}, we are ready to conduct the sample complexity analysis. Suppose w.l.o.g. that $d_x\leq d_y$; otherwise, flip their roles in the derivation below. For each $\mathbf{A}\in\cD_{R^2}$, let $\Phi_{\mathbf{A}}$ be the class of of optimal dual potential pairs for $\mathsf{OT}_{\mathbf{A}}(\mu,\nu)$ (see \eqref{eq:OT_dual}). 
Define $\cF_\mathbf{A}\coloneqq\mathrm{proj}_{\cF_R}\big(\Phi_\mathbf{A} \cap (\cF_R\times\cG_R)\big)$ and let $\cF_\mathbf{A}^{\mspace{1mu}c}$ be its $c$-transform w.r.t. $c_\mathbf{A}$. 
We may now further upper bound the RHS of \eqref{eq:GW_sample_complex_S2}, to arrive at
\begin{align*}
    \EE\Big[\big|\EGW^2(\mu,\nu)-\EGW^2(\hat{\mu}_n,\hat{\nu}_n)\big|\Big] 
    &\leq \EE\left[\sup_{\varphi\in\cup_{\mathbf{A}}\cF_\mathbf{A}}\big|(\mu-\hat{\mu}_n)\varphi\big|\right] +\EE\left[\sup_{\psi\in\cup_\mathbf{A}\cF_\mathbf{A}^{\mspace{1mu}c}}\big|(\nu-\hat{\nu}_n)\psi\big|\right].\numberthis\label{eq:GW_sample_complex_S2_emp}
\end{align*}

As \cref{lemma:smoothness_of_GW_potential} implies that $\cup_\mathbf{A}\cF_\mathbf{A}\subset\cF_R$, the first term above is controlled by the expected supremum of an empirical process indexed by $\cF_R$. Dudley's entropy integral formula yields 
\begin{align*}
\EE\left[\sup_{\varphi\in\cF_R} \big|(\mu-\hat{\mu}_n)\varphi\big|\right] \lesssim \inf_{\alpha>0} \alpha + \frac{1}{\sqrt{n}} \int_\alpha^{2\sup_{\varphi\in\cF_R} \|\varphi\|_\infty} \sqrt{\log N(\xi,\cF_R,\|\cdot\|_\infty)}  d \xi.
\end{align*}
Theorem 1 from \cite{guntuboyina2012l1} provides a bound on the metric entropy of bounded, convex, Lipschitz functions, whereby if $\widetilde{\cF}_d\coloneqq \left\{f:B_{d}(0,1)\to\RR: f\text{ convex, } \|f\|_{\infty}\vee\|f\|_\lip\leq1\right\}$, then $\log N(\xi,\tilde\cF_{d},\|\cdot\|_\infty) \leq C_{d}\,\xi^{-\frac{d}{2}}$. For any $\varphi:\RR^d\to \RR$, define its rescaled version\footnote{With some abuse of notation, we apply this re-scaling transform to functions defined on spaces of possibly different dimensions without explicitly reflecting this in the notation.} $(S\varphi)(z)\coloneqq \varphi(Rz)/(1+C_{d_x,d_y} R^4)$, where $C_{d_x,d_y} = 10\big(1+4\sqrt{d_x d_y}\big)$, and note that $S\varphi\in\widetilde\cF_{d_x}$, for any $\varphi\in\cF_R$.
We also define the map $s:x\mapsto x/R$. Combining the above, for $d_x\geq 4$, we have
\begin{align*}
\EE\left[\sup_{\varphi\in\cF_R} \big|(\mu-\hat{\mu}_n)\varphi\big|\right]
    &\lesssim_{d_x,d_y} (1+ R^4)\EE\left[\sup_{\varphi\in\cF_R} \big|\big(s_\sharp\mu-s_\sharp\hat{\mu}_n\big)(S\varphi)\big|\right]\\
    &\lesssim_{d_x,d_y} (1+R^4)\left( \inf_{\alpha>0} \alpha + \frac{1}{\sqrt{n}} \int_\alpha^{2} \xi^{-\frac{d_x}{4}}  d \xi\right)\\
    &\lesssim_{d_x,d_y} (1+R^4)n^{-\frac{2}{d_x}}(\log n)^{\mathds{1}_{\{d_x=4\}}}.
\end{align*}
When $d_x<4$, the entropy integral is finite and we may pick $\alpha=0$. Hence, in this case, $\cF_R$ is a Donsker class and the resulting convergence rate is parametric $n^{-1/2}$. Altogether, we have
\begin{equation}
\EE\left[\sup_{\varphi\in\cup_\mathbf{A}\cF_\mathbf{A}}\big|(\mu - \hat{\mu}_n)\varphi\big|\right]\lesssim_{d_x,d_y}\big(1+R^4\big)n^{-\frac{2}{d_x\vee 4}}(\log n)^{\mathds{1}_{\{d_x=4\}}}.\label{eq:GW_sample_complex_term1}
\end{equation}

We now move to treat the second term on the RHS of \eqref{eq:GW_sample_complex_S2_emp}. First, observe that one may control it by the expected supremum of an empirical process indexed by $\cG_R$, which~is bounded by $\big(1+R^4\big)n^{-2/(d_y\vee 4)}(\log n)^{\mathds{1}_{\{d_y=4\}}}$ via similar steps as above. Together with \eqref{eq:GW_sample_complex_term1}, this would yield a two-sample empirical convergence rate bound of $n^{-2/(d_x\vee d_y\vee 4)}(\log n)^{\mathds{1}_{\{d_x\vee d_y=4\}}}$ for the squared $(2,2)$-GW distance. However, we aim to arrive at an upper bound that depends on the smaller dimension $d_x \wedge d_y$, as opposed to the larger one. As pointed out in Remark 5.6 of \cite{groppe2023lower}, this is possible by employing the LCA principle from \cite[Lemma 2.1]{hundrieser2022empirical}, which states that 
for any cost function $c$ and function class $\cF$, we have $N(\xi,\cF^c,\|\cdot\|_\infty)\leq N(\xi,\cF,\|\cdot\|_\infty)$. Starting from a rescaling step as before, we obtain
\begin{equation}
\EE\left[\sup_{\psi\in\cup_\mathbf{A}\cF_\mathbf{A}^{\mspace{1mu}c}}\big|(\nu-\hat{\nu}_n)\psi\big|\right]\lesssim_{d_x,d_y} (1+ R^4)\EE\left[\sup_{\psi\in \cup_\bA \cF_\bA^{\mspace{1mu}c}} \big|\big(s_\sharp\nu-s_\sharp\hat{\nu}_n\big)(S\psi)\big|\right].\label{eq:GW_sample_complexity_term2}
\end{equation}
Using the LCA principle, we have the following bound on the covering number of the union of rescaled $c$-transformed classes.

\begin{lemma}\label{lemma:covering}For any $\xi>0$, we have the covering bound
\[
N\left(\xi,\cup_{\bA\in\cD_{R^2}} S(\cF_\bA^{\mspace{1mu}c}),\|\cdot\|_\infty\right) \leq N\left(\frac{\xi}{64R^2},\cD_{R^2},\|\cdot\|_\op\right) N\left(\frac{\xi}{2},\widetilde\cF_{d_x},\|\cdot\|_\infty\right).
\]
\end{lemma}

Armed with the lemma, we proceed from \eqref{eq:GW_sample_complexity_term2} and, for $d_x\geq 4$, obtain
\begin{align*}
\EE\left[\sup_{\psi\in\cup_\bA\cF_\bA^c} \big|(\nu-\hat{\nu}_n)\psi\big|\right] 
    &\lesssim_{d_x,d_y} (1+R^4)\left( \inf_{\alpha>0} \alpha + \frac{1}{\sqrt{n}} \int_\alpha^{2} \xi^{-\frac{d_x}{4}} + \log\frac{R^4}{\xi} d \xi\right)\\
    &\lesssim_{d_x,d_y} (1+R^4)n^{-\frac{2}{d_x}}(\log n)^{\mathds{1}_{\{d_x=4\}}}.
\end{align*}
As before, when $d_x<4$, a parametric rate bound holds instead. Inserting the above along with 
\eqref{eq:GW_sample_complex_S2_emp} into \eqref{eq:GW_sample_complex_term1} concludes the proof of the two-sample upper bound for the squared distance.

\medskip
Lastly, observe that if $\GW(\mu,\nu)>0$, then the two-sample rate for $\GW(\mu,\nu)^2$ readily extends to $\GW(\mu,\nu)$, since $\EE\big[\big|\GW(\mu,\nu)-\GW(\hat{\mu}_n,\hat{\nu}_n)\big|\big]\leq \GW(\mu,\nu)^{-1}\EE\big[\big|\GW(\mu,\nu)^2-\GW(\hat{\mu}_n,\hat{\nu}_n)^2\big|\big]$, and similarly for the one-sample case. We note, however, that unlike the bounds for $\GW^2$, this  bound is not uniform over pairs of distributions with compact supports.

\paragraph{\textbf{Lower bounds}}
We now move to establish the lower bounds. As the parametric lower bound of $n^{-1/2}$ trivially holds for our problem, we assume w.l.o.g. that $4<d_x\leq d_y$ and $R=4$.\footnote{To treat general $R$, one only needs to include a factor of $R^4/256$ in front of the one- and two-sample~errors. 
} Denoting $d\coloneqq d_x$, we shall construct compactly supported distributions $\mu,\nu\in \RR^{d}$ with the desired $n^{-2/d}$ empirical convergence rate lower bound. This is sufficient since lower-dimensional distributions can be canonically embedded into higher dimensions without changing the value of $\GW$. As the lower bound holds for $n$ sufficiently large, we occasionally absorb terms of order $O(1/n)$, $O(1/\sqrt{n})$ and $O(\sqrt{\log(n)/n})$ into the $n^{-2/d}$ convergence rate. Consider the uniform distributions $\mu=\mathrm{Unif}\big(B_d(0,1)\big)$ and $\nu=\mathrm{Unif}\big(B_d(0,2)\big)$. 

\medskip

We start from the one-sample case and establish $\EE\big[\big|\GW(\mu,\nu)^2-\GW(\hat{\mu}_n,\nu)^2\big|\big]\geq n^{-2/d}$.~Theorem 9.21 of \cite{sturm2012space} implies that $T: x\mapsto 2x$ is an optimal Gromov-Monge map from $\mu$ and $\nu$, and thus $\GW(\mu,\nu)^2 = \int_{\cX\times\cX} \big|\|x-x'\|^2-\|2x-2x'\|^2\big|^2 d\mu\otimes\mu(x,x')$. Let $\pi_n\in\Pi(\hat{\mu}_n,\nu)$ be an optimal coupling for $\GW(\hat{\mu}_n,\nu)$ and notice that $\pi'_n = (\mathrm{id},\cdot/2)_\sharp\pi_n\in\Pi(\hat{\mu}_n,\mu)$ is optimal for $\GW(\hat{\mu}_n,\mu)$. By completing the square, we then have
\begin{align*}
\GW(\hat{\mu}_n,\nu)^2 
&= \int \big|\|y-y'\|^2-\|z-z'\|^2\big|^2 d\pi_n\otimes\pi_n(y,z,y',z') \\
& = \int \big|\|y-y'\|^2-\|2x-2x'\|^2\big|^2  d \pi'_n\otimes\pi'_n(y,x,y',x')\\
&= 4\GW(\hat{\mu}_n,\mu)^2 - 3\int \|y-y'\|^4 d \hat{\mu}_n\otimes \hat{\mu}_n(y,y') + 12\int \|x-x'\|^4 d \mu\otimes \mu(x,x').\numberthis \label{eq:GW_LB_1sample}
\end{align*}
Combining this with the above expression for $\GW(\mu,\nu)^2$, we obtain
\begin{align*}
    &\EE\big[\big|\GW(\mu,\nu)^2-\GW(\hat{\mu}_n,\nu)^2\big|\big ]\\
    &\geq 4\EE\big[\GW(\hat{\mu}_n,\mu)^2\big]+ 3\EE\left[\int \|x-x'\|^4 d \mu\otimes\mu(x,x') - \int \|y-y'\|^4 d \hat{\mu}_n\otimes \hat{\mu}_n(y,y')\right].
\end{align*}
Evidently, the second term decays as $n^{-1}$ since 
\[\EE\left[\int \|y-y'\|^4 d \hat{\mu}_n\otimes \hat{\mu}_n(y,y')\right] - \int \|x-x'\|^4 d \mu\otimes\mu(x,x')=\frac 1n\int \|x-x'\|^4 d \mu\otimes\mu(x,x').\]
For the first term, let $\tilde{\mu}_n$ be the centered version of $\hat{\mu}_n$ and invoke  \cref{lemma:equivalence_gw_w} to obtain
\begin{align*}
    \EE\big[\GW^2(\hat{\mu}_n,\mu)\big]&=\EE\big[\GW^2(\tilde{\mu}_n,\mu)\big]\\
    &\gtrsim  \lambda_{\mathrm{min}}(\bsigma_\mu)\inf_{\mathbf{U}\in O(d)} \mathsf{W}_2(\tilde{\mu}_n,\mathbf{U}_\sharp\mu)^2\\
    &=\lambda_{\mathrm{min}}(\bsigma_\mu)\,\EE\big[\sW_2(\tilde{\mu}_n,\mu)^2\big]\\
    &\geq\lambda_{\mathrm{min}}(\bsigma_\mu)\,\Big(\EE\big[\sW_1(\hat{\mu}_n,\mu) - \sW_1(\hat{\mu}_n,\tilde{\mu}_n)\big]\Big)^2,
\end{align*}
where the equality uses the rotational invariance of $\mu$, while the last step is by monotonicity of $p\mapsto \mathsf{W}_p$ and Jensen's inequality. Observe that $\EE[\mathsf{W}_1(\hat{\mu}_n,\tilde{\mu}_n)]\leq \EE[\|\bar{x}_n\|]\leq \sqrt{M_2(\mu)/n}$, where $\bar x_n \coloneqq  \int x \hat{\mu}_n(x)$ is the sample mean. Combining this with the fact that $\EE[\mathsf{W}_1(\hat{\mu}_n,\mu)]\gtrsim n^{-1/d}$ \cite{dudley1969speed}, produces the desired lower bound on the one-sample GW convergence rate.

\medskip

We proceed with the two-sample lower bound, which requires more work. Given the empirical measures $\hat{\mu}_n,\hat{\nu}_n$, define $\hat\mu'_n\coloneqq(\cdot/2)_\sharp\hat{\nu}_n$ and note that it forms an empirical distribution of $\mu$ that is independent of $\hat{\mu}_n$. Write $X'_1,\ldots,X'_n$ for the samples comprising $\hat\mu'_n$. Let $\pi_n\in\Pi(\hat\mu
_n,\hat{\nu}_n)$ be an optimal GW coupling for $\GW(\hat\mu
_n,\hat{\nu}_n)$ and set $\pi'_n \coloneqq (\mathrm{id},\cdot/2)_\sharp\pi_n\in\Pi(\hat{\mu}_n,\hat\mu'_n)$, which is optimal for $\GW(\hat{\mu}_n,\hat\mu'_n)$. Repeating the steps in \eqref{eq:GW_LB_1sample}, with $\hat{\nu}_n,\hat\mu'_n$ in place of $\nu,\mu$ yields
\[
\GW(\hat{\mu}_n,\hat{\nu}_n)^2=4\GW(\hat{\mu}_n,\hat\mu'_n)^2 - 3\int \|y-y'\|^4 d \hat{\mu}_n\otimes \hat{\mu}_n(y,y') + 12\int \|y-y'\|^4 d \hat\mu'_n\otimes \hat\mu'_n(y,y').
\]
Consequently, we represent the two-sample error as
\begin{align*}
    \GW(\hat{\mu}_n,\hat{\nu}_n)^2-\GW(\mu,\nu)^2 &= 4\GW(\hat{\mu}_n,\hat\mu'_n)^2 - 3\int \|y-y'\|^4 d \hat{\mu}_n\otimes \hat{\mu}_n(y,y')\\
    &+ 12\int \|y-y'\|^4 d \hat\mu'_n\otimes  \hat\mu'_n(y,y') - 9\int \|y-y'\|^4 d \mu\otimes \mu(y,y').\numberthis\label{eq:GW_LB_2sample_decomposition}
\end{align*}
As before, we have $\EE\left[\int \|y-y'\|^4 d \hat{\mu}_n\otimes\hat{\mu}_n(y,y')\right]=\frac{n-1}{n}\int \|y-y'\|^4 d \mu\otimes\mu(y,y')$ and similarly for $\EE\mspace{-2mu}\left[\int \|y-y'\|^4 d \hat\mu'_n\otimes\hat\mu'_n(y,y')\right]$,
and the problem reduces to lower bounding $\EE[\GW(\hat{\mu}_n,\hat\mu'_n)^2]$. We have the technical lemma below, which is proven in \cref{appen:proof:lemma:2sample_LB}

\begin{lemma}[Intermediate lower bound]\label{lemma:2sample_LB}
The following bound holds    
    \begin{equation}
    \EE[\GW(\hat{\mu}_n,\hat\mu'_n)^2]\gtrsim\EE\left[\lambda_{\mathrm{min}}(\bsigma_{\hat{\mu}_n})\EE\left[\inf_{\mathbf{U}\in O(d)} \mathsf{W}_1(\hat{\mu}_n,\mathbf{U}_\sharp\hat\mu'_n)^2\middle|X_1,\ldots,X_n\right]\right]-2\sqrt{\frac{M_2(\mu)}{n}}.\numberthis\label{eq:GW_LB_2sample}
    \end{equation}
\end{lemma}

To treat the inner (conditional) expectation on the RHS of \eqref{eq:GW_LB_2sample}, we make use of the next lemma; see \cref{appen:proof:lemma:technical_lemma} for the proof.

\begin{lemma}\label{lemma:technical_lemma}
For any $\mu,\nu\in\cP(\RR^d)$ with $\supp(\mu),\supp(\nu)\subset B_d(0,1)$, we have
    \begin{align*}
    \EE\left[\inf_{\mathbf{U}\in O(d)}\sW_1(\hat{\mu}_n,\mathbf{U}_\sharp\nu)\right]\geq \inf_{\mathbf{U}\in O(d)} \EE\big[\sW_1(\hat{\mu}_n,\mathbf{U}_\sharp\nu)\big] - C_{d}\sqrt{\frac{\log\,n}{n}},
    \end{align*}
    where $C_d$ depends only on the dimension $d$.
\end{lemma}

Applying the lemma, we obtain
\begin{align*}
    &\EE\left[\inf_{\mathbf{U}\in O(d)}\sW_1(\hat{\mu}_n,\mathbf{U}_\sharp\hat\mu'_n) \middle| X_1,\cdots,X_n\right]\\
    &\qquad\qquad\qquad\qquad\qquad\qquad\geq \inf_{\mathbf{U}\in O(d)} \EE\big[\sW_1(\hat{\mu}_n,\mathbf{U}_\sharp\hat\mu'_n) \big|X_1,\cdots,X_n\big] - C_{d}\sqrt{\frac{\log\,n}{n}}
\end{align*}
Note that for any $\mathbf{U}\in O(d)$, we have $\EE\big[\sW_1(\hat{\mu}_n,\mathbf{U}_\sharp\hat\mu'_n)\big|X_1,\cdots,X_n\big]\geq \sW_1(\mu,\mathbf{U}_\sharp\hat\mu'_n) = \sW_1(\mu,\hat\mu'_n)$, where the first inequality follows because $\EE[\sW_1(\hat{\mu}_n,\nu)]\geq \sW_1(\mu,\nu)$ for any $\mu,\nu$ (due to convexity), while the second equality uses the fact that $\sW_p(\mu,\nu)=\sW_p(f_\sharp\mu,f_\sharp\nu)$ for any isometry $f$ and the rotational invariance of $\mu$. Inserting this back into \eqref{eq:GW_LB_2sample}, yields 
\[
\EE[\GW(\hat{\mu}_n,\hat\mu'_n)^2]\gtrsim \EE\left[\lambda_{\mathrm{min}}(\bsigma_{\hat{\mu}_n})\inf_{\mathbf{U}\in O(d)} \mathsf{W}_2(\hat{\mu}_n,\mathbf{U}_\sharp\hat\mu'_n)^2\right]\geq \EE\big[\lambda_{\mathrm{min}}(\bsigma_{\hat{\mu}_n})\mathsf{W}_1(\hat{\mu}_n,\mu)^2\big].
\]

To lower bound the expectation on the RHS, recall that by Proposition 2.1 in \cite{dudley1969speed} (see also \cite[Proposition 6]{weed2019sharp}), for $n$ sufficiently large, we have $\sW_1(\alpha,\beta_n)\gtrsim_d n^{-1/d}$ for any distributions $\alpha,\beta_n\in\cP(\RR^d)$, such that $\alpha$ has a Lebesgue density and $\beta_n$ is supported on $n$ points. In particular, we conclude that there exists $n_0\in\NN$ and $c_d>0$, such that for all $n>n_0$, we have $\sW_1(\mu,\hat\mu'_n)\geq c_d n^{-1/d}$ a.s. Inserting this into the bound above gives 
\begin{equation}
\EE[\GW(\hat{\mu}_n,\hat\mu'_n)^2]\gtrsim_{\,d}\EE\big[\lambda_{\mathrm{min}}(\bsigma_{\hat{\mu}_n})\big]\cdot n^{-2/d},\label{eq:GW_LB_pre_final}
\end{equation}
and the problem reduces to lower bounding the expected smallest eigenvalue.

Write $\EE\big[\lambda_{\mathrm{min}}(\bsigma_{\hat{\mu}_n})\big]=\EE\big[\inf_{\|v\|=1}\hat{\mu}_n|v\cdot x|^2\big]$. We again control this quantity via bounds on an empirical processes indexed by the Donsker class $\{|v\cdot x|^2:\,\|v\|=1\}$. Specifically, there is an $n_1\in\NN$ that depends only on $d$, such that for any  $n>n_1$, we have $\EE\big[\sup_{\|v\|=1}\big|(\hat{\mu}_n-\mu)|v\cdot x|^2\big|\big]\leq\lambda_{\mathrm{min}}(\bsigma_\mu)/2$. Consequently
\begin{align*}
    \EE\left[\inf_{\|v\|=1}\hat{\mu}_n|v\cdot x|^2\right] &=  \EE\left[\inf_{\|v\|=1}\hat{\mu}_n|v\cdot x|^2 - \inf_{\|v\|=1}\EE\left[\hat{\mu}_n|v\cdot x|^2\right]\right] + \inf_{\|v\|=1}\EE\left[\hat{\mu}_n|v\cdot x|^2\right]\\
    &\geq \EE\left[\inf_{\|v\|=1}\hat{\mu}_n|v\cdot x|^2 - \EE\left[\hat{\mu}_n|v\cdot x|^2\right]\right] + \inf_{\|v\|=1}\mu|v\cdot x|^2\\
    & = \EE\left[\inf_{\|v\|=1}(\hat{\mu}_n- \mu)|v\cdot x|^2 \right] + \inf_{\|v\|=1}\mu|v\cdot x|\\
    &\geq \inf_{\|v\|=1}\mu|v\cdot x|^2 - \EE\left[\sup_{\|v\|=1}|(\hat{\mu}_n-\mu)|v\cdot x|^2|\right]\\
    &\geq \frac{\lambda_{\mathrm{min}}(\bsigma_\mu)}{2}.
\end{align*}

Inserting this back into \eqref{eq:GW_LB_pre_final} and recalling the decomposition of the empirical estimation error from \eqref{eq:GW_LB_2sample_decomposition} concludes the proof of the two-sample lower bound. \qed

\begin{remark}[Wasserstein Procrustes empirical convergence rate]
    Our two-sample analysis above essentially establishes an $n^{-1/d}$ lower bound on the Wasserstein Procrustes empirical convergence rate, whenever $d\geq 3$. Since the Procrustes is trivially upper bounded by standard $\sW_2$ and is a pseudometric, it inherits the $n^{-1/d}$ upper bound on the rate from it as well. Together, these show that the $n^{-1/d}$ empirical convergence rate is sharp in general. Our argument is readily adjusted to cover both the one- and two-sample settings and can be extended to any order $p\geq 1$.
\end{remark}
}

\section{Outlook and Concluding Remarks}\label{sec:summary}

 {This paper established a dual formulation for both the standard $(2,2)$-GW distance and its entropically regularized version, between distributions supported on Euclidean spaces of different dimensions $d_x$ and $d_y$. The dual forms represented GW and EGW as infima of a class of OT and EOT problems, respectively, indexed by a $d_x\times d_y$ auxiliary matrix with bounded entries, which specified the associated cost function. This connection to the well-understood (standard and entropic) OT problem enabled lifting analysis techniques from statistical OT to establish, for the first time, sharp empirical convergences rates for GW and EGW. 
The derived two-sample rates are $n^{-2/((d_x\wedge d_y)\vee 4)}$ (up to a log factor when $d_x\wedge d_y=4$) for GW and $n^{-1/2}$ for EGW; in one-sample setting, when, say, $\nu$ is not estimated, $d_y$ is omitted form the GW rate. The GW result accounts for compactly supported distributions, and provides matching upper and lower rate bound. For EGW, our analysis allows for unbounded domains subject to a 4-sub-Weibull condition. These results are in line with the empirical convergence rates of OT \cite{manole2021sharp,hundrieser2022empirical} and EOT \cite{mena2019statistical,groppe2023lower}.}

We have also explored stability and continuity of the EGW problem in the entropic regularization parameter $\eps$. We provided an $O\big(\eps\log(1/\eps)\big)$ approximation bound on the GW cost and a continuity result for the optimal couplings in the weak topology. Lastly, we reexamined the open problem of the one-dimensional GW distance between discrete distributions on $n$ points. Leveraging our duality theory, we shed new light on the peculiar example from \cite{beinert2022assignment}, that showed that the identity and anti-identity permutations are not necessarily optimal. Specifically, the dual form allows representing the GW distance as a sum of concave and convex functions, illuminating that, in certain regimes, the optimum is not necessarily attained on the boundary.

 {Future research directions stemming from this work are aplenty. Due to the central role of duality for statistical and algorithmic advancements, a first key objective is to extend our duality theory beyond the $(2,2)$-cost and to non-Euclidean mm spaces. While our techniques are rather specialized for the $(2,2)$-cost and treating arbitrary $(p,q)$ values may require new ideas, we comment here on one relatively direct extension. Consider the GW distance of order $(p,q)=(2,2k)$, for some $k\in\NN$, between distributions $(\mu,\nu)\in\cP_{4k}(\RR^{d_x})\times \cP_{4k}(\RR^{d_y})$ (in fact, we can treat any even $p$ parameter as well, but restrict to $p=2$ for simplicity). Following a decomposition along the lines of \eqref{eq:proof_of_EGW_decomposition}, in \cref{appen:generalized_dual} we show that 
\begin{align*}\numberthis\label{eq:gen_dual}
&\GW_{2,2k}(\mu,\nu)^2=4\sup_{a\in \RR^{\ell}}\inf_{b\in \RR^{m-\ell}}\Bigg\{ -\|a\|^2+\|b\|^2 \\
&\quad+ \inf_{\pi\in\Pi(\mu,\nu)} \int\left( -\|x\|^{2k}\|y\|^{2k} + \sum_{i=1}^{\ell}a_i g_i(x,y) - \sum_{i=\ell+1}^{m}b_{i-\ell} g_i(x,y) \right)d\pi(x,y)\Bigg\},
\end{align*}
where $g_1,\ldots,g_{m}$ are polynomials of degree at most $4k$, $m$ corresponds to the number of polynomials emerging from the quadratic expansion of the $(2,2k)$-cost, and $\ell\leq m$ is determined by a certain diagonalization argument (see \cref{appen:generalized_dual} for the specifics). One may further show that $\int g_i d\pi$ are uniformly bounded for all $i=1,\ldots,m$ and $\pi\in\Pi(\mu,\nu)$, and so we may restrict optimization over $a,b$ to bounded domains. In the appendix, we also show how the above dual reduces to the one from \cref{cor:gw_duality} once we set $k=1$ and assume that $\mu,\nu$ are centered. A similar representation holds for the $(2,2k)$-EGW variant, but with the entropic penalty $\eps \mathsf{D}_\mathsf{KL}(\pi\|\mu\otimes \nu)$ added to the transportation cost in the second line above.

Notice now that the inner optimization over $\pi$ specifies an OT problem with cost 
\[c_{a,b}:(x,y)\mapsto -\|x\|^{2k}\|y\|^{2k} + \sum_{i=1}^{\ell}a_i g_i(x,y) - \sum_{i=\ell+1}^{m}b_{i-\ell}\,g_i(x,y),\]
which is smooth (indeed, a polynomial) but not necessarily concave in $x$ or $y$. For the standard $(2,2k)$-GW distance between compactly supported distributions, an argument similar to the proof of \cref{thm:GW_sample_complex}, would result in a two-sample convergence rate of $O\big(n^{-1/(d_x\wedge d_y)}\big)$. This rate stems from the fact that the corresponding dual potentials are Lipschitz continuous, but it is unclear whether they posses further convexity/concavity properties. For the EGW case, under proper tail conditions (say, $4k$-sub-Weibull), smoothness of the cost would allow to reproduce the current derivation of \cref{thm:egw_sample_complex} and arrive at the parametric convergence rate. In sum, while a duality theory for general $(p,q)$ remains an open question, our results for the quadratic GW and EGW distances can be extended to cover any even $q$ value.

As mentioned above, extending our duality to cover non-Euclidean mm spaces is of~great interest, as this would enable accounting for graph and manifold data modalities. We also believe that our dual can be used to derive new and efficient algorithms for computing the GW and EGW distances. Lastly, we mention the avenue of generalizing the GW empirical convergence results to distributions with unbounded supports. Identifying sufficient conditions for deriving explicit rates seems non-trivial and may require assumptions along the lines of Theorem 13 from \cite{manole2021sharp}, where empirical convergence of OT on unbounded domains was treated.
}

\bibliographystyle{amsalpha}
\bibliography{references}

\newcommand{\etalchar}[1]{$^{#1}$}
\providecommand{\bysame}{\leavevmode\hbox to3em{\hrulefill}\thinspace}
\providecommand{\MR}{\relax\ifhmode\unskip\space\fi MR }
\providecommand{\MRhref}[2]{%
  \href{http://www.ams.org/mathscinet-getitem?mr=#1}{#2}
}
\providecommand{\href}[2]{#2}
\begin{thebibliography}{dBGSLNW22b}

\bibitem[AMJ18]{alvarez2018gromov}
David Alvarez-Melis and Tommi Jaakkola, \emph{Gromov-wasserstein alignment of
  word embedding spaces}, Proceedings of the 2018 Conference on Empirical
  Methods in Natural Language Processing, Association for Computational
  Linguistics, 2018, pp.~1881--1890.

\bibitem[BAMKJ19]{bunne2019learning}
Charlotte Bunne, David Alvarez-Melis, Andreas Krause, and Stefanie Jegelka,
  \emph{Learning generative models across incomparable spaces}, 2019.

\bibitem[BB00]{benamou2000}
Jean-David Benamou and Yann Brenier, \emph{A computational fluid mechanics
  solution to the {M}onge-{K}antorovich mass transfer problem}, Numerische
  Mathematik \textbf{84} (2000), no.~3, 375--393.

\bibitem[BCM{\etalchar{+}}20]{blumberg2020mrec}
Andrew~J Blumberg, Mathieu Carriere, Michael~A Mandell, Raul Rabadan, and
  Soledad Villar, \emph{{MREC: a fast and versatile framework for aligning and
  matching point clouds with applications to single cell molecular data}},
  arXiv preprint arXiv:2001.01666 (2020).

\bibitem[BCP19]{bigot2019central}
J{\'e}r{\'e}mie Bigot, Elsa Cazelles, and Nicolas Papadakis, \emph{Central
  limit theorems for entropy-regularized optimal transport on finite spaces and
  statistical applications}, Electronic Journal of Statistics \textbf{13}
  (2019), no.~2, 5120--5150.

\bibitem[BHS22]{beinert2022assignment}
Robert Beinert, Cosmas Heiss, and Gabriele Steidl, \emph{On assignment problems
  related to {G}romov-{W}asserstein distances on the real line}, arXiv preprint
  arXiv:2205.09006 (2022).

\bibitem[CDPS17]{carlier2017convergence}
Guillaume Carlier, Vincent Duval, Gabriel Peyr{\'e}, and Bernhard Schmitzer,
  \emph{Convergence of entropic schemes for optimal transport and gradient
  flows}, SIAM Journal on Mathematical Analysis \textbf{49} (2017), no.~2,
  1385--1418.

\bibitem[CGP16]{chen2016relation}
Yongxin Chen, Tryphon~T Georgiou, and Michele Pavon, \emph{On the relation
  between optimal transport and {S}chr{\"o}dinger bridges: A stochastic control
  viewpoint}, Journal of Optimization Theory and Applications \textbf{169}
  (2016), no.~2, 671--691.

\bibitem[Com05]{Commander2005}
Clayton~W. Commander, \emph{A survey of the quadratic assignment problem, with
  applications}, Morehead Electronic Journal of Applicable Mathematics
  \textbf{4} (2005), MATH--2005--01.

\bibitem[CRL{\etalchar{+}}20]{chizat2020faster}
Lenaic Chizat, Pierre Roussillon, Flavien L{\'e}ger, Fran{\c{c}}ois-Xavier
  Vialard, and Gabriel Peyr{\'e}, \emph{Faster {W}asserstein distance
  estimation with the {S}inkhorn divergence}, Advances in Neural Information
  Processing Systems \textbf{33} (2020), 2257--2269.

\bibitem[CS96]{Constantine1996AMF}
Gregory~M. Constantine and Thomas~H. Savits, \emph{A multivariate {F}aa di
  {B}runo formula with applications}, Transactions of the American Mathematical
  Society \textbf{348} (1996), 503--520.

\bibitem[CT21]{conforti2021formula}
Giovanni Conforti and Luca Tamanini, \emph{A formula for the time derivative of
  the entropic cost and applications}, Journal of Functional Analysis
  \textbf{280} (2021), no.~11, 108964.

\bibitem[dBGSLNW22a]{del2022improved}
Eustasio del Barrio, Alberto Gonzalez-Sanz, Jean-Michel Loubes, and Jonathan
  Niles-Weed, \emph{An improved central limit theorem and fast convergence
  rates for entropic transportation costs}, arXiv preprint arXiv:2204.09105
  (2022).

\bibitem[dBGSLNW22b]{delbarrio22EOT}
Eustasio del Barrio, Alberto Gonz\'{a}lez-Sanz, Jean-Michel Loubes, and
  Jonathan Niles-Weed, \emph{An improved central limit theorem and fast
  convergence rates for entropic transportation costs}, arXiv preprint
  arXiv:2204.09105 (2022).

\bibitem[dBL19]{del2019central}
Eustasio del Barrio and Jean-Michel Loubes, \emph{Central limit theorems for
  empirical transportation cost in general dimension}, The Annals of
  Probability \textbf{47} (2019), no.~2, 926--951.

\bibitem[DDS22]{delon2022gromov}
Julie Delon, Agnes Desolneux, and Antoine Salmona, \emph{Gromov-{W}asserstein
  distances between {G}aussian distributions}, Journal of Applied Probability
  (2022), 1--21.

\bibitem[DLV22]{dumont2022existence}
Theo Dumont, Th{\'e}o Lacombe, and Fran{\c{c}}ois-Xavier Vialard, \emph{On the
  existence of {M}onge maps for the {G}romov-{W}asserstein distance}, arXiv
  preprint arXiv:2210.11945 (2022).

\bibitem[DM12]{dal2012introduction}
Gianni Dal~Maso, \emph{An introduction to {$\Gamma$}-convergence}, vol.~8,
  Springer Science \& Business Media, 2012.

\bibitem[DSS13]{dereich2013constructive}
Steffen Dereich, Michael Scheutzow, and Reik Schottstedt, \emph{Constructive
  quantization: Approximation by empirical measures}, Annales de l'IHP
  Probabilit{\'e}s et statistiques, vol.~49, 2013, pp.~1183--1203.

\bibitem[DSS{\etalchar{+}}22]{demetci2020gromov}
Pinar Demetci, Rebecca Santorella, Bj{\"o}rn Sandstede, William~Stafford Noble,
  and Ritambhara Singh, \emph{{SCOT: single-cell multi-omics alignment with
  optimal transport}}, Journal of Computational Biology \textbf{29} (2022),
  no.~1, 3--18.

\bibitem[Dud69]{dudley1969speed}
Richard~Mansfield Dudley, \emph{The speed of mean {G}livenko-{C}antelli
  convergence}, The Annals of Mathematical Statistics \textbf{40} (1969),
  no.~1, 40--50.

\bibitem[FG15]{fournier2015rate}
Nicolas Fournier and Arnaud Guillin, \emph{On the rate of convergence in
  wasserstein distance of the empirical measure}, Probability theory and
  related fields \textbf{162} (2015), no.~3-4, 707--738.

\bibitem[GCB{\etalchar{+}}19]{genevay2019sample}
Aude Genevay, L{\'e}naic Chizat, Francis Bach, Marco Cuturi, and Gabriel
  Peyr{\'e}, \emph{Sample complexity of {S}inkhorn divergences}, The 22nd
  International Conference on Artificial Intelligence and Statistics, PMLR,
  2019, pp.~1574--1583.

\bibitem[GG20]{goldfeld2020gaussian}
Ziv Goldfeld and Kristjan Greenewald, \emph{Gaussian-smoothed optimal
  transport: Metric structure and statistical efficiency}, International
  Conference on Artificial Intelligence and Statistics, PMLR, 2020,
  pp.~3327--3337.

\bibitem[GH23]{groppe2023lower}
Michel Groppe and Shayan Hundrieser, \emph{Lower complexity adaptation for
  empirical entropic optimal transport}, arXiv preprint arXiv:2306.13580
  (2023).

\bibitem[GJB19]{grave2019unsupervised}
Edouard Grave, Armand Joulin, and Quentin Berthet, \emph{Unsupervised alignment
  of embeddings with {W}asserstein procrustes}, The 22nd International
  Conference on Artificial Intelligence and Statistics, PMLR, 2019,
  pp.~1880--1890.

\bibitem[GKNR22]{goldfeld2022limitsmooth}
Ziv Goldfeld, Kengo Kato, Sloan Nietert, and Gabriel Rioux, \emph{Limit
  distribution theory for smooth $p$-{W}asserstein distances}, arXiv preprint
  arXiv:2203.00159 (2022).

\bibitem[GKRS22a]{goldfeld2022limit}
Ziv Goldfeld, Kengo Kato, Gabriel Rioux, and Ritwik Sadhu, \emph{Limit theorems
  for entropic optimal transport maps and the {S}inkhorn divergence}, arXiv
  preprint arXiv:2207.08683 (2022).

\bibitem[GKRS22b]{goldfeld2022statistical}
\bysame, \emph{Statistical inference with regularized optimal transport}, arXiv
  preprint arXiv:2205.04283 (2022).

\bibitem[GLR17]{gentil2017analogy}
Ivan Gentil, Christian L{\'e}onard, and Luigia Ripani, \emph{About the analogy
  between optimal transport and minimal entropy}, Annales de la Facult{\'e} des
  sciences de Toulouse: Math{\'e}matiques, vol.~26, 2017, pp.~569--600.

\bibitem[GN16]{gine2016mathematical}
E.~Gin{\'e} and R.~Nickl, \emph{Mathematical foundations of
  infinite-dimensional statistical models}, Cambridge Series in Statistical and
  Probabilistic Mathematics, Cambridge University Press, 2016.

\bibitem[Goo91]{goodall1991procrustes}
Colin Goodall, \emph{Procrustes methods in the statistical analysis of shape},
  Journal of the Royal Statistical Society: Series B (Methodological)
  \textbf{53} (1991), no.~2, 285--321.

\bibitem[GS12]{guntuboyina2012l1}
Adityanand Guntuboyina and Bodhisattva Sen, \emph{$l_1$ covering numbers for
  uniformly bounded convex functions}, Conference on Learning Theory, JMLR
  Workshop and Conference Proceedings, 2012, pp.~12--1.

\bibitem[GSH23]{gonzalez2023weak}
Alberto Gonz{\'a}lez-Sanz and Shayan Hundrieser, \emph{Weak limits for
  empirical entropic optimal transport: Beyond smooth costs}, arXiv preprint
  arXiv:2305.09745 (2023).

\bibitem[GSLNW22]{gonzalez2022weak}
Alberto Gonz\'{a}lez-Sanz, Jean-Michel Loubes, and Jonathan Niles-Weed,
  \emph{Weak limits of entropy regularized optimal transport; potentials, plans
  and divergences}, arXiv preprint: arXiv 2207.07427 (2022).

\bibitem[GT20]{gigli2020benamou}
Nicola Gigli and Luca Tamanini, \emph{Benamou--brenier and duality formulas for
  the entropic cost on $\mathsf{RCD}^*(k,n)$ spaces}, Probability Theory and
  Related Fields \textbf{176} (2020), no.~1, 1--34.

\bibitem[GX21]{gunsilius2021matching}
Florian Gunsilius and Yuliang Xu, \emph{Matching for causal effects via
  multimarginal unbalanced optimal transport}, arXiv preprint arXiv:2112.04398
  (2021), updated on July 9, 2022.

\bibitem[HKSM22]{hundrieser2022}
Shayan Hundrieser, Marcel Klatt, Thomas Staudt, and Axel Munk, \emph{A unifying
  approach to distributional limits for empirical optimal transport}, arXiv
  preprint: arXiv 2202.12790 (2022).

\bibitem[HSM22]{hundrieser2022empirical}
Shayan Hundrieser, Thomas Staudt, and Axel Munk, \emph{Empirical optimal
  transport between different measures adapts to lower complexity}, arXiv
  preprint arXiv:2202.10434 (2022).

\bibitem[KDO23]{koehl2023computing}
Patrice Koehl, Marc Delarue, and Henri Orland, \emph{{Computing the
  Gromov-Wasserstein distance between two surface meshes using optimal
  transport}}, Algorithms \textbf{16} (2023), no.~3, 131.

\bibitem[KTM20]{klatt2020empirical}
Marcel Klatt, Carla Tameling, and Axel Munk, \emph{Empirical regularized
  optimal transport: Statistical theory and applications}, SIAM Journal on
  Mathematics of Data Science \textbf{2} (2020), no.~2, 419--443.

\bibitem[LLN{\etalchar{+}}22]{le2022entropic}
Khang Le, Dung~Q Le, Huy Nguyen, Dat Do, Tung Pham, and Nhat Ho, \emph{Entropic
  {G}romov-{W}asserstein between {G}aussian distributions}, International
  Conference on Machine Learning, PMLR, 2022, pp.~12164--12203.

\bibitem[MBNWW21]{manole2021plugin}
Tudor Manole, Sivaraman Balakrishnan, Jonathan Niles-Weed, and Larry Wasserman,
  \emph{Plugin estimation of smooth optimal transport maps}, arXiv preprint
  arXiv:2107.12364 (2021).

\bibitem[M{\'e}m09]{memoli2009spectral}
Facundo M{\'e}moli, \emph{Spectral {G}romov-{W}asserstein distances for shape
  matching}, 2009 IEEE 12th International Conference on Computer Vision
  Workshops, ICCV Workshops, IEEE, 2009, pp.~256--263.

\bibitem[M{\'e}m11]{Memoli11}
\bysame, \emph{Gromov-{W}asserstein distances and the metric approach to object
  matching}, Found. Comput. Math. \textbf{11} (2011), no.~4, 417--487.

\bibitem[ML18]{maron2018probably}
Haggai Maron and Yaron Lipman, \emph{(probably) concave graph matching},
  Advances in Neural Information Processing Systems \textbf{31} (2018).

\bibitem[MNW19]{mena2019statistical}
Gonzalo Mena and Jonathan Niles-Weed, \emph{Statistical bounds for entropic
  optimal transport: sample complexity and the central limit theorem}, Advances
  in Neural Information Processing Systems \textbf{32} (2019).

\bibitem[MNW21]{manole2021sharp}
Tudor Manole and Jonathan Niles-Weed, \emph{Sharp convergence rates for
  empirical optimal transport with smooth costs}, arXiv preprint
  arXiv:2106.13181 (2021).

\bibitem[NW21]{nutz2021entropic}
Marcel Nutz and Johannes Wiesel, \emph{Entropic optimal transport: convergence
  of potentials}, Probability Theory and Related Fields (2021), 1--24.

\bibitem[NWR22]{niles2022estimation}
Jonathan Niles-Weed and Philippe Rigollet, \emph{Estimation of wasserstein
  distances in the spiked transport model}, Bernoulli \textbf{28} (2022),
  no.~4, 2663--2688.

\bibitem[PC{\etalchar{+}}19]{peyre2019computational}
Gabriel Peyr{\'e}, Marco Cuturi, et~al., \emph{Computational optimal transport:
  with applications to data science}, Foundations and Trends{\textregistered}
  in Machine Learning \textbf{11} (2019), no.~5-6, 355--607.

\bibitem[PCS16]{peyre2016gromov}
Gabriel Peyr{\'e}, Marco Cuturi, and Justin Solomon,
  \emph{{G}romov-{W}asserstein averaging of kernel and distance matrices},
  International Conference on Machine Learning, PMLR, 2016, pp.~2664--2672.

\bibitem[R\"04]{romisch2004}
Werner R\"{o}misch, \emph{Delta method, infinite dimensional}, Encyclopedia of
  Statistical Sciences, Wiley, 2004.

\bibitem[RGK23]{rioux2023entropic}
Gabriel Rioux, Ziv Goldfeld, and Kengo Kato, \emph{Entropic gromov-wasserstein
  distances: Stability, algorithms, and distributional limits}, arXiv preprint
  arXiv:2306.00182 (2023).

\bibitem[RS22]{rigollet2022sample}
Philippe Rigollet and Austin~J Stromme, \emph{On the sample complexity of
  entropic optimal transport}, arXiv preprint arXiv:2206.13472 (2022).

\bibitem[San15]{santambrogio15}
Filippo Santambrogio, \emph{Optimal transport for applied mathematicians},
  Birk{\"a}user, NY \textbf{55} (2015), no.~58-63, 94.

\bibitem[Sch66]{schonemann1966generalized}
Peter~H Sch{\"o}nemann, \emph{A generalized solution of the orthogonal
  procrustes problem}, Psychometrika \textbf{31} (1966), no.~1, 1--10.

\bibitem[Sha90]{shapiro1990}
Alexander Shapiro, \emph{On concepts of directional differentiability}, Journal
  of Optimization Theory and Applications \textbf{66} (1990), 477--487.

\bibitem[SM18]{Sommerfeld2018}
Max Sommerfeld and Axel Munk, \emph{Inference for empirical {W}asserstein
  distances on finite spaces}, Journal of Royal Statistical Society: Series B
  (Statistical Methodology) \textbf{80} (2018), 219--238.

\bibitem[SPC22]{scetbon2022linear}
Meyer Scetbon, Gabriel Peyr{\'e}, and Marco Cuturi, \emph{{Linear-time Gromov-
  Wasserstein distances using low rank couplings and costs}}, International
  Conference on Machine Learning, PMLR, 2022, pp.~19347--19365.

\bibitem[SPKS16]{solomon2016entropic}
Justin Solomon, Gabriel Peyr{\'e}, Vladimir~G Kim, and Suvrit Sra,
  \emph{Entropic metric alignment for correspondence problems}, ACM
  Transactions on Graphics (ToG) \textbf{35} (2016), no.~4, 1--13.

\bibitem[Stu12]{sturm2012space}
Karl-Theodor Sturm, \emph{The space of spaces: curvature bounds and gradient
  flows on the space of metric measure spaces}, arXiv preprint arXiv:1208.0434
  (2012).

\bibitem[SVP21]{sejourne2021unbalanced}
Thibault S{\'e}journ{\'e}, Fran{\c{c}}ois-Xavier Vialard, and Gabriel
  Peyr{\'e}, \emph{The unbalanced gromov wasserstein distance: Conic
  formulation and relaxation}, Advances in Neural Information Processing
  Systems \textbf{34} (2021), 8766--8779.

\bibitem[TSM19]{Tameling2019}
Carla Tameling, Max Sommerfeld, and Axel Munk, \emph{Empirical optimal
  transport on countable metric spaces: Distributional limits and statistical
  applications}, The Annals of Applied Probability \textbf{29} (2019),
  2744--2781.

\bibitem[vdVW96]{van1996weak}
Aad~W. van~der Vaart and Jon~A. Wellner, \emph{Weak convergence and empirical
  processes: with applications to statistics}, Springer Science \& Business
  Media, 1996.

\bibitem[VFT{\etalchar{+}}20]{vayer2020sliced}
Titouan Vayer, R{\'e}mi Flamary, Romain Tavenard, Laetitia Chapel, and Nicolas
  Courty, \emph{{S}liced {G}romov-{W}asserstein}, 2020.

\bibitem[Vil09]{villani2009optimal}
C{\'e}dric Villani, \emph{Optimal transport: old and new}, vol. 338, Springer,
  2009.

\bibitem[WB19]{weed2019sharp}
Jonathan Weed and Francis Bach, \emph{Sharp asymptotic and finite-sample rates
  of convergence of empirical measures in wasserstein distance}, Bernoulli
  \textbf{25} (2019), no.~4 A, 2620--2648.

\bibitem[XLC19]{xu2019scalable}
Hongteng Xu, Dixin Luo, and Lawrence Carin, \emph{Scalable
  {G}romov-{W}asserstein learning for graph partitioning and matching}, 2019.

\bibitem[XLZD19]{xu2019gromov}
Hongteng Xu, Dixin Luo, Hongyuan Zha, and Lawrence~Carin Duke,
  \emph{{G}romov-{W}asserstein learning for graph matching and node embedding},
  International conference on machine learning, PMLR, 2019, pp.~6932--6941.

\bibitem[YLW{\etalchar{+}}18]{yan2018semi}
Yuguang Yan, Wen Li, Hanrui Wu, Huaqing Min, Mingkui Tan, and Qingyao Wu,
  \emph{Semi-supervised optimal transport for heterogeneous domain
  adaptation.}, IJCAI, vol.~7, 2018, pp.~2969--2975.

\bibitem[ZMGS22]{zhang2022cycle}
Zhengxin Zhang, Youssef Mroueh, Ziv Goldfeld, and Bharath Sriperumbudur,
  \emph{Cycle consistent probability divergences across different spaces},
  International Conference on Artificial Intelligence and Statistics, PMLR,
  2022, pp.~7257--7285.

\end{thebibliography}

\newpage
\appendix

\section{Proof of Proposition \ref{prop:cont_of_cost_in_eps}} \label{proof:prop:cont_of_cost_in_eps}

W.l.o.g, suppose that both $\mu,\nu$ are centered. We follow the block approximation idea from \cite{genevay2019sample}. Let $\pi_0$ be the optimal coupling for the original GW problem and define the block approximation $\pi^\ell$ of side length $\ell\in(0,1]$ as follows. For $d\geq1$ and any $k\in\ZZ^d$, consider the Euclidean cube $Q_k^\ell\coloneqq \prod_{m=1}^d\big[k_m\ell,(k_m+1)\ell\big) \subset\RR^d$, and for each $(i,j)\in\ZZ^{d_x}\times\ZZ^{d_y}$, set $Q_{ij}^\ell = Q_i^\ell\times Q_j^\ell$. Define
\begin{align*}
    \pi^\ell|_{Q_{ij}^\ell}\coloneqq  \frac{\pi_0(Q_{ij}^\ell)}{\mu(Q_{i}^\ell)\nu(Q_{j}^\ell)} \mu|_{Q_{i}^\ell}\otimes \nu|_{Q_{j}^\ell},
\end{align*}
where the restriction of a measure $\pi$ to a measurable set $A$ is defined as $\pi|_A(\cdot)\coloneqq \pi(A\cap \cdot)$. Taking $\pi^\ell=\sum_{i=1}^{d_x}\sum_{j=1}^{d_y}\pi^\ell|_{Q_{ij}^\ell}$, it is straightforward to verify that $\pi^\ell\in \Pi(\mu,\nu)$.

To simplify notation, define $\sG(\pi)\coloneqq \|\Delta\|^2_{L^2(\pi\otimes\pi)}$, and first note that 
\begin{align*}
0\leq \EGW_\eps(\mu,\nu)-\GW(\mu,\nu)^2
\leq \sG(\pi^\ell)-\sG(\pi_0) + \eps \dkl(\pi^\ell\|\mu\otimes\nu).
\end{align*}
We start by bounding $\sG(\pi^\ell)-\sG(\pi_0)$. By cancelling out terms that depend only on $\mu,\nu$ and appear in both expressions (and using the fact that the marginals are centered), we have
\begin{equation}
    \begin{split}
        \sG(\pi^\ell)-\sG(\pi_0) &\lesssim \Big|\int \|x\|^2\|y\|^2  d (\pi^\ell-\pi_0)(x,y)\Big| \\
        &\qquad\qquad + \sum_{\substack{1\leq m\leq d_x\\1\leq p \leq d_y}}\bigg|\left( \int x_m y_p d \pi^\ell(x,y) \right)^2 - \left( \int x_m y_p d \pi_0(x,y)\right)^2\bigg|.
    \end{split}
    \label{eq:approx_error_decomp}
\end{equation}
For the first term above, consider
\begin{align*}
    \bigg|\int \|x\|^2\|y\|^2 & d (\pi^\ell-\pi_0)(x,y)\bigg|\\
    &\leq \sum_{(i,j)\in \ZZ^{d_x+d_y}} \pi_0\big(Q_{ij}^\ell\big) \Bigg(\sup_{(x,y)\in Q_{ij}^\ell} \|x\|^2\|y\|^2 - \inf_{(x,y)\in Q_{ij}^\ell} \|x\|^2\|y\|^2 \Bigg)\\
    &\lesssim \sum_{(i,j)\in \ZZ^{d_x+d_y}} \pi_0\big(Q_{ij}^\ell\big) \sup_{(x,y)\in Q_{ij}^\ell}   \|x\|\|y\|\Bigg( \sup_{(x,y)\in Q_{ij}^\ell} \|x\|\sqrt{d_y}\ell + \|y\|\sqrt{d_x}\ell \Bigg) \\
    &\lesssim \sum_{(i,j)\in \ZZ^{d_x+d_y}} \ell\big(\sqrt{d_x}+\sqrt{d_y}\big)\pi_0\big(Q_{ij}^\ell\big)\cdot \sup_{(x,y)\in Q_{ij}^\ell}  \big (\|x\|^2\|y\| + \|x\|\|y\|^2\big) \\
    &\lesssim \mspace{-3mu} \sum_{(i,j)\in \ZZ^{d_x+d_y}}\mspace{-3mu} \ell\big(\sqrt{d_x}\mspace{-3mu}+\mspace{-3mu}\sqrt{d_y})\pi_0(Q_{ij}^\ell\big)\cdot\sup_{(x,y)\in Q_{ij}^\ell} \big(\|x\|^4\mspace{-3mu}+\mspace{-3mu}\|x\|^2\mspace{-3mu}+\mspace{-3mu}\|y\|^4\mspace{-3mu}+\mspace{-3mu}\|y\|^2\big)
\end{align*}
 {where the third inequality follows by the mean value theorem for the function $\|x\|^2\|y\|^2$ applied to any two points in $Q_{ij}^\ell$.} By integrability of $(\|x\|^4+\|x\|^2+\|y\|^4+\|y\|^2)$ w.r.t. $\pi_0$, the RHS above can be bounded in terms of the 4th moments of $\mu,\nu$, as quantified by the following lemma. 

\begin{lemma}\label{lemma:block_approximate}
For any $m\in\NN$, $\mu\in\cP(\RR^d)$, and partition $\{Q_k^\ell\}_{k\in\ZZ^d}$, we have
\begin{align*}
    \left|\int \|x\|^m d \mu(x) - \sum_{k\in\ZZ^{d}} \sup_{x\in Q_k^\ell}\|x\|^m\mu\big(Q_k^\ell\big)\right|\leq \sum_{p=0}^{m-1}\frac{m!M_p(\mu)(\sqrt{d}\ell)^{m-p}}{p!}  
\end{align*}
\end{lemma}
\begin{proof}
We prove by induction. For $m=1$, we obtain
\[
\left|\int \|x\| d \mu(x) - \sum_{k\in\ZZ^{d}} \sup_{x\in Q_k^\ell}\|x\|\mu\big(Q_k^\ell\big)\right|\leq \sum_{k\in\ZZ^{d}} \left|\sup_{x\in Q_i^\ell}\|x\|-\inf_{x\in Q_k^\ell}\|x\|\right|\mu\big(Q_k^\ell\big)\leq \ell \sqrt{d}.\]
For $m>1$, consider
\begin{align*}
    \Bigg|\int \|x\|^m d \mu(x) - &\sum_{k\in\ZZ^{d}} \sup_{x\in Q_k^\ell}\|x\|^m\mu\big(Q_k^\ell\big)\Bigg|\\
    & \leq \sum_{k\in\ZZ^{d}} \bigg|\sup_{x\in Q_k^\ell}\|x\|^m-\inf_{x\in Q_k^\ell}\|x\|^m\bigg|\mu\big(Q_k^\ell\big)\\
    &\leq m \ell \sqrt{d}\sum_{k\in\ZZ^{d}}\sup_{x\in Q_k^\ell} \|x\|^{m-1}\mu\big(Q_k^\ell\big)\\
    &\leq m \ell \sqrt{d} \bigg(\bigg|M_{m-1}(\mu)- \sum_{k\in\ZZ^{d}}\sup_{x\in Q_k^\ell} \|x\|^{m-1}\mu\big(Q_k^\ell\big) \bigg|+M_{m-1}(\mu)\bigg).
\end{align*}
Hence, by induction we have
\begin{align*}
    \left|\int \|x\|^m d \mu(x) - \sum_{i\in\ZZ^{d_x}} \sup_{x\in Q_i^\ell}\|x\|^m\mu\big(Q_i^\ell\big)\right| \leq \sum_{p=0}^{m-1}\frac{m!M_p(\mu)(\sqrt{d}\ell)^{m-p}}{p!}.
\end{align*}
\end{proof}

Invoking the lemma, yields
\[
\left|\int \|x\|^2\|y\|^2  d (\pi^\ell-\pi_0)(x.y)\right| \lesssim \ell\big(\sqrt{d_x}+\sqrt{d_y}\big)^5 \big(1+M_4(\mu)+M_4(\nu)\big).
\]

A similar approach can be applied to the second term in \eqref{eq:approx_error_decomp}. Namely, we write it as
\[
\begin{split}
    &\bigg|\mspace{-3mu}\left( \int\mspace{-4.5mu} x_m y_p d \pi^\ell(x,y)\mspace{-3mu} \right)^{\mspace{-2mu}2} \mspace{-3mu}- \left( \int \mspace{-4.5mu}x_m y_p d \pi_0(x,y)\mspace{-3mu}\right)^{\mspace{-3mu}2}\bigg|\mspace{-1.5mu}\\
    &\qquad\qquad\qquad\qquad\qquad\qquad\qquad=\mspace{-1.5mu}\bigg|\mspace{-2.5mu}\int \mspace{-4.5mu}x_m y_p d (\pi^\ell+\pi_0)(x,y)\bigg|\bigg|\mspace{-2.5mu}\int\mspace{-4.5mu} x_m y_p d (\pi^\ell-\pi_0)(x,y)\bigg|,
\end{split}
\]
bound the first absolute value on the RHS using the Cauchy-Schwarz inequality, by which we have 
$|\int x_m y_p d \pi|\leq \sqrt{M_2(\mu)M_2(\nu)}$ for any coupling $\pi$, and proceed to bound the second expression as
\begin{align*}
    \left|\int x_m y_p d (\pi^\ell-\pi_0)(x,y)\right|&\leq \sum_{(i,j)\in \ZZ^{d_x+d_y}} \pi_0(Q_{ij}^\ell)  \left(\sup_{(x,y)\in Q_{ij}^\ell} x_m y_p - \inf_{(x,y)\in Q_{ij}^\ell} x_m y_p\right) \\
    &\lesssim \sum_{(i,j)\in \ZZ^{d_x+d_y}} \ell\left(\sqrt{d_x}+\sqrt{d_y}\right) \pi_0(Q_{ij}^\ell)\cdot\sup_{(x,y)\in Q_{ij}^\ell} (\|x\|+\|y\|).
\end{align*}
Again by Lemma \ref{lemma:block_approximate} we have $\big|\int x_m y_p d (\pi^\ell-\pi_0)\big|\lesssim \ell\left(\sqrt{d_x}+\sqrt{d_y}\right)^2  \big(1+M_1(\mu)+M_1(\nu)\big)$. Inserting the derived bounds back into \eqref{eq:approx_error_decomp}, we obtain
\begin{equation}
     \sG(\pi^\ell)-\sG(\pi_0) \lesssim \ell\left(\sqrt{d_x}+\sqrt{d_y}\right)^{6} \big(1+M_4(\mu)+M_4(\nu)\big).\label{eq:approx_term1}
\end{equation}

It remain to bound the KL divergence term. Observe that
\begin{align*}
    \dkl\big(\pi^\ell\big\|\mu\otimes\nu\big) & = \sum_{(i,j)\in \ZZ^{d_x + d_y}} \log\left(\frac{\pi_0\big(Q_{ij}^\ell\big)}{\mu\big(Q_{i}^\ell\big)\nu\big(Q_{j}^\ell\big)}\right) \pi_0\big(Q_{ij}^\ell\big)\\
    &\leq -\sum_{i\in \ZZ^{d_x}} \mu\big(Q_{i}^\ell\big)\log\big(\mu\big(Q_{i}^\ell\big)\big) - \sum_{j\in \ZZ^{d_y}} \nu(Q_{i}^\ell)\log\big(\nu\big(Q_{i}^\ell\big)\big),\numberthis\label{eq:approx_kl_bound}
\end{align*}
and proceed by bounding the two entropy terms as follows. Define $\mu^\ell$ such that it has a constant Lebesgue density $Q_{i}^\ell$, for $i\in\ZZ^{d_x}$, given by $\frac{ d \mu^\ell}{ d x}=\frac{\mu(Q_{i}^\ell)}{\ell^{d_x}}$. Clearly 
\begin{align*}
-\sum_{i\in \ZZ^{d_x}} \mu\big(Q_{i}^\ell\big)\log\big(\mu\big(Q_{i}^\ell\big)\big)& =- \int_{\RR^{d_x}}\frac{ d \mu^\ell}{ d x}(x)\log\left(\frac{ d \mu^\ell}{ d x}(x)\right)  d  x - d_x\log(\ell)\\
&= \sh\big(\mu^\ell\big) - d_x\log(\ell),
\end{align*}
where $\sh$ denotes the differential entropy. Since differential entropy is maximized by the Gaussian distribution with the same covariance matrix, we further have
\begin{align*}
    \sh\big(\mu^\ell\big)\lesssim d_x + \log\det\big(\bsigma_{\mu^\ell}\big)
\end{align*}
where $\bsigma_{\mu^\ell}$ is the covariance matrix of $\mu^\ell$. A similar bound applies for the second term in \eqref{eq:approx_kl_bound}. Combining these with \eqref{eq:approx_term1}, yields
\begin{align*}
    \EGW_\eps(\mu,\nu)-\GW(\mu,\nu)^2  &\lesssim  \ell\big(\sqrt{d_x}+\sqrt{d_y}\big)^{6} \big(1+M_4(\mu)+M_4(\nu)\big)\\
    &\quad\quad+ \eps \big(d_x+d_y + \log\det\big(\bsigma_{\mu^\ell}\big) + \log\det\big(\bsigma_{\nu^\ell}\big) - (d_x+d_y)\log(\ell)\big).\numberthis\label{eq:approx_final_bound_delta}
\end{align*}
To eliminate the dependence on $\ell$ in the above bound we again invoke Lemma \ref{lemma:block_approximate} to bound the entries of $\bsigma_{\mu^\ell}$. We have
\begin{align*}
    \left|\int x_m x_p  d (\mu^\ell - \mu)(x)\right| &\leq  \sum_{i\in \ZZ^{d_x}} \mu\big(Q_{i}^\ell\big)\left(\sup_{x\in Q_i^\ell}  x_m x_p  - \inf_{x\in Q_i^\ell}  x_m x_p \right)\\
    &\lesssim \sum_{i\in \ZZ^{d_x}} \ell\sqrt{d_x}\,\mu\big(Q_{i}^\ell\big)\sup_{x\in Q_i^\ell} \|x\|\\
    &\lesssim  \ell d_x\big(1+M_1(\mu)\big).\numberthis\label{eq:for_later}
\end{align*}
Given this entrywise bound on $\bsigma_{\mu^\ell}$, we obtain $\det(\bsigma_{\mu^\ell})\leq d_x! \big(M_2(\mu)+d_x\big(1+M_1(\mu)\big)\big)^{d_x}$ and similarly for $\det(\bsigma_\nu^\ell)$. Inserting the determinant bounds into \eqref{eq:approx_final_bound_delta}, we minimize the RHS over $\ell$ and set\footnote{This choice is optimal for small $\eps$ and, generally, always feasible.} $\ell = \frac{(d_x+d_y)\eps}{(\sqrt{d_x}+\sqrt{d_y})^{6} (1+M_4(\mu)+M_4(\nu))}\land 1$ 
to obtain 
\[
    \EGW_\eps(\mu,\nu)-\GW(\mu,\nu)^2 \lesssim_{d_x,d_y,M_4(\mu),M_4(\nu)} \eps \log\frac{1}{\eps},
\]
as claimed.\qed

\section{Proof of Proposition \ref{prop:convergence-of-plans}}\label{proof:prop:convergence-of-plans}

To establish weak convergence of optimal EGW couplings towards an optimal coupling of the unregularized problem as $\eps\to 0$, we use the notion of $\Gamma$-convergence. The proof technique is inspired by ideas from \cite{carlier2017convergence,goldfeld2020gaussian}. Let $\sF,\sF_k:\cP(\RR^{d_x}\times\RR^{d_y})\to \RR$ for $k\in\NN$. We say that the sequence $\{\sF_k\}_{k\in\NN}$ is $\Gamma$-convergent to $\sF$ if the following 2 conditions hold for any 
$\gamma\in\cP(\RR^{d_x}\times \RR^{d_y})$:
\begin{itemize}
    \item For any sequence $\{\gamma_k\}_{k\in\NN}$ with $\gamma_k\xrightarrow[]{w}\gamma$,
    \begin{align*}
        \sF(\gamma)\leq\liminf_{k\to\infty} \sF_k(\gamma_k).
    \end{align*}
    \item
    There exists a sequence $\{\gamma_k\}_{k\in\NN}$ with $\gamma_k\xrightarrow[]{w}\gamma$, and
    \begin{align*}
        \sF(\gamma)\geq\limsup_{k\to\infty} \sF_k(\gamma_k).
    \end{align*}
\end{itemize}

After establishing $\Gamma$-converge of the appropriate EGW functional, we will invoke the following result to deduce convergence of optimal couplings; cf., e.g., \cite[Proposition 7.18]{dal2012introduction}.

\begin{proposition}[Optimality of cluster points]\label{prop:cluster_optimal}
Suppose that $\sF_k$ has a minimizer $\gamma_k$, for each $k\in\NN$. If $\{\sF_k\}_{k\in\NN}$ is equi-coersive and $\Gamma$-converges to $\sF$, then any cluster point $\gamma$ of $\{\gamma_k\}_{k\in\NN}$ minimizes $\sF$.
\end{proposition}

For any positive sequence $\eps_k$ decreasing to $\eps\geq 0$ and with $\sG(\pi)\coloneqq \|\Gamma\|^2_{L^2(\pi\otimes\pi)}$ as before, consider the functionals
\begin{align*}
    \sF_k(\gamma)&=\begin{cases} \sG(\gamma)+\eps_k \dkl(\gamma\|\mu\otimes\nu),&\text{ if } \gamma\in\Pi(\mu,\nu)\\+\infty,&\text{ o.w.}\end{cases}\\
    \sF(\gamma)&= \begin{cases} \sG(\gamma)+\eps \dkl(\gamma\|\mu\otimes\nu),&\text{ if } \gamma\in\Pi(\mu,\nu)\\+\infty,&\text{ o.w.}\end{cases}.
\end{align*}
Clearly, these functionals are equi-coersive since $\Pi(\mu,\nu)$ is tight. For any $\eps>0$, the two conditions of $\Gamma$-convergence holds by lower semicontinuity of the KL-divergence and simply picking $\gamma_k=\gamma$, for all $k\in\NN$

To prove the claim, it remains to establish $\Gamma$-convergence for $\eps_k\to\eps=0$. The first condition is easy to verify as the entropy term is always positive and $\sF$ is lower semicontinuous. Indeed, $\dkl$ is lower semicontinuous in both of its arguments, while $\sG$ is weakly continuous on $\Pi(\mu,\nu)$ for $\mu,\nu$ with bounded 4th moments. For the second condition, we construct the appropriate sequence. We start from an arbitrary convergent sequence  $\tilde\gamma_k\to \gamma$, and assume without loss of generality that $\gamma,\tilde\gamma_k\in\Pi(\mu,\nu)$, for all $k\in\NN$. We now use the block approximation idea from the previous section (cf. the start of \cref{proof:prop:cont_of_cost_in_eps}): define a sequence of diameters $\ell_k=\eps_k\land1$ and, using the notation from previous section, consider the block approximation $\gamma_k\coloneqq \tilde\gamma_k^{\ell_k}\in\Pi(\mu,\nu)$. Using the KL divergence bounds leading to \eqref{eq:approx_final_bound_delta} and the determinant bound after \eqref{eq:for_later}, we obtain
\begin{align*}
    &\eps_k \dkl(\gamma_k\|\mu\otimes\nu)  \lesssim \eps_k \Big[d_x +d_y+ \log\det\left(d_x! \big(M_2(\mu)+d_x\big(1+M_1(\mu)\big)\big)^{d_x}\right)\\
    & \hspace{9em}+ \log\det\left(d_x! \big(M_2(\mu)+d_x\big(1+M_1(\mu)\big)\big)^{d_x}\right)- (d_x+d_y)\log(\ell)\Big].
\end{align*}
Also note that $\gamma_k\xrightarrow[]{w}\gamma$, since for any 1-Lipschitz function $f$, we have 
\[\left|\int f d (\gamma_k-\tilde{\gamma}_k)\right|\leq \sum_i \tilde{\gamma}_k(Q_i)\bigg|\sup_{x\in Q_i}f-\inf_{x\in Q_i}f\bigg|\lesssim_{d_x,d_y}\ell_k,\]
and consequently $\sG(\gamma_k)\rightarrow \sG(\gamma)$. Combining the pieces, we deduce that 
\begin{align*}
    \sG(\gamma) = \lim_{k\to\infty} \sG(\gamma_k) + \eps_k \dkl(\gamma_k\|\mu\otimes\nu).
\end{align*}
This concludes the proof of $\Gamma$-convergence, from which weak convergence of optimal couplings as stated in \cref{prop:convergence-of-plans} follows by \cref{prop:cluster_optimal}. \qed

 {
\section{Proof of Lemma \ref{lemma:equivalence_gw_w}}\label{appen:proof:lemma:equivalence_gw_w}

Throughout this proof, we omit the dummy variables from the probability measure in our notation for integrals, writing $\int f(x,y,x',y')d\pi\otimes\pi$ instead of $\int f(x,y,x',y')d\pi\otimes\pi(x,y,x',y')$. For the first inequality, we now have
\begin{align*}
    &\GW_{p,q}(\mu,\nu)^p\\
    &=\int \big|\|x-x'\|^q-\|y-y'\|^q\big|^p d\pi\otimes\pi\\
    &\leq q^p \int \left(\|x-x'\|^{q-1}+\|y-y'\|^{q-1}\right)^p \big(\|x-y\|+\|x'-y'\|\big)^p  d\pi\otimes\pi \\
    &\leq q^p\mspace{-3mu}\left(\int\mspace{-3mu}\big(\|x\mspace{-3mu}-\mspace{-3mu}x'\|^{q-1}\mspace{-3mu}+\mspace{-3mu}\|y\mspace{-3mu}-\mspace{-3mu}y'\|^{q-1}\big)^{\frac{pq}{q-1}}d\pi\otimes \pi\mspace{-3mu}\right)^{\frac{q-1}{q} }   \mspace{-3mu}\left(\int\mspace{-3mu}\big(\|x\mspace{-3mu}-\mspace{-3mu}y\|\mspace{-3mu}+\mspace{-3mu}\|x'\mspace{-3mu}-\mspace{-3mu}y'\|\big)^{qp}  d\pi\otimes \pi \right)^{\frac 1q} \\
    &\leq q^p 2^{2p-1}\mspace{-3mu}\left(\int\mspace{-3mu}  \|x\mspace{-3mu}-\mspace{-3mu}x'\|^{pq}\mspace{-3mu}+\mspace{-3mu}\|y\mspace{-3mu}-\mspace{-3mu}y'\|^{pq} d\pi\otimes \pi\right)^{\frac{q-1}{q}}  \left(\int\mspace{-3mu}\|x\mspace{-3mu}-\mspace{-3mu}y\|^{qp}\mspace{-3mu}+\mspace{-3mu}\|x'\mspace{-3mu}-\mspace{-3mu}y'\|^{qp} d\pi\otimes \pi  \right)^{\frac 1q} \\
    &\leq q^p 2^{pq+p-1+1/q}\big(M_{pq}(\mu)+M_{pq}(\nu) \big)^{\frac{q-1}{q}}  \left(\int  \|x-y\|^{qp}d\pi\right)^{\frac 1q},
\end{align*}
where the second line follows by mean value theorem for the function $x\mapsto x^q$, while the third line uses by H\"older's inequality.

\medskip
For the second inequality, suppose first that $\mu,\nu$ are centered. We may now expand 
\begin{align*}
    &\GW(\mu,\nu)^2\\
    &\hspace{1em}=\mspace{-5mu}  \inf_{\pi\in\Pi(\mu,\nu)}  \mspace{-6mu}2\big(M_2(\mu)\mspace{-3mu}-\mspace{-3mu}M_2(\nu)\big)^2 \mspace{-4mu}+\mspace{-2mu} 2 \mspace{-4mu}\int\mspace{-4mu} \big(\|x\|^2\mspace{-3mu}-\mspace{-3mu}\|y\|^2\big)^2 d\pi \mspace{-2mu}+\mspace{-2mu} 4 \mspace{-3mu}\int\mspace{-3mu} \big(\langle x,x'\rangle\mspace{-2mu}-\mspace{-2mu}\langle y,y'\rangle\big)^2 d\pi\mspace{-2mu}\otimes\mspace{-2mu}\pi\\
    &\hspace{1em}=\mspace{-2mu}  2\big(M_2(\mu)\mspace{-3mu}-\mspace{-3mu}M_2(\nu)\big)^2 \mspace{-4mu}+\mspace{-6mu} \inf_{\pi\in\Pi(\mu,\nu)} \mspace{-4mu}2\mspace{-4mu}\int\mspace{-4mu} \big(\|x\|^2\mspace{-3mu}-\mspace{-3mu}\|y\|^2\big)^2 d\pi \mspace{-2mu}+\mspace{-2mu} 4 \big(\|\bsigma_\mu\|_\F^2\mspace{-2mu} +\mspace{-2mu} \|\bsigma_\nu\|_\F^2 \mspace{-2mu} -\mspace{-2mu}  2\|\bgamma_\pi\|_\F^2 \big),
\end{align*}
where 
$\bm\Gamma_\pi=\int x y^\intercal d\pi$ is the cross-covariance of $(X,Y)\sim\pi$.

As the bound trivializes when $\GW(\mu,\nu)=0$, suppose that $\GW(\mu,\nu)^2=\iota>0$ and let $\pi$ be the corresponding optimal coupling. This implies $4 \big(\|\bsigma_\mu\|_\F^2+\|\bsigma_\nu\|_\F^2 - 2\|\bgamma_\pi\|_\F^2 \big)\leq \iota$. Consider the singular value decomposition $\bgamma_\pi=\bP\blambda \bQ^\intercal$, where $\bP,\bQ\in O(d)$, and $\bgamma$ is diagonal. 
By invariance of the GW distance to rotations and since $\tilde\pi=(\bP^\intercal,\bQ^\intercal)_\sharp\pi$ is optimal for $\mathsf{GW}(\bP^\intercal_\sharp\mu,\bQ^\intercal_\sharp\nu)$, we similarly obtain $4 \big(\|\bP^\intercal\bsigma_\mu \bP\|_\F^2+\|\bQ^\intercal\bsigma_\nu \bQ\|_\F^2 - 2\|\blambda \|_\F^2 \big)\leq \iota$.
Denote the singular values of a matrix $\mathbf{A}\in\RR^{d\times d}$ by $\sigma_1(\mathbf{A}),\ldots,\sigma_d(\mathbf{A})$. Also denote the diagonal entries of $\bP\bsigma_\mu \bP^\intercal, \bQ\bsigma_\nu \bQ^\intercal$ as $a_1,\cdots,a_d$ and $b_1,\cdots, b_d$, respectively. We thus obtain
\begin{align*}
    \sum_{i=1}^d \big(a_i^2+b_i^2 -2 \sigma_i(\bgamma_\pi)^2\big) \leq \iota.
\end{align*}
Observing that $a_i+b_i-2\sigma_i(\bgamma_\pi)\geq 0$, as $\int x_i^2+y_i^2-2x_iy_i d \tilde\pi \geq0$, we further have
\begin{align*}
    \sqrt{\frac{a_i^2+b_i^2}{2}}\geq \frac{a_i+b_i}{2}\geq \sigma_i(\bgamma_\pi),\quad\forall i=1,\ldots, d,
\end{align*}
which implies
\begin{align*}
    \iota&\geq8\sum_{i=1 }^d\left(\sqrt{\frac{a_i^2+b_i^2}{2}}+\sigma_i(\bgamma_\pi)\right)\left(\sqrt{\frac{a_i^2+b_i^2}{2}}-\sigma_i(\bgamma_\pi)\right)\\
    &\geq 8 \min_{i=1,\ldots,d} \sqrt{\frac{a_i^2+b_i^2}{2}} \cdot \sum_{i=1}^d \left(\sqrt{\frac{a_i^2+b_i^2}{2}}-\sigma_i(\bgamma_\pi)\right).
\end{align*} 
Having that, we compute
\begin{align*}
    \W_2\big(\mu,(\bP \bQ^\intercal)_\sharp\nu\big)^2 &= \W_2\big(\bP^\intercal_\sharp\mu,\bQ^\intercal_\sharp\nu\big)^2\\
    &\leq \int \|x-y\|^2 d\tilde\pi\\ 
    &= \sum_{i=1}^d\big( a_i + b_i - 2\sigma_i(\bgamma_\pi)\big)\\
    &\leq \sum_{i=1 }^d\left(\sqrt{\frac{a_i^2+b_i^2}{2}}-\sigma_i(\bgamma_\pi)\right)\\
    &\leq \frac{\iota}{8 \min_i \sqrt{\frac{a_i^2+b_i^2}{2}}}.
\end{align*}
Notice that $\lambda_{\mathrm{min}}(\bsigma_\mu)\leq a_i$ and $\lambda_{\mathrm{min}}(\bsigma_\nu)\leq b_i$, for all $i=1,\ldots,d$, and use the fact that $\bP \bQ^\intercal\in O(d)$ to conclude that 
\begin{align*}
    \Big(  32\big(\lambda_{\mathrm{min}}(\bsigma_\mu)^2+\lambda_{\mathrm{min}}(\bsigma_\nu)^2\big)\Big)^{\frac 14}\inf_{\mathbf{U}\in O(d)} \mathsf{W}_2(\mu,\mathbf{U}_\sharp\nu) \leq  \GW(\mu,\nu),
\end{align*}
whenever $\mu,\nu$ are centered. To remove the centering assumption one only has to replace $O(d)$ above with the isometry group $E(d)$, which contains translations in addition to rotations.\qed

}

\section{Proofs of Lemmas for Theorem \ref{thm:egw_sample_complex}}\label{appen:proof:egw_sample_complex_lemmas}

\subsection{Proof of \cref{lemma:bias}}\label{appen:proof:lemma:bias}

Let $\bar x_n \coloneqq  \int x \hat{\mu}_n$, $\bar y_n \coloneqq  \int y\hat{\nu}_n$ denote the sample means and define $\tilde{\mu}_n,\tilde{\nu}_n$ as the centered versions of the empirical distributions, i.e., $\tilde{\mu}_n=(\cdot-\bar x_n)_\sharp \hat{\mu}_n$ and similarly for $\tilde\nu$. Note that $\EGW_\eps(\hat{\mu}_n,\hat{\nu}_n) = \EGW_\eps(\tilde{\mu}_n,\tilde{\nu}_n) = \EGW^1(\tilde{\mu}_n,\tilde{\nu}_n)+\EGW^2_{\eps} (\tilde{\mu}_n,\tilde{\nu}_n)$ and so
\begin{align*}
    \EE\big[\big|\EGW_\eps(\mu,\nu)-\EGW_\eps(\hat{\mu}_n,\hat{\nu}_n)\big|\big] &\leq \EE\big[\big|\EGW^1(\mu,\nu)\mspace{-3mu}-\mspace{-3mu}\EGW^1(\hat{\mu}_n,\hat{\nu}_n)\big|\big] \mspace{-3mu}+ \mspace{-3mu}\EE\big[\big|\EGW^2_{\eps} (\mu,\nu)\mspace{-3mu}-\mspace{-3mu}\EGW^2_{\eps} (\hat{\mu}_n,\hat{\nu}_n)\big|\big] \\
    &\quad \mspace{-3mu}+ \mspace{-3mu}\EE\big[\big|\EGW^1(\hat{\mu}_n,\hat{\nu}_n)\mspace{-3mu}-\mspace{-3mu}\EGW^1(\tilde{\mu}_n,\tilde{\nu}_n)\big|\big]\mspace{-3mu} + \mspace{-3mu}\EE\big[\big|\EGW^2_{\eps} (\hat{\mu}_n,\hat{\nu}_n)\mspace{-3mu}-\mspace{-3mu}\EGW^2_{\eps} (\tilde{\mu}_n,\tilde{\nu}_n)\big|\big]
\end{align*}
We proceed by bounding the terms in the second line. For the first one, observe
\begin{align*}
    \EE\mspace{-3mu}\big[\big|\EGW^1(\hat{\mu}_n,\hat{\nu}_n)\mspace{-3mu}-\mspace{-3mu}\EGW^1(\tilde{\mu}_n,\tilde{\nu}_n)\big|\big] &\lesssim \EE\mspace{-3mu}\left[\left| \int \Big( \|x-\bar x_n\|^2\|y-\bar y_n\|^2 - \|x\|^2\|y\|^2\Big)  d  \hat{\mu}_n\mspace{-3mu}\otimes\mspace{-3mu}\hat{\nu}_n(x,y)\right|\right]\\
    & = \EE\mspace{-3mu}\left[\left| \|\bar x_n\|^2\|\bar y_n\|^2 \mspace{-5mu} -\mspace{-5mu}\|\bar x_n\|^2\mspace{-7mu}\int\mspace{-5mu}\|y\|^2  d \hat{\nu}_n\mspace{-1mu}(y) \mspace{-5mu} -\mspace{-5mu}\|\bar y_n\|^2\mspace{-7mu}\int\mspace{-5mu}\|x\|^2 d \hat{\mu}_n\mspace{-1mu}(x) \right|\right]\\
    &\lesssim \frac{\sigma^2}{n}.\numberthis\label{eq:sample_complex_bias_S1}
\end{align*}
In the last step above we have used the following bound on the 4th absolute moment of the sample mean. Write $\bar x_n=\frac{1}{n}\sum_{i=1}^n X_i$, where $X_1,\ldots,X_n$ are the i.i.d. samples defining the empirical measure $\hat{\mu}_n$. Consider:
\begin{align*}
    \EE[\|\bar x_n\|^4] &
    = \frac{1}{n^4}\EE\left[\left(\sum\nolimits_{i,j} \langle X_i,X_j\rangle \right)^2 \right] \\
    &= \frac{1}{n^4}\EE\left[2\sum\nolimits_{i\neq j} \langle X_i,X_j\rangle^2 +   \sum\nolimits_{i,j}\langle X_i,X_i\rangle \langle X_j,X_j\rangle   \right]\\
    &= \frac{1}{n^4}\EE\left[  2\sum\nolimits_{i\neq j} \langle X_i,X_j\rangle^2 +   \sum\nolimits_{i\neq j}\langle X_i,X_i\rangle \langle X_j,X_j\rangle + \sum\nolimits_i \langle X_i,X_i\rangle^2 \right] \\
    &=  \frac{1}{n^4}\left(2n(n-1) \|\Sigma_\mu\|_\F^2 + n(n-1)M_2(\mu)^2 + nM_4(\mu)\right)\\
    &\leq \frac{3n^2 M_4(\mu)}{n^4}\\
    &\leq \frac{6\sigma^2}{n^2}
\end{align*}
where the two last steps bound $\|\Sigma_\mu\|_\F^2\leq M_2(\mu)^2\leq M_4(\mu)$ and $M_4(\mu)\leq 2\sigma^2$.
  
\medskip
It remains to analyze the centering bias of $\EGW^2_{\eps} $. Consider
\begin{align*}
    \EE\big[\big|&\EGW^2_{\eps} (\hat{\mu}_n,\hat{\nu}_n)-\EGW^2_{\eps} (\tilde{\mu}_n,\tilde{\nu}_n)\big|\big]\\ &\lesssim \EE\left[\sup_{\pi\in\Pi( \hat{\mu}_n,\hat{\nu}_n )} \left|\int \Big( \|x-\bar x_n\|^2\|y-\bar y_n\|^2 - \|x\|^2\|y\|^2\Big)  d  \pi(x,y)\right|\right ]\\
    &+ \EE\left[\sup_{\pi\in\Pi( \hat{\mu}_n,\hat{\nu}_n )}\left|  \sum_{\substack{1\leq i\leq d_x\\1\leq j \leq d_y}}\left(\int  x_iy_j d \pi(x,y)\right)^{\mspace{-2mu}2} \mspace{-5mu}- \mspace{-3mu}\left(\int  (x_i-\bar{x}_{n,i})(y_j-\bar{y}_{n,j}) d \pi(x,y)\right)^{\mspace{-2mu}2} \right|\right].\numberthis\label{eq:sample_complex_bias_S2}
\end{align*}
For the first term above, we have
\begin{align*}
    &\EE\left[\sup_{\pi\in\Pi( \hat{\mu}_n,\hat{\nu}_n )} \left|\int \Big( \|x-\bar x_n\|^2\|y-\bar y_n\|^2 - \|x\|^2\|y\|^2\Big)  d  \pi(x,y)\right|\right ] \\
    &\qquad= \EE\Bigg[\sup_{\pi\in\Pi(\hat{\mu}_n,\hat{\nu}_n )}\bigg|2\int \Big(2\langle x, \bar x_n\rangle \langle y,\bar y_n\rangle -\langle x, \bar x_n\rangle \|y\|^2 -\langle y,\bar y_n\rangle \|x\|^2\Big)
     d  \pi(x,y) \\
    &\qquad\qquad\qquad\qquad-3\|\bar x_n\|^2\|\bar y_n\|^2 + \|\bar x_n\|^2\int\|y\|^2 d \hat{\nu}_n(y) + \|\bar y_n\|^2\int\|x\|^2 d \hat{\mu}_n(x)\bigg|\Bigg]\\
    &\quad\lesssim \frac{\sigma^2}{\sqrt{n}}  ,
\end{align*}
using the same fourth moment expansion of $\bar x_n$ as above. For the second term, we have
\begin{align*}
    &\EE\left[\sup_{\pi\in\Pi( \hat{\mu}_n,\hat{\nu}_n )}\left|  \sum_{\substack{1\leq i\leq d_x\\1\leq j \leq d_y}}\Big(\int  x_iy_j d \pi(x,y)\Big)^2 - \Big(\int  (x_i-\bar{x}_{n,i})(y_j-\bar{y}_{n,j}) d \pi(x,y)\Big)^2 \right|\right]\\
    &\qquad=  \EE\left[\sup_{\pi\in\Pi( \hat{\mu}_n,\hat{\nu}_n )}\left|   \sum_{\substack{1\leq i\leq d_x\\1\leq j \leq d_y}}\left(\int  x_iy_j d \pi(x,y) -\int  (x_i-\bar{x}_{n,i})(y_j-\bar{y}_{n,j}) d \pi(x,y)\right)\right.\right.\\    &\qquad\qquad\qquad\ \ \ \left.\left.\phantom{\sum_{\substack{1\leq i\leq d_x\\1\leq j \leq d_y}}}\times\left(\int  x_iy_j d \pi(x,y) +\int  (x_i-\bar{x}_{n,i})(y_j-\bar{y}_{n,j}) d \pi(x,y)\right)  \right|\right]
\end{align*}
with
\begin{align*}
    &\EE\left[\sup_{\pi\in\Pi( \hat{\mu}_n,\hat{\nu}_n )}   \sum_{\substack{1\leq i\leq d_x\\1\leq j \leq d_y}} \left(\int  x_iy_j d \pi(x,y) -\int  (x_i-\bar{x}_{n,i})(y_j-\bar{y}_{n,j}) d \pi(x,y)\right)^2\right]\\
    &\qquad\qquad\qquad\qquad=  \EE\left[\sup_{\pi\in\Pi(\hat{\mu}_n,\hat{\nu}_n )}  \sum_{\substack{1\leq i\leq d_x\\1\leq j \leq d_y}}\left(\int x_i \bar{y}_{n,j} + \bar{x}_{n,i} y_j- \bar{x}_{n,i}\bar{y}_{n,j}  d \pi(x,y)\right)^2\right]\\
    &\qquad\qquad\qquad\qquad\lesssim \sum_{\substack{1\leq i\leq d_x\\1\leq j \leq d_y}}\EE\left[  \int   (x_i \bar{y}_{n,j})^2   d \hat{\mu}_n(x)+ \int(\bar{x}_{n,i} y_j)^2 d \hat{\nu}_n(x) +  (\bar{x}_{n,i}\bar{y}_{n,j})^2   \right ]\\
    &\qquad\qquad\qquad\qquad=  \EE\left[  \|\bar y_n\|^2 \int   \|x\|^2  d \hat{\mu}_n(x)+ \|\bar x_n\|^2 \int   \|y\|^2  d \hat{\nu}_n(y) +  \|\bar x_n\|^2\|\bar y_n\|^2    \right]\\
    &\qquad\qquad\qquad\qquad\lesssim \frac{\sigma^2}{n}
\end{align*}
and
\begin{align*}
&\EE\left[\sup_{\pi\in\Pi(\hat{\mu}_n,\hat{\nu}_n )}   \sum_{\substack{1\leq i\leq d_x\\1\leq j \leq d_y}} \left(\int  x_iy_j d \pi(x,y) +\int  (x_i-\bar{x}_{n,i})(y_j-\bar{y}_{n,j}) d \pi(x,y)\right)^2\right]\\
    &=  \EE\left[\sup_{\pi\in\Pi(\hat{\mu}_n,\hat{\nu}_n )}  \sum_{\substack{1\leq i\leq d_x\\1\leq j \leq d_y}}\left(\int    2x_iy_j  -x_i \bar{y}_{n,j}  - \bar{x}_{n,i} y_j + \bar{x}_{n,i}\bar{y}_{n,j} d \pi(x,y)\right)^2 \right ]\\
    &\lesssim \mspace{-3mu}\sum_{\substack{1\leq i\leq d_x\\1\leq j \leq d_y}}\mspace{-3mu}\EE\mspace{-3mu}\left[  \int\mspace{-3mu}   (x_i\bar{y}_{n,j})^2   d \hat{\mu}_n\mspace{-2mu}(x)\mspace{-3mu}+\mspace{-3mu} \int\mspace{-3mu}(\bar{x}_{n,i} y_j)^2 d \hat{\nu}_n\mspace{-2mu}(y) +  (\bar{x}_{n,i}\bar{y}_{n,j})^2   \mspace{-3mu}+\mspace{-3mu}  \int \mspace{-3mu}x_i^2 d \hat{\mu}_n\mspace{-2mu}(x)\mspace{-3mu}\int\mspace{-3mu} y_j^2  d \hat{\nu}_n\mspace{-2mu}(y)\right ]\\
    &=\EE\left[  \|\bar y_n\|^2 \int   \|x\|^2  d \hat{\mu}_n(x)+ \|\bar x_n\|^2 \int   \|y\|^2  d \hat{\nu}_n(y)+  \|\bar x_n\|^2\|\bar y_n\|^2\right]+ \EE\big[M_2(\hat{\mu}_n)M_2(\hat{\nu}_n)\big]\\
    &\lesssim \sigma^2.
\end{align*}
Combine the pieces, we obtain $\EE\big[\big|\EGW^2_{\eps} (\hat{\mu}_n,\hat{\nu}_n)-\EGW^2_{\eps} (\tilde{\mu}_n,\tilde{\nu}_n)\big|\big] \lesssim \sigma^2n^{-1/2}$, which together with \eqref{eq:sample_complex_bias_S1} concludes the proof.\qed

\subsection{Proof of \cref{lemma:EGW_S1}}\label{appen:proof:lemma:EGW_S1}

First, rewrite
\begin{align*}
&\EGW^1(\mu,\nu)\\
&=\int \|x-x'\|^4 d  \mu\otimes\mu(x,x') + \int\|y-y'\|^4 d  \nu\otimes\nu(y,y') -4\int\|x\|^2\|y\|^2  d  \mu\otimes\nu(x,y)\\
&=2 \big(M_4(\mu) +M_4(\nu)\big) +2\big(M_2(\mu)^2+M_2(\nu)^2\big)
  + 4 \big(\|\Sigma_\mu\|_\F^2+\|\Sigma_\nu\|_\F^2\big) -4M_2(\mu)M_2(\nu).
\end{align*}
With this expansion, the empirical estimation error of $\EGW^1$ can be bounded as 
\begin{align*}
&    \EE\big[\big|\EGW^1(\mu,\nu) - \EGW^1(\hat{\mu}_n,\hat{\nu}_n)\big|\big]\\
    &\lesssim \EE\big[\big|M_4(\mu)\mspace{-3mu}-\mspace{-3mu}M_4(\hat{\mu}_n)\big|\big] \mspace{-3mu}+\mspace{-3mu} \EE\big[\big|M_4(\nu)\mspace{-3mu}-\mspace{-3mu}M_4(\hat{\nu}_n)\big|\big]\mspace{-3mu}+\mspace{-3mu} \EE\big[\big|M_2(\mu)M_2(\nu)\mspace{-3mu}-\mspace{-3mu}M_2(\hat{\mu}_n)M_2(\hat{\nu}_n)\big|\big]\\   
    &+ \sqrt{\EE\left[\| \Sigma_\mu-\Sigma_{\hat{\mu}_n}\|_\F^2\right]\EE\left[\|\Sigma_\mu+\Sigma_{\hat{\mu}_n}\|_\F^2\right]}+  \sqrt{\EE\left[\| \Sigma_\nu-\Sigma_{\hat{\nu}_n}\|_\F^2\right]\EE\left[\| \Sigma_\nu+\Sigma_{\hat{\nu}_n}\|_\F^2\right]}\\
    &+\sqrt{\EE\Big[\big(M_2(\mu)+M_2(\hat{\mu}_n)\big)^{2}\Big]\EE\Big[\big(M_2(\mu)-M_2(\hat{\mu}_n)\big)^{2}\Big]}\\
    &\qquad\qquad\qquad\qquad\qquad\qquad\qquad+\sqrt{\EE\Big[\big(M_2(\nu)+M_2(\hat{\nu}_n)\big)^{2}\Big]\EE\Big[\big(M_2(\nu)-M_2(\hat{\nu}_n)\big)^{2}\Big]}.\numberthis\label{eq:S1_decomposition}
\end{align*}
For any distribution $\eta\in\cP(\RR^d)$ that is 4-sub-Weibull with parameter $\sigma^2$, we have
\begin{align*}
\|\Sigma_\eta\|_\F^2 &\leq M_4(\eta) \leq 2\sqrt{2}\sigma^2\\
\EE\left[\| \Sigma_\eta-\Sigma_{\hat\eta_n}\|_\F^2\right] &\leq \frac{M_4(\eta)}{n} \leq \frac{2\sqrt{2}\sigma^2}{n},\\
\EE\Big[\big|M_k(\eta)-M_k(\hat{\eta}_n)\big|^2\Big]&\leq \frac{M_{2k}(\eta)}{n}\leq \frac{(2\sigma^2)^{k/2}(k!)^{1/2}}{n},\quad \forall k\in\NN.
\end{align*}
As $\mu,\nu$ are assumed to satisfy the sub-Weibull condition, inserting the above bounds into \eqref{eq:S1_decomposition} yields
\begin{align*}
    \EE\big[\big|\EGW^1(\mu,\nu) - \EGW^1(\hat{\mu}_n,\hat{\nu}_n)\big|\big] 
    \lesssim \frac{1+\sigma^4}{\sqrt{n}}.
\end{align*}

\subsection{Proof of \cref{lemma:smoothness_of_EGW_potential}}\label{appen:proof:lemma:smoothness_of_EGW_potential}

For any fixed $\mathbf{A}\in\cD_M$, the existence of optimal potentials $(\varphi,\psi)\in  L^1(\mu)\times L^1(\nu)$ follows by standard EOT arguments; see, e.g., \cite[Lemma 1]{goldfeld2022limit}. Also recall that EOT potentials are unique up to additive constants in the sense that if $(\tilde\varphi,\tilde\psi)$ is another pair of EOT potentials, then there exists a constant $a\in\RR$ such that $\tilde\varphi=\varphi+a$ $\mu$-a.e. and $\tilde\psi=\psi-a$ $\nu$-a.e. Thus, let $(\varphi_0,\psi_0)\in  L^1(\mu)\times L^1(\nu)$ be optimal EOT potentials for the cost $c_{\mathbf{A}}$ and assume, without loss of generality, that they are normalized such that $\int \varphi_0 d \mu= \int \psi_0 d \nu = \frac{1}{2} \mathsf{OT}_{\mathbf{A},1}(\mu,\nu)$.

Recall that the optimal potentials satisfies the Schr\"{o}dinger system from \eqref{EQ:Schrodinger}. Define new functions $\varphi$ and $\psi$ as
\begin{align*}
\varphi(x) &\coloneqq  -\log\int e^{\psi_0(y) - c_{\mathbf{A}}(x,y) }  d \nu(y),\qquad x\in\RR^{d_x}\\
\psi(y) &\coloneqq  -\log\int e^{\varphi(x)- c_{\mathbf{A}}(x,y)}  d \mu(x),\ \qquad y\in\RR^{d_y}.
\end{align*}

These integrals are clearly well-defined as the integrands are everywhere positive, and the $\varphi_0,\psi_0$ functions are defined on the supports of, $\mu,\nu$ respectively. We next show that $\varphi,\psi$ are pointwise finite. For the upper bound, by Jensen's inequality, we have
\begin{equation}
\varphi(x)\leq \int c_{\mathbf{A}}(x,y)-\psi_0(y) d \nu(y)\leq   4\sigma^2+8M\sqrt{2\sigma d_xd_y}\left(\sqrt{\sigma}+2\|x\|^2\right),
\label{eq:EGW_sample_comlex_phi_bound}
\end{equation}
with the second inequality following from $\int \psi_0 d \nu = \frac{1}{2} \mathsf{OT}_{\mathbf{A},1}(\mu,\nu)$ and
\begin{align*}
-\mathsf{OT}_{\mathbf{A},1}(\mu,\nu)&\mspace{-3mu}\leq\mspace{-3mu} 4\left(\int \mspace{-3mu}\|x\|^4 \mspace{-2mu}d \mu(x)\mspace{-5mu}\int \mspace{-3mu}\|y\|^4 \mspace{-2mu}d \nu(x)\mspace{-3mu}\right)^{\mspace{-2mu}\frac 12} \mspace{-10mu}+\mspace{-3mu} 16M\mspace{-3mu}\sqrt{d_xd_y}\mspace{-3mu}\left(\int \mspace{-3mu}\|x\|^2 \mspace{-2mu}d \mu(x)\mspace{-5mu}\int \mspace{-3mu}\|y\|^2 \mspace{-2mu}d \nu(y)\mspace{-3mu}\right)^{\mspace{-2mu}\frac 12}\\
&\leq  8\sigma^2+16M\sigma \sqrt{2d_xd_y},
\end{align*}
where we have used $\|\mathbf{A}\|_2\leq \|\mathbf{A}\|_\F\leq \sqrt{d_xd_y}M/2$, the Cauchy–Schwarz inequality, and the fact that the $2k$-th moment of $\sigma^2$-sub-Gaussian distributions is bounded by $(2\sigma^2)^k k!$.
The upper bound holds similarly for $\psi$ globally and for $\psi_0$ on the support of $\nu$. For lower bound, consider the following 
\begin{align*}
    -\varphi(x)&\leq\log\int e^{4\|x\|^2\|y\|^2+16M\sqrt{d_xd_y}\|x\|\|y\| +4\sigma^2+8M\sigma \sqrt{2d_xd_y} + 16M2^{1/4}\sqrt{\sigma d_xd_y}\|y\| }  d \nu(y) \\
    &= 4\sigma^2\mspace{-3mu}+\mspace{-3mu}8M\sigma \sqrt{2d_xd_y}\mspace{-3mu}+\mspace{-3mu}\log\int e^{4\|x\|^2\|y\|^2+16M\sqrt{d_xd_y}\|x\|\|y\| + 16M2^{1/4}\sqrt{\sigma d_xd_y}\|y\|}  d \nu(y) \\
    &\leq 4\sigma^2+8M\sigma \sqrt{2d_xd_y}+8M\sqrt{d_xd_y}\|x\|^2  + 8M2^{1/4}\sqrt{\sigma d_xd_y}\\
    &\qquad\qquad\qquad\quad\quad\quad\quad\ \ +\log\mspace{-3mu}\int e^{4\|x\|^2\|y\|^2+8M\sqrt{d_xd_y}\|y\|^2 + 8M2^{1/4}\sqrt{\sigma d_xd_y} \|y\|^2 }  d \nu(y) \\
    &\leq\log2\mspace{-3mu}+\mspace{-3mu}4\sigma^2\mspace{-3mu}+\mspace{-3mu}8M\sqrt{2\sigma d_xd_y}\mspace{-3mu}\left(\mspace{-3mu}1\mspace{-3mu}+\mspace{-3mu}\sqrt{\sigma}\mspace{-3mu}+\mspace{-3mu}\frac{\|x\|^2}{\sqrt{2\sigma}}\right)\mspace{-3mu}\\
    &\qquad\qquad\qquad\qquad\qquad\qquad   \qquad\qquad\quad+\mspace{-3mu}8\sigma^2\mspace{-3mu}\left(2M\mspace{-3mu}\sqrt{d_xd_y}\big(1\mspace{-3mu}+\mspace{-3mu}\sqrt{2\sigma}\big)\mspace{-3mu}+\mspace{-3mu}\|x\|^2\right)^{\mspace{-2mu}2},
\end{align*}
where we have used the bounds $2\|x\|\|y\|\leq\|x\|^2+\|y\|^2$ and $2\|y\|\leq1+\|y\|^2$ inside the exponents, with the last step also utilizing the fact that $\int e^{t|x|} d \eta(x)\leq 2e^{t^2\sigma^2/2}$ for a $\sigma^2$-sub-Gaussian $\eta$.

Notice that $\varphi$ and $\psi$ are defined on the whole spaces $\RR^{d_x}$ and $\RR^{d_y}$, respectively, with pointwise bounds proven above. Further observe that they are also optimal for $\mathsf{OT}_{\mathbf{A},1}(\mu,\nu)$. Indeed, by Jensen's inequality we have
\begin{align*}
    \int (\varphi_0-\varphi) d \mu +\int (\psi_0-\psi) d \nu&\leq \log\int e^{\varphi_0-\varphi}  d \mu+\log\int e^{\psi_0-\psi}  d \nu\\
    &= \log\int e^{\varphi_0(x)+\psi_0(y)-c_{\mathbf{A}}(x,y)} d \mu\otimes\nu(x,y)\\
    &\qquad\quad+\log\int e^{\varphi(x)+\psi_0(y)-c_{\mathbf{A}}(x,y)}  d \mu\otimes\nu(x,y)\\
    &=0.
\end{align*}
By the strict concavity of the logarithm function we further conclude that $\varphi=\varphi_0$ $\mu$-a.s and $\psi=\psi_0$ $\nu$-a.s. The differentiability of $(\varphi,\psi)$ is clear from their definition, and it thus remains to establish the magnitude of derivative bounds stated in the lemma.

For any multi-index $\alpha$, the multivariate Faa di Bruno formula (see \cite[Corollary 2.10]{Constantine1996AMF}) implies
\begin{equation}
    -D^\alpha \varphi(x) = \sum_{r=1}^{|\alpha|}\sum_{p(\alpha,r)} \frac{\alpha!(r-1)!(-1)^{r-1}}{\prod_{j=1}^{|\alpha|}(k_j!)(\beta_j!)^{k_j}}  \prod_{j=1}^{|\alpha|}\left(\frac{ D^{\beta_j}\int e^{\psi_0(y)-c_{\mathbf{A}}(x,y)}  d \nu(y)}{\int e^{\psi_0(y)-c_{\mathbf{A}}(x,y)}  d \nu(y)} \right)^{k_j}\label{eq:sample_complex_derivative}
\end{equation}
where $p(\alpha,r)$ is the collection of all tuples $(k_1,\cdots,k_{|\alpha|};\beta_1,\cdots,\beta_{|\alpha|})\in\NN^{|\alpha|}\times\NN^{d_x\times|\alpha|}$ satisfying $\sum_{i=1}^{|\alpha|} k_i=r,\sum_{i=1}^{|\alpha|} k_i\beta_i=|\alpha|$, and for which there exists $s\in\{1,\ldots,|\alpha|\}$ such that $k_i=0$ and $\beta_i=0$ for all $i=1,\ldots |\alpha|-s$, $k_i>0$ for all $i=|\alpha|-s+1,\ldots,|\alpha|$, and $0\prec\beta_{|\alpha|-s+1}\prec\cdots\prec\beta_{|\alpha|}$. 
For a detailed discussion of this set including the linear order $\prec$, please refer to \cite{Constantine1996AMF}. For the current proof we only use the fact that the number of elements in this set solely depends on $|\alpha|$ and $r$. Given the above, it clearly suffices to bound $| D^{\beta_j}\int e^{\psi_0(y)-c_{\mathbf{A}}(x,y)}  d \nu(y)|$. First, we apply the same formula to $D^{\beta_j} e^{-c_{\mathbf{A}}(x,y)}$ and obtain
\begin{align*}
     D^{\beta_j} e^{-c_{\mathbf{A}}(x,y)} = \sum_{r'=1}^{|\beta_j|}\sum_{p(\beta_j,r')} \frac{\beta_j!}{\prod_{i=1}^{|\beta_j|}(k'_j!)(\eta_i!)^{k'_i}}   e^{-c_{\mathbf{A}}(x,y)} \prod_{i=1}^{|\beta_j|}\Big( D^{\eta_i} \big(4\|x\|^2\|y\|^2+32x^\intercal  \mathbf{A}y \big) \Big)^{k'_i}
\end{align*}
where $p(\beta_j,r')$ is a set of tuples $(k_1',\ldots,k_{|\beta_j|}';\eta_1,\ldots,\eta_{|\beta_j|})\in\NN^{|\beta_j|}\times\NN^{d_x\times|\beta_j|}$ defined similarly to the above. Observe that 
\begin{align*}
    \big|D^{\eta_i} \big(4\|x\|^2\|y\|^2+32x^\intercal  \mathbf{A}y \big)\big|&\leq 8\|y\|^2\big(1+\|x\|\big)+32\big|(\mathbf{A}y)_{m_{\eta_i}}\big|\\
    &\leq 8\|y\|^2\big(1+\|x\|\big)+16\sqrt{d_y}M\|y\|\\
    &\leq 8\|y\|^2\big(1+\|x\|\big)+8\sqrt{d_y}M\big(1+\|y\|^2\big)\\
    &\leq 8\big(1+\|y\|^2\big)\big(1+\sqrt{d_y}M+\|x\|\big)
\end{align*}
where the first inequality is justified as follows. We first observe that for any $\eta_i$ with an entry larger than 2 or that has more than a single nonzero entry, the term nullifies. Hence we denote by $m_{\eta_i}\in[1,d_x]$ the index of the first nonzero entry of $\eta_i$, and upper bound the derivative as follows. For the quadratic term, we either have $2x_{m_{\eta_i}}$ or 2, which is bounded by the sum of their absolute value, while for linear term we have the $m_{\eta_i}$-th entry of $\mathbf{A}y$ or 0. Lastly, we note that the norm of any row of $\mathbf{A}$ is bounded by $\sqrt{d_y}M/2$. Consequently, we obtain
\begin{align*}
    &\left| D^{\beta_j}\int e^{\psi_0(y)-c_{\mathbf{A}}(x,y)}  d \nu(y)\right|\\
    &\qquad\qquad\qquad\qquad\leq C_{\beta_j}\big(1+M\sqrt{d_y}+\|x\|\big)^{|\beta_j|}\int e^{\psi_0(y)-c_{\mathbf{A}}(x,y)}  \big(1+\|y\|^2\big)^{|\beta_j|} d \nu(y).
\end{align*}
Inserting this back into \eqref{eq:sample_complex_derivative}, we proceed to bound $\Big|\frac{\int e^{\psi_0(y)-c_{\mathbf{A}}(x,y)}  (\|y\|^2+1)^{|\beta_j|} d \nu(y)}{\int e^{\psi_0(y)-c_{\mathbf{A}}(x,y)}  d \nu(y)}\Big|$ by splitting the integral in the numerator into $\|y\|<\tau$ and $\|y\|\geq \tau$, for a constant $\tau>0$ to be specified later. For the former part we have \[\left|\frac{\int_{\|y\|\leq\tau} e^{\psi_0(y)-c_{\mathbf{A}}(x,y)}  (\|y\|^2+1)^{|\beta_j|} d \nu(y)}{\int e^{\psi_0(y)-c_{\mathbf{A}}(x,y)}  d \nu(y)}\right|\leq (1+\tau^2)^{|\beta_j|}, \] 
For the latter part, first observe that by \eqref{eq:EGW_sample_comlex_phi_bound} the denominator is bounded as
\begin{align*}
\left|\int_{\RR^{d_y}}e^{\psi_0(y)-c_{\mathbf{A}}(x,y)}  d \nu(y)\right|^{-1}=e^{\varphi(x)}\leq e^{4\sigma^2+8M\sqrt{\sigma d_xd_y}\left(\sqrt{2\sigma}+2^{5/4}\|x\|\right)},
\end{align*}
while for the numerator, we have 
\begin{align*}
    &\int_{\|y\|\geq\tau} e^{\psi_0(y)-c_{\mathbf{A}}(x,y)}  \big(1+\|y\|^2\big)^{|\beta_j|} d \nu(y)\\
    &\leq  e^{4\sigma^2+ 8M\sqrt{\sigma d_xd_y}\left(2^{1/4}+\sqrt{2\sigma}+\frac{\|x\|^2}{\sqrt{\sigma}} \right)}    \mspace{-5mu}\int_{\|y\|\geq\tau} \mspace{-25mu}e^{\left(4\|x\|^2+8\sqrt{d_xd_y}M\mspace{-3mu}\big(1+2^{1/4}\sqrt{\sigma}\big)  \right) \|y\|^2 }\mspace{-3mu}\big(1\mspace{-3mu}+\mspace{-3mu}\|y\|^2\big)^{|\beta_j|} d \nu(y)\\
    &\leq e^{4\sigma^2+ 8M\sqrt{\sigma d_xd_y}\left(2^{1/4}+\sqrt{2\sigma}+\frac{\|x\|^2}{\sqrt{\sigma}} \right)}    \left(\int e^{2\left(4\|x\|^2+8\sqrt{d_xd_y}M\big(1+2^{1/4}\sqrt{\sigma}\big)  \right) \|y\|^2 }  d \nu(y)\right)^{\frac 12}\\
    &\qquad\qquad\qquad\qquad\qquad\qquad\qquad\qquad\qquad\qquad\qquad\times\left(\int_{\|y\|\geq\tau} \big(1+\|y\|^2\big)^{2|\beta_j|} d \nu(y)\right)^{\frac 12}\\
    &\leq 2^{|\beta_j|-\frac{1}{2}} e^{4\sigma^2+ 8M\sqrt{\sigma d_xd_y}\left(2^{1/4}+\sqrt{2\sigma}+\frac{\|x\|^2}{\sqrt{\sigma}} \right) + 16\sigma^2\left(2M\sqrt{d_xd_y}\big(1+2^{1/4}\sqrt{\sigma}\big)+\|x\|^2\right)^2 }\\&\qquad\qquad\qquad\qquad\qquad\qquad\qquad\qquad\qquad\qquad\qquad\times\left( \int_{\|y\|\geq\tau}\big(1+\|y\|^{4|\beta_j|}\big) d \nu(y)\right)^{\frac 12}.
\end{align*}
To bounds the remaining integral on the RHS, consider
\begin{align*}
&\left| \int_{\|y\|\geq\tau}\|y\|^{4|\beta_j|} d \nu(y)\right|\leq e^{-\frac{\tau^4}{4\sigma^2}} \int e^{\frac{\|y\|^4}{4\sigma^2}} \|y\|^{4|\beta_j|} d \nu(y)\leq  e^{-\frac{\tau^4}{4\sigma^2}} \sqrt{2} (2\sigma^2)^{|\beta_j|}\sqrt{(2|\beta_j|)!}\\
&\left| \int_{\|y\|\geq\tau} d \nu(y)\right|\leq 
 e^{-\frac{\tau^4}{4\sigma^2}} \int e^{\frac{\|y\|^4}{4\sigma^2}}  d \nu(y)    \leq \sqrt{2}e^{-\frac{\tau^4}{4\sigma^2}},
\end{align*}
which follow by the Cauchy–Schwarz inequality and since $\nu$ is 4-sub-Weibull. Combining the pieces, we arrive at
\begin{align*}
    &\left|\frac{\int e^{\psi_0(y)-c_{\mathbf{A}}(x,y)}  \big(1+\|y\|^2\big)^{|\beta_j|} d \nu(y)}{\int e^{\psi_0(y)-c_{\mathbf{A}}(x,y)}  d \nu(y)}\right|\\
    &\leq (1\mspace{-2mu}+\mspace{-2mu}\tau^2)^{|\beta_j|} \mspace{-2mu}+\mspace{-2mu}  2^{|\beta_j|-\frac 14}   \left(1\mspace{-2mu}+\mspace{-2mu}\big(2\sigma^2\big)^{|\beta_j|/2}\big((2|\beta_j|)!\big)^{1/4}\right) \\
    &\qquad\quad\times e^{8\sigma^2+8M\sqrt{\sigma d_xd_y}\left(2^{1/4}+ 2\sqrt{2\sigma}+2^{5/4}\|x\|  +\frac{\|x\|^2}{\sqrt{\sigma}}\right) +16\sigma^2\left(2M\sqrt{d_xd_y}\big(1+2^{1/4}\sqrt{\sigma}\big)+\|x\|^2\right)^2 -\frac{\tau^4}{8\sigma^2} }  
\end{align*}
Now choose $\tau>0$ such that
\begin{align*}
    \frac{\tau^4}{8\sigma^2}&\geq8\sigma^2+8M\sqrt{\sigma d_xd_y}\left(2^{1/4}+2\sqrt{2\sigma}+2^{5/4}\|x\| +\frac{\|x\|^2}{\sqrt{\sigma}} \right)\\
    &\qquad\qquad\qquad\qquad\qquad\qquad\qquad\qquad+16\sigma^2\left(2M\sqrt{d_xd_y}\big(1+2^{1/4}\sqrt{\sigma}\big)+\|x\|^2\right)^2 ,
\end{align*}
namely, $\tau^4 = C \big(1+M\sqrt{d_xd_y}\big)^2\Big(1+\sigma^5 +\sigma^4\|x\|^4 \Big)$ for an appropriate universal constant $C>0$. Plugging back, we obtain 
\begin{align*}
    &\left|\frac{\int e^{\psi_0(y)-c_{\mathbf{A}}(x,y)}  (\|y\|^2\mspace{-3mu}+\mspace{-3mu}1)^{|\beta_j|} d \nu(y)}{\int e^{\psi_0(y)-c_{\mathbf{A}}(x,y)}  d \nu(y)}\right|\mspace{-2mu}\\
    &\qquad\qquad\qquad\qquad\qquad\leq \mspace{-2mu}C_{\beta_j}\Big(1\mspace{-3mu}+\mspace{-3mu}\sigma^{|\beta_j|} \mspace{-3mu}+\mspace{-3mu}\big (1\mspace{-3mu}+\mspace{-3mu}M\sqrt{d_xd_y}\big)^{|\beta_j|}\Big( 1\mspace{-3mu}+\mspace{-3mu}\sigma^5\mspace{-3mu}+\mspace{-3mu}\sigma^4\|x\|^4\Big)^{\frac{|\beta_j|}{2}}\Big),
\end{align*}
and thus
\begin{align*}
|D^\alpha \varphi(x)|&\leq C_{\alpha} \big(1\mspace{-3mu}+\mspace{-3mu}M\sqrt{d_y}\mspace{-3mu}+\mspace{-3mu}\|x\|\big)^{|\alpha|}\left(1\mspace{-3mu} +\mspace{-3mu}\sigma^{|\alpha|} \mspace{-3mu}+ \mspace{-3mu}\big(1\mspace{-3mu}+\mspace{-3mu}M\sqrt{d_xd_y}\big)^{|\alpha|}\big(1\mspace{-3mu}+\mspace{-3mu}\sigma^5 \mspace{-3mu}+\mspace{-3mu}\sigma^4\|x\|^4\big)^{\frac{|\alpha|}{2}}\right),
\end{align*}
as claimed.\qed

\section{Proofs of Lemmas for Theorem \ref{thm:GW_sample_complex}}\label{appen:proof:GW_sample_complex_lemmas}

\subsection{Proof of \cref{lemma:smoothness_of_GW_potential}}\label{appen:proof:lemma:smoothness_of_GW_potential}
With some abuse of notation, let $\cX=B_{d_x}(0,R)$ and $\cY=B_{d_y}(0,R)$ be the ambient spaces. 
Recall from \cref{sec:OT_EOT} that, for any $\mathbf{A}\in\cD_{R^2}$, we have $\Phi_\mathbf{A}\subset C_b(\cX)\times C_b(\cY)$ and we may further restrict to pairs of potentials that can be written as $(\varphi^{c\bar{c}},\varphi^c)$, for some $\varphi\in C_b(\cX)$. Since $\varphi^{c\bar{c}c}=\varphi^{c}$, the potentials are $c$- and $\bar c$-transforms of each other, i.e., we may only consider pairs $(\varphi,\psi)$ with 
\[
\varphi(x) = \inf_{y\in \cY} c_{\mathbf{A}}(x,y) - \psi(y)\qquad \mbox{and} \qquad\psi(y) = \inf_{x\in \cX} c_{\mathbf{A}}(x,y) - \varphi(x).
\]
Observing that $c_{\mathbf{A}}$ is concave in both arguments, we see that $\varphi$ and $\psi$ are both concave. Indeed, one readily verifies that the epigraphs of $-\varphi$ and $-\psi$ are convex sets, since for any $\alpha\in[0,1]$ and $x_1,x_2\in\cX$, we have
\[
\varphi\big(\alpha x_1+(1-\alpha)x_2\big)\geq \inf_{y\in\cY}\alpha c_\mathbf{A}(x_1,y)+(1-\alpha )c_\mathbf{A}(x_1,y)-\psi(y)\geq \alpha\varphi(x_1)+(1-\alpha)\varphi(x_2)
\]
and similarly for the other dual potential. 

\medskip
To bound the sup-norm of the augmented potentials, observe that the functional value is invariant to translations (i.e., $(\varphi-a,\psi+a)$ for some constant $a$). Since 
\[
\mathsf{OT}_{c_{\mathbf{A}}}(\mu,\nu) \geq  -\|c_{\mathbf{A}}\|_\infty \geq -4\big(1+4\sqrt{d_x d_y}\big)R^4,
\]
we further restrict to a class of functions with $\int \varphi d \mu+\int\psi d \nu\geq -2\Big(1+2\big(1+4\sqrt{d_x d_y}\big)R^4\Big)$. For such functions there must exist a point $(x_0,y_0)$, for which
\begin{align*}
    \varphi(x_0)+\psi(y_0) \geq -2\Big(1+2\big(1+4\sqrt{d_x d_y}\big)R^4\Big),
\end{align*}
and by shifting the potentials to coincide on $(x_0,y_0)$, i.e., $\varphi(x_0)=\psi(y_0)$, we obtain
\[
    \varphi(x_0) \geq -1-2\big(1+4\sqrt{d_x d_y}\big)R^4\qquad\mbox{and}\qquad \psi(y_0) \geq -1-2\big(1+4\sqrt{d_x d_y}\big)R^4.
\]
By the constraint, we then have
\begin{align*}
    \varphi(x)&\leq c_{\mathbf{A}}(x,y_0) - \psi(y_0)\leq 1+6\big(1+4\sqrt{d_x d_y}\big)R^4\\
    \psi(y)&\leq c_{\mathbf{A}}(x_0,y) - \varphi(x_0)\leq 1+6\big(1+4\sqrt{d_x d_y}\big)R^4.
\end{align*}
From the above, we also deduce 
\begin{align*}
    -\varphi(x)&\leq \|c_{\mathbf{A}}\|_\infty + \|\psi\|_\infty  \leq 1+10\big(1+4\sqrt{d_x d_y}\big)R^4\\
    -\psi(y)&\leq \|c_{\mathbf{A}}\|_\infty + \|\varphi\|_\infty  \leq 1+10\big(1+4\sqrt{d_x d_y}\big)R^4,
\end{align*}
which concludes the boundedness. 

\medskip
For Lipschitzness of optimal potentials note that for any $x\in\RR^{d_x}$ we can find a sequence $\{y_k\}_{k\in\NN}\subset\RR^{d_y}$, such that $\varphi(x)\leq c_{\mathbf{A}}(x,y_k)-\psi(y_k)\leq \varphi(x)+1/k$. So for any $x'\neq x$,
\begin{align*}
    \varphi(x') - \varphi(x) &\leq c_{\mathbf{A}}(x',y_k)-\psi(y_k) - \big(c_{\mathbf{A}}(x,y_k)-\psi(y_k)\big) + \frac{1}{k}\\
    & =\frac{1}{k} + 4\|y_k\|^2\big(\|x\|^2-\|x'\|^2\big) + 32 (x-x')^\intercal  \mathbf{A}y\\
    &\leq \frac{1}{k} + 8\big(1+2\sqrt{d_x d_y}\big)R^3 \|x-x'\|.
\end{align*}
Now take $k\to\infty$ and interchange $x,x'$ to conclude that the $\varphi$ is Lipschitz. Applying the same argument for $\psi$ concludes the proof of the lemma.\qed

 {
\subsection{Proof of \cref{lemma:covering}}\label{appen:proof:lemma:covering}
We aim to prove the covering bound
\[
N\left(\xi,\cup_{\bA\in\cD_{R^2}} S(\cF_\bA^{\mspace{1mu}c}),\|\cdot\|_\infty\right) \leq N\left(\frac{\xi}{64R^2},\cD_{R^2},\|\cdot\|_\op\right) N\left(\frac{\xi}{2},\widetilde\cF_{d_x},\|\cdot\|_\infty\right).
\]
First, note that by \cref{lemma:smoothness_of_GW_potential}, we have
\begin{equation}
N\left(\xi,\cup_{\bA\in\cD_{R^2}} S(\cF_\bA^{\mspace{1mu}c}),\|\cdot\|_\infty\right) \leq N\left(\xi,\cup_{\bA\in\cD_{R^2}} S(\cF_R^{\mspace{1mu}c}),\|\cdot\|_\infty\right). \label{eq:covering}
\end{equation}
Set $\xi_1=\frac{\xi}{64R^2}$ and $\xi_2=\frac{\xi}{2}$, and take a $\xi_1$-net $\{\bA_i\}_{i=1}^{N_1}$ of $\cD_{R^2}$ and a $\xi_2$-net $\{\varphi_i\}_{i=1}^{N_2}$ of $\widetilde\cF_{d_x}$. For $i=1,\ldots,N_1$ and $j=1,\ldots, N_2$, define the functions $g_{i,j}:\RR^{d_y}\to\RR$ by
\[
g_{i,j}(y)=S\Big[\inf_x \big(c_{\bA_j}(x,y)-(S^{-1}\varphi_i)(x)\big)\Big],
\]
where $(S\varphi)(z)\coloneqq \varphi(Rz)/(1+C_{d_x,d_y} R^4)$ is the rescaling operator defined after Eq. \eqref{eq:GW_sample_complex_S2_emp}. We will show that $\{g_{i,j}\}_{i,j=(1,1)}^{(N_1,N_2)}$ forms a $\xi$-net of $\cup_{\bA\in\cD_{R^2}} S(\cF_R^{\mspace{1mu}c})$, which together with the covering bound from \eqref{eq:covering} yields the result. Indeed, for any $\varphi\in\cF_R$, we have
\begin{align*}
    \left\|S\Big[\inf_x \big(c_\bA(x,\cdot)-\varphi(x)\big)\Big] -g_{i,j} \right\|_\infty    &\leq \sup_{x,y} \frac{| 32x^\intercal (\bA  -  \bA_i) y |}{1+C_{d_x,d_y} R^4} + \|\varphi-\varphi_j\|_\infty\\
    &\leq \frac{ 32 R^2}{1+C_{d_x,d_y} R^4}\xi_1 + \xi_2\\
    &\leq \xi,
\end{align*}
which concludes the proof. \qed
}

 {
\subsection{Proof of Lemma \ref{lemma:2sample_LB}}\label{appen:proof:lemma:2sample_LB}

Using \cref{lemma:equivalence_gw_w} along with the centering step from the proof of the one-sample lower bound, we have
\begin{align*}
\EE[\GW(\hat{\mu}_n,\hat\mu'_n)^2]& = \EE[\GW(\tilde{\mu}_n,\tilde\mu'_n)^2]\\
&\gtrsim \EE\left[\lambda_{\mathrm{min}}(\bsigma_{\tilde{\mu}_n})\inf_{\mathbf{U}\in O(d)} \mathsf{W}_2(\tilde{\mu}_n,\mathbf{U}_\sharp\tilde\mu'_n)^2\right]\\
&\gtrsim \EE\left[\lambda_{\mathrm{min}}(\bsigma_{\hat{\mu}_n})\inf_{\mathbf{U}\in O(d)} \mathsf{W}_2(\hat{\mu}_n,\mathbf{U}_\sharp\hat\mu'_n)^2\right] -2\EE\left[\|\bar{x}_n\|\right] \\
&\geq\EE\left[\lambda_{\mathrm{min}}(\bsigma_{\hat{\mu}_n})\EE\left[\inf_{\mathbf{U}\in O(d)} \mathsf{W}_1(\hat{\mu}_n,\mathbf{U}_\sharp\hat\mu'_n)^2\middle|X_1,\ldots,X_n\right]\right]-2\sqrt{\frac{M_2(\mu)}{n}}.
\end{align*}
To justify the third step above, observe that
\begin{align*}
    \big|\lambda_{\mathrm{min}}(\bsigma_{\hat{\mu}_n})-\lambda_{\mathrm{min}}(\bsigma_{\tilde{\mu}_n})\big|\leq \sup_{\|v\|=1} \hat{\mu}_n\big||v\cdot x|^2-|v\cdot (x-\bar{x}'_n)|^2\big|\leq 3\|\bar{x}'_n\|,
\end{align*}
and
\begin{align*}  
\bigg|\inf_{\mathbf{U}\in O(d)} &\mathsf{W}_2(\tilde{\mu}_n,\mathbf{U}_\sharp\tilde\mu'_n)^2 - \inf_{\mathbf{U}\in O(d)} \mathsf{W}_2(\hat{\mu}_n,\mathbf{U}_\sharp\hat\mu'_n)^2\bigg|\\
&\leq \sup_{\mathbf{U}\in O(d)}\left|\mathsf{W}_2(\tilde{\mu}_n,\mathbf{U}_\sharp\tilde\mu'_n)^2 - \mathsf{W}_2(\hat{\mu}_n,\mathbf{U}_\sharp\hat\mu'_n)^2\right|\\
&= \sup_{\mathbf{U}\in O(d)}\big(\sW_2(\tilde{\mu}_n,\mathbf{U}_\sharp\tilde\mu'_n) + \sW_2(\hat{\mu}_n,\mathbf{U}_\sharp\hat\mu'_n)\big)\big| \mathsf{W}_2(\tilde{\mu}_n,\mathbf{U}_\sharp\tilde\mu'_n) - \mathsf{W}_2(\hat{\mu}_n,\mathbf{U}_\sharp\hat\mu'_n)\big|\\
    &\leq \sup_{\mathbf{U}\in O(d)} \big(\sW_2(\tilde{\mu}_n,\mathbf{U}_\sharp\tilde\mu'_n) + \sW_2(\hat{\mu}_n,\mathbf{U}_\sharp\hat\mu'_n)\big)\big(\sW_2(\hat{\mu}_n,\tilde{\mu}_n)+ \sW_2(\mathbf{U}_\sharp\hat\mu'_n,\mathbf{U}_\sharp\tilde\mu'_n)\big)\\
    &\leq 6\big(\sW_2(\hat{\mu}_n,\tilde{\mu}_n)+ \sW_2(\hat\mu'_n,\tilde\mu'_n)\big)\\
    &\leq 6(\|\bar{x}_n\|+\|\bar{x}'_n\|).
\end{align*}
Together, these imply the desired bound as $\inf_{\mathbf{U}\in O(d)} \mathsf{W}_2(\hat{\mu}_n,\mathbf{U}_\sharp\hat\mu'_n)^2\leq4$, $\lambda_{\mathrm{min}}(\bsigma_{\tilde{\mu}_n})\leq 1$, and
\begin{align*}
    \bigg|\lambda_{\mathrm{min}}(\bsigma_{\tilde{\mu}_n})\inf_{\mathbf{U}\in O(d)}& \mathsf{W}_2(\tilde{\mu}_n,\mathbf{U}_\sharp\tilde\mu'_n)^2 - \lambda_{\mathrm{min}}(\bsigma_{\hat{\mu}_n})\inf_{\mathbf{U}\in O(d)} \mathsf{W}_2(\hat{\mu}_n,\mathbf{U}_\sharp\hat\mu'_n)^2 \bigg|\\
    &\leq \big|\lambda_{\mathrm{min}}(\bsigma_{\tilde{\mu}_n})-\lambda_{\mathrm{min}}(\bsigma_{\hat{\mu}_n}) \big|\inf_{\mathbf{U}\in O(d)} \mathsf{W}_2(\hat{\mu}_n,\mathbf{U}_\sharp\hat\mu'_n)^2 \\
    &\qquad+ \lambda_{\mathrm{min}}(\bsigma_{\tilde{\mu}_n})\bigg|\inf_{\mathbf{U}\in O(d)} \mathsf{W}_2(\tilde{\mu}_n,\mathbf{U}_\sharp\tilde\mu'_n)^2 - \inf_{\mathbf{U}\in O(d)} \mathsf{W}_2(\hat{\mu}_n,\mathbf{U}_\sharp\hat\mu'_n)^2 \bigg|,
\end{align*}
which validates \eqref{eq:GW_LB_2sample}.

}

 {
\subsection{Proof of Lemma \ref{lemma:technical_lemma}}\label{appen:proof:lemma:technical_lemma}

Consider the following decomposition
    \begin{align*}
    \EE &\left[\inf_{\mathbf{U}\in O(d)}\sW_1(\hat{\mu}_n,\mathbf{U}_\sharp\nu)\right]\\
    &\quad= \EE\left[\inf_{\mathbf{U}\in O(d)}\sW_1(\hat{\mu}_n,\mathbf{U}_\sharp\nu)\mspace{-3mu} - \mspace{-3mu}\inf_{\mathbf{U}\in O(d)}\EE\big[\sW_1(\hat{\mu}_n,\mathbf{U}_\sharp\nu)\big]\right] \mspace{-3mu}+\mspace{-3mu} \inf_{\mathbf{U}\in O(d)}\EE\big[\sW_1(\hat{\mu}_n,\mathbf{U}_\sharp\nu)\big]\\
    &\quad \geq \EE\left[\inf_{\mathbf{U}\in O(d)}\Big(\sW_1(\hat{\mu}_n,\mathbf{U}_\sharp\nu) - \EE\big[\sW_1(\hat{\mu}_n,\mathbf{U}_\sharp\nu)\big]\Big)\right] + \inf_{\mathbf{U}\in O(d)}\EE\big[\sW_1(\hat{\mu}_n,\mathbf{U}_\sharp\nu)\big]\\
    &\quad= -\EE\left[\sup_{\mathbf{U}\in O(d)}\Big( \EE\big[\sW_1(\hat{\mu}_n,\mathbf{U}_\sharp\nu)\big] - \sW_1(\hat{\mu}_n,\mathbf{U}_\sharp\nu) \Big) \right] + \inf_{\mathbf{U}\in O(d)}\EE\big[\sW_1(\hat{\mu}_n,\mathbf{U}_\sharp\nu)\big].
    \end{align*}
    Denoting $R_\mathbf{U}\coloneqq \EE\big[\sW_1(\hat{\mu}_n,\mathbf{U}_\sharp\nu)\big] - \sW_1(\hat{\mu}_n,\mathbf{U}_\sharp\nu)$, we proceed to upper bound $\EE[\sup_\mathbf{U} R_\mathbf{U}]$. Note that $|R_\mathbf{U}-R_\mathbf{V}|\leq 2\sW_1(\mathbf{U}_\sharp\nu,\mathbf{V}_\sharp\nu)\leq 2\|\mathbf{U}-\mathbf{V}\|_\op$, and thus the process $\{R_\mathbf{U}\}_{\mathbf{U}\in O(d)}$ is Lipschitz in $\mathbf{U}$. We further claim that $\{R_\mathbf{U}\}_{\mathbf{U}\in O(d)}$ is a sub-Gaussian process. To see, for fixed $\mathbf{U}\in O(d)$, define the function $w_{\mathbf{U}}\mspace{-3mu}:\mspace{-3mu}(x_1,\cdots,x_n)\mspace{-3mu}\in \mspace{-3mu}B_d(0,1)^n\mspace{-3mu}\mapsto \mspace{-3mu} \sW_1\big(n^{-1}\sum_{i=1}^n\delta_{x_i},\mspace{-3mu}\mathbf{U}\nu\big)$ and note that it has bounded differences:
            \begin{align*}
            &\sup_{x_i,x'_i}\big|w_\mathbf{U}(x_1,\mspace{-1mu}\cdots\mspace{-1mu},x_{i-1},x_i,x_{i+1},\mspace{-1mu}\cdots\mspace{-1mu},x_n)-w_\mathbf{U}(x_1,\mspace{-1mu}\cdots\mspace{-1mu},x_{i-1},x'_i,x_{i+1},\mspace{-1mu}\cdots\mspace{-1mu},x_n)\big|\\
            &\qquad\leq \frac{\|x_i-x'_i\|}{n}\leq \frac{2}{n},
            \end{align*}
   For $X_1,\ldots,X_n$ i.i.d. from $\mu$ (for which $\supp(\mu)\subset B_d(0,1)$ by assumption), McDiarmid's inequality now yields
            \begin{align*}
                \PP\big(\big|w_\mathbf{U}(X_1,\ldots,X_n)-\EE[w_\mathbf{U}(X_1,\ldots,X_n)]\big|\geq t\big)\leq 2e^{-nt^2/2}.
            \end{align*}
    Observing that $R_\mathbf{U}=\EE[w_\mathbf{U}(X_1,\ldots,X_n)]-w_\mathbf{U}(X_1,\ldots,X_n)$, by equivalence between definitions of sub-Gaussianity, we further obtain $\EE[e^{sR_\mathbf{U}}]\leq e^{s^2\sigma^2/2}$, for all $s$, where $\sigma = \frac{3\sqrt{2}}{\sqrt{n}}$. Thus, $\{R_\mathbf{U}\}_{\mathbf{U}\in O(d)}$ is indeed sub-Gaussian.

    Combining Lipschitzness and sub-Gaussianity, we deploy a standard $\epsilon$-net argument. Let $\{\mathbf{U}_i\}_{i=1}^N$ is an $\epsilon$-net of $O(d)$ w.r.t. the operator norm. We have
    \begin{align*}
    \EE\big[\sup\nolimits_\mathbf{U} R_\mathbf{U}\big] &\leq \inf_{\epsilon>0}2\epsilon + \EE\left[\max_{i=1,\ldots,N} R_{\mathbf{U}_i} \right]\\
    &\leq \inf_{\epsilon>0}2\epsilon + \frac{3\sqrt{2}}{\sqrt{n}}\sqrt{\log\big(N(O(d),\epsilon,\|\cdot\|_\op)\big)}\\
    &\lesssim_{d} \sqrt{\log(n)}/\sqrt{n},
    \end{align*}
    where the last step uses the fact that $\log\big(N(O(d),\epsilon,\|\cdot\|_\op)\big)\leq (c\sqrt{d}\epsilon^{-1})^{d^2}$, for a universal constant $c$; cf. Lemma 4 from \cite{niles2022estimation}. This concludes the proof. \qed
}

\section{Proof of Proposition \ref{prop:1d_opt}}\label{sec:prop:1d_opt_proof}
As mentioned before, for any $\sigma^\star\in\cS^\star$, 
we have 
\[
a^\star=\frac{1}{2}\int xy d \pi^\star(x,y)\in\cA^\star,
\]
where $\pi^\star$ is the coupling induced by $\sigma^\star$, and consequently $(a^\star,\pi^\star)$ jointly minimize \eqref{eq:GW_1d_S2}. For the first direction, suppose that $\cA^\star\subset\{0.5 W_-,0.5 W_+\}$ but that there exists $\tilde\sigma\notin\{\mathrm{id},\overline{\mathrm{id}}\}$ that optimizes~\eqref{eq:GW_1d}. Denoting the corresponding coupling by $\tilde{\pi}$, the above implies that $\tilde{a}=\frac{1}{2}\int xy d \tilde\pi(x,y)\in\cA^*$. However, by the rearrangement inequality 
$0.5 W_-<\tilde{a}<0.5 W_+$, which is a contradiction. Since the minimum in \eqref{eq:GW_1d} is achieved, we conclude that $\cS^\star\subset\{\mathrm{id},\overline{\mathrm{id}}\}$.

\medskip
For the other direction, suppose that $\cS^\star\mspace{-3mu}\subset\mspace{-3mu}\{\mathrm{id},\overline{\mathrm{id}}\}$ but that there exists $\tilde a\in(0.5 W_-,0.5 W_+)$ with $\tilde a\in\cA^\star$. We first argue the $f+g$ is differentiable at $\tilde a$. This follows because $g$ is piecewise linear and concave, and so for any non-differentiability point $a_0$, the left and right derivatives satisfy $g'_-(a_0)>g'_+(a_0)$. Since $f$ is smooth, we further obtain $(f+g)'_-(a_0)>(f+g)'_+(a_0)$, so $a_0$ cannot be a local minimum. We conclude that $f+g$ is differentiable at $\tilde{a}$ with $(f+g)'(\tilde a)=0$.

Having that, let $\Pi_{\tilde a}\subset\Pi(\mu,\nu)$ be the argmin set for $g(\tilde a)$ and fix $\tilde\pi\in\Pi_{\tilde a}$. Since $(f+g)'(\tilde a)=0$, computing the derivative, we obtain 
\[
64\tilde a-32\int xy d \tilde\pi(x,y)=0.
\]
Thus, $\int xy d \tilde\pi(x,y)=2\tilde a$, for every $\pi\in\Pi_{\tilde a}$. Now, Since $(\tilde a,\tilde\pi)$ minimize \eqref{eq:GW_1d_S2}, consider
\begin{align*}
    \EGW^2(\mu,\nu)&= 32\tilde{a}^2+  \inf_{\pi\in\Pi_{\tilde a}}\int \big(-4x^2y^2-32\tilde{a}xy \big) d \pi(x,y)\\
    &=   \inf_{\pi\in\Pi_{\tilde a}}32\tilde{a}^2+\int \big(-4x^2y^2-32\tilde{a}xy \big) d \pi(x,y)\\
    & = \inf_{\pi\in\Pi_{\tilde a}}8 \left(\int xy d \pi(x,y)\right)^2+\int -4x^2y^2  d \pi(x,y) - 16 \left(\int xy d \pi(x,y)\right)^2 \\
    & = \inf_{\pi\in\Pi_{\tilde a}}\int -4x^2y^2  d \pi(x,y) - 8 \left(\int xy d \pi(x,y)\right)^2.
\end{align*}
Recalling that 
\[\EGW^2(\mu,\nu) = \inf_{\pi\in\Pi(\mu,\nu)}\int -4x^2y^2  d \pi(x,y) - 8 \left(\int xy d \pi(x,y)\right)^2    
    \]
by definition (see below \eqref{eq:egw_decomposition}), we conclude that all elements of $\Pi_{\tilde a}$ are minimizers for $\EGW^2$, and hence also minimizers of $\GW(\mu,\nu)$.

To get a contradiction, recall that $\int xy d \tilde\pi(x,y)=2\tilde a$, for all $\tilde{\pi}\in\Pi_{\tilde a}$. Since $\tilde a\in(0.5 W_-,0.5 W_+)$, one readily verifies $W_-<\int xy d \tilde\pi(x,y)<W_+$. However, the couplings induced by $\mathrm{id}$ and $\overline{\mathrm{id}}$ achieve exactly $W_+$ and $W_-$ for the said integral, and thus they are not contained in $\Pi_{\tilde a}$. Since by assumption the argmin is $\cS^\star=\{\mathrm{id},\overline{\mathrm{id}}\}$, we again have a contradiction and $\tilde a$ cannot be optimal for~\eqref{eq:GW_1d_S2}. The infimum must therefore be achieved on the boundary, i.e., $\cA^\star\subset\{0.5W_-,0.5W_+\}$, which concludes the proof. \qed

\section{Generalized Duality}\label{appen:generalized_dual}

 {
We derive here the generalized dual representation for the GW distance of order $(p,q)=(2,2k)$, where $k\in\NN$. The approach naturally extends to any even $p$ value, but the cost of tedious technical details, which we prefer to avoid for presentation. Like in \cref{appen:proof:lemma:equivalence_gw_w}, we omit dummy variables from our integral notation, writing $\int f(x,y,x',y')d\pi\otimes\pi$ instead of $\int f(x,y,x',y')d\pi\otimes\pi(x,y,x',y')$. Let $(\mu,\nu)\in\cP_{4k}(\RR^{d_x})\times \cP_{4k}(\RR^{d_y})$, and expand the distortion cost to obtain
\[
        \GW_{2,2k}(\mu,\nu)^2\mspace{-3mu} = \mspace{-3mu}\inf_{\pi\in\Pi(\mu,\nu)} \int \mspace{-3mu}\left( \big| (\|x\|^2\mspace{-3mu}-\mspace{-3mu}2x\cdot x' \mspace{-3mu}+\mspace{-3mu}\|x'\|^2)^k \mspace{-3mu}-\mspace{-3mu} (\|y\|^2\mspace{-3mu}-\mspace{-3mu}2y\cdot y'\mspace{-3mu} +\mspace{-3mu}\|y'\|^2)^k \big|^2\right) d\pi \otimes\pi.
\]
Collecting terms that depend only on the marginals into $S^1(\mu,\nu)$ as before and omitting them for now, we seek a dual for the optimization problem
\begin{align*}
    \inf_{\pi\in\Pi(\mu,\nu)} \int \left( - 2(\|x\|^2-2x\cdot x' +\|x'\|^2)^k (\|y\|^2-2y\cdot y' +\|y'\|^2)^k \right) d\pi \otimes\pi.
\end{align*}
The integrand is a homogeneous polynomial that is symmetric in $(x,y)$ and $(x',y')$. Consequently, there exist polynomials $f_1,\ldots,f_m$ of degree at most $4k$, and a symmetric matrix $\mathbf{C}\in\RR^{m\times m}$ (whose entries are denoted by $C_{ij}$, for $i,j=1,\ldots,m$), such~that 
\begin{align*}
    &\inf_{\pi\in\Pi(\mu,\nu)} \int \left( -2 (\|x\|^2-2x\cdot x' +\|x'\|^2)^k (\|y\|^2-2y\cdot y' +\|y'\|^2)^k \right) d\pi \otimes\pi\\
    &= \inf_{\pi\in\Pi(\mu,\nu)} \int \Big( -2\|x\|^{2k}\|y\|^{2k} - 2\|x'\|^{2k}\|y'\|^{2k} + \sum_{1\leq i,j\leq m} C_{ij}f_i(x,y)f_j(x',y')\Big) d\pi \otimes\pi\\
    & = \inf_{\pi\in\Pi(\mu,\nu)} \int -4\|x\|^{2k}\|y\|^{2k} d\pi   + \sum_{1\leq i,j\leq m} C_{ij}\int f_i(x',y') d\pi \int f_j(x,y) d\pi .
\end{align*}
Note that $m$ is bounded by the number of monomials of degree at most $4k$ that can constructed from entries of $x,y$, i.e., $m = O((d_x+d_y)^{4k})$.\footnote{In practice, $m$ is often be much smaller, as seen from the example below.} By diagonalizing $C$, we rewrite
\begin{equation}
\sum_{1\leq i,j\leq m}\mspace{-3mu} C_{ij}\int f_i(x,y) d\pi \int f_j(x',y') d\pi= \sum_{i=1}^{\ell} \left(\int g_i(x,y) d\pi \right)^2  \mspace{-3mu}- \mspace{-3mu} \sum_{i=\ell+1}^{m} \left(\int g_i(x,y) d\pi \right)^2,\label{eq:gen_dual_SoS}
\end{equation}
where $g_i$ are linear combinations of $f_i$, and $\ell$ is the number of positive eigenvalues of $\mathbf{C}$. Notice that the sum of squares on the RHS above can have positive or negative coefficient, which differs form the $(2,2)$ case where only a negative coefficient is present. 

\medskip
Armed with \eqref{eq:gen_dual_SoS}, we proceed with the same linearization step from the proof of \cref{thm:egw_duality} by introducing the new auxiliary optimization variables $a\in\RR^{\ell}$ and $b\in\RR^{m-\ell}$, as follows
\begin{align*}
    &\inf_{\pi\in\Pi(\mu,\nu)} \int \left( - 2(\|x\|^2-2x\cdot x' +\|x'\|^2)^k (\|y\|^2-2y\cdot y' +\|y'\|^2)^k \right) d\pi \otimes\pi \\
    & = \inf_{\pi\in\Pi(\mu,\nu)} \int -4\|x\|^{2k}\|y\|^{2k} d\pi   + \sum_{i=1}^{\ell} \left(\int g_i(x,y) d\pi \right)^{\mspace{-3mu}2} \mspace{-5mu}-\mspace{-5mu} \sum_{i=\ell+1}^{m} \left(\int g_i(x,y) d\pi \right)^{\mspace{-3mu}2}\\
    & = \inf_{\pi\in\Pi(\mu,\nu)}\int -4\|x\|^{2k}\|y\|^{2k} d\pi   + \sup_{a\in\RR^\ell}\sum_{i=1}^{\ell} \left( 4a_i\int g_i(x,y) d\pi -4a_i^2\right) \\
    &\hspace{13.3em}\qquad\qquad+\inf_{b\in\RR^{m-\ell}} \sum_{i=\ell+1}^{m} \left(4b_{i-\ell}^2 -4b_{i-\ell}\int g_i(x,y) d\pi \right)\\
    & = 4\sup_{a\in\RR^\ell} \inf_{b\in\RR^{m-\ell}}-\|a\|^2\mspace{-3mu}+\|b\|^2\\
    &\qquad\qquad\qquad\quad +\inf_{\pi\in\Pi(\mu,\nu)}\int \left(-\|x\|^{2k}\|y\|^{2k} +\sum_{i=1}^{\ell}a_i g_i(x,y)- \sum_{i=\ell+1}^{m} b_{i-\ell} g_i(x,y)\right)d\pi,
\end{align*}
where the last step follows from Sion's minimax theorem. The RHS above is the desired dual representation from \eqref{eq:gen_dual}. Further observe that as $\int g_i d\pi$ are uniformly bounded for all $i$ and $\pi$, we may restrict optimization domains for $a$ and $b$ to compact sets. We identify the inner optimization over $\pi$ as an OT problem with cost $c_{a,b}:(x,y) \mapsto -\|x\|^{2k}\|y\|^{2k} + \sum_{i=1}^{\ell}a_i g_i(x,y) - \sum_{i=\ell+1}^{m}b_{i-\ell} g_i(x,y)$, which is smooth (indeed, a polynomial) but not necessarily concave in $x$ or $y$. Considering compactly supported distributions, one may invoke OT duality and establish Lipschitzness of the optimal potential, although convexity seems challenging to obtain in general. As explain in \cref{sec:summary}, by following the steps in the proof of \cref{thm:GW_sample_complex}, this leads to a two-sample empirical convergence rate of $O(n^{-1/(d_x\wedge d_y)})$. We leave further refinements of this rate as well as proofs of lower bounds for future work.

\medskip
To illustrate the above procedure, we consider the special case of $p=q=2$. This will also show how the duality formula from \eqref{eq:gen_dual} reduces back to that from \cref{cor:gw_duality}, after assuming that the populations are centered. 
As above, we start by expanding the $(2,2)$-cost and omitting terms that depend only on the marginals (cf. \eqref{eq:proof_of_EGW_decomposition}), to arrive at
\begin{align*}
    &\inf_{\pi\in\Pi(\mu,\nu)}  \int -4\|x\|^2\|y\|^2  d \pi  \\
    &\qquad\qquad+ 4\int \Big( \langle x,x' \rangle (\|y\|^2+\|y'\|^2) + (\|x\|^2+\|x'\|^2)\langle y,y' \rangle-2\langle x,x'\rangle\langle y,y'\rangle\Big) d \pi\otimes\pi \\
    & = \inf_{\pi\in\Pi(\mu,\nu)}  \int \left( -4\sum_{ 1\leq i\leq d_x, 1\leq j \leq d_y} x_i^2y_j^2 \right) d \pi\\
    &\qquad\qquad+ 4\int \left(\sum_{ 1\leq i\leq d_x, 1\leq j \leq d_y}  x_ix'_i(y_j^2+y'{}_j^2) + (x_i^2+x'{}_i^2)y_jy'_j-2x_ix'_iy_jy'_j \right)d \pi\otimes\pi\numberthis\label{eq:gen_duality_22_functional}
\end{align*}
To diagonalize the second term, consider the set of linearly independent monomials $\{x_i,y_j,x_iy_j^2,x_i^2y_j,x_iy_j\}_{1\leq i\leq d_x,1\leq j \leq d_y}$, of which there are $d_x+d_y+3d_xd_y$ in total (these are denoted by $f_i$ in the general derivation above). 
For concreteness and simplicity, we henceforth assume $d_x=d_y=1$. Define the vector 
$$v(\pi) = \left(\int x\dd\pi,\int y\dd\pi,\int x y^2\dd\pi,\int x^2y\dd\pi,\int xy\dd\pi\right)^\intercal,$$
and construct coefficient matrix 
$$\mathbf{C}=\begin{bmatrix}
0 & 0 & 1 & 0 & 0\\
0 & 0 & 0 & 1 & 0\\
1 & 0 & 0 & 0 & 0\\
0 & 1 & 0 & 0 & 0\\
0 & 0 & 0 & 0 & -2
\end{bmatrix}.$$
For instance, we set $C_{1,3}=1$ since the term $\int x d\pi\int xy^2 d\pi$, which is the product of $v_1(\pi)$ and $v_3(\pi)$, appears inside the functional from \eqref{eq:gen_duality_22_functional}. We may now express
\begin{align*}
    \inf_{\pi\in\Pi(\mu,\nu)} -4 \int x^2y^2 d \pi  + 4\int &\left(  xx'y^2 +x'xy'^2 + x^2yy' + x'^2y'y - 2xx'yy'\right) d \pi\otimes\pi\\
    &\qquad\qquad\qquad\quad = \inf_{\pi\in\Pi(\mu,\nu)} -4 \int x^2y^2 d \pi + 4v(\pi)^\intercal \mathbf{C}v(\pi).
\end{align*}
Diagonalizing $\mathbf{C}$, further yields
\begin{align*}
    v(\pi)^\intercal \mathbf{C}v(\pi) &= \left(\int \frac{\sqrt{2}}{2}x + \frac{\sqrt{2}}{2}xy^2 d\pi\right)^{\mspace{-3mu}2} + \left(\int \frac{\sqrt{2}}{2}y + \frac{\sqrt{2}}{2}x^2y d\pi\right)^{\mspace{-3mu}2}\\
    & -\left(\int -\frac{\sqrt{2}}{2}x + \frac{\sqrt{2}}{2}xy^2 d\pi\right)^{\mspace{-3mu}2} \mspace{-3mu}-\mspace{-3mu} \left(\int -\frac{\sqrt{2}}{2}y + \frac{\sqrt{2}}{2}x^2y d\pi\right)^{\mspace{-3mu}2} \mspace{-3mu}-\mspace{-3mu} \left(\int \sqrt{2}xy d\pi\right)^{\mspace{-3mu}2}\!.
\end{align*}
We proceed by introducing $a\in\RR^2$ and $b\in\RR^3$, as follows
\begin{align*}
    &\inf_{\pi\in\Pi(\mu,\nu)} -4 \int   x^2y^2 d \pi  +   4\int \left( xx'y^2 +x'xy'^2 + x^2yy' + x'^2y'y - 2xx'yy'\right) d \pi\otimes\pi\\
    &=\inf_{\pi\in\Pi(\mu,\nu)}  -4\int x^2y^2 d \pi  +  \left(\int \sqrt{2}x + \sqrt{2}xy^2 d\pi\right)^{\mspace{-3mu}2} + \left(\int \sqrt{2}y + \sqrt{2}x^2y d\pi\right)^{\mspace{-3mu}2}\\
    &\qquad\qquad\ \ -\left(\int -\sqrt{2}x + \sqrt{2}xy^2 d\pi\right)^{\mspace{-3mu}2} - \left(\int -\sqrt{2}y + \sqrt{2}x^2y d\pi\right)^{\mspace{-3mu}2} - \left(\int 2\sqrt{2}xy d\pi\right)^{\mspace{-3mu}2}\\
    &=\inf_{\pi\in\Pi(\mu,\nu)}  -4\int x^2y^2 d \pi \\
    &\quad+  \sup_{a\in\RR^2} 4a_1 \int \left( \sqrt{2}x + \sqrt{2}xy^2 \right)d\pi -4a_1^2 + 4a_2\int\left(  \sqrt{2}y + \sqrt{2}x^2y\right) d\pi - 4a_2^2\\
    &\quad+\inf_{b\in\RR^3} 4b_1^2-4b_1\int \left(-\sqrt{2}x + \sqrt{2}xy^2 \right) d\pi  +4b_2^2- 4b_2\int \left( -\sqrt{2}y + \sqrt{2}x^2y\right) d\pi  \\
    &\hspace{26em}+4b_3^2 - 4b_3\int 2\sqrt{2}xy d\pi\\
    &=4\sup_{a\in\RR^2} \inf_{b\in\RR^3} -\|a\|^2+\|b\|^2 + \inf_{\pi\in\Pi(\mu,\nu)} \int c_{a,b}(x,y) d\pi ,
\end{align*}
where the cost function is 
\begin{align*}
c_{a,b}(x,y) = -x^2y^2 + \sqrt{2}a_1x + \sqrt{2}a_1xy^2 + \sqrt{2}a_2y &+ \sqrt{2}a_2x^2y +\sqrt{2}b_1x -\sqrt{2}b_1xy^2\\
&+ \sqrt{2}b_2y - \sqrt{2}b_2x^2y - 2\sqrt{2}b_3  xy.
\end{align*}
Lastly, notice that if $\mu,\nu$ are centered, then
\begin{align*}
    v(\pi)^\intercal \mathbf{C}v(\pi) = - 2\left(\int xy d\pi\right)^{\mspace{-3mu}2},
\end{align*}
which immediately recovers the dual form from \cref{cor:gw_duality}, where the OT cost function is $c_{\mathbf{A}}(x,y) = -4x^2y^2 - 32 \bA_{1,1}xy$. The cost $c_{a,b}$ that arises from the general derivation is evidently more complex and comprises additional mixed terms (of order 3). This makes it harder to analyze, e.g., it is unclear whether $c_{a,b}$ is marginally convex/concave in each argument. Consequently, this approach may not lead to the same regularity profile for dual potentials as we have in \cref{lemma:smoothness_of_GW_potential}, which, in turn, may result in suboptimal empirical convergence rates.}

\end{document}